\documentclass[pdftex,a4paper,12pt]{scrartcl}

\usepackage[utf8]{inputenc}
\usepackage[T1]{fontenc}
\usepackage{lmodern} 
\usepackage{exscale}

\usepackage[pdftex]{graphicx}
\usepackage{amsmath,amssymb,amsthm}
\usepackage{mathtools}
\usepackage[unicode]{hyperref}
\usepackage[capitalize]{cleveref}
\usepackage{enumitem}

\usepackage{tikz}
\usetikzlibrary{cd,calc}

\tikzset{
  x=.7cm, y=.7cm,
  baseline=-.5ex,
  quadratic/.style={
    to path={
      (\tikztostart) .. controls
      ($#1!1/3!(\tikztostart)$) and ($#1!1/3!(\tikztotarget)$)
      .. (\tikztotarget)
    }
  },
  xlen/.style={
    x={(0pt,#1)}
  },
  ylen/.style={
    y={(#1,0pt)}
  },
}

\makeatletter
\newcommand\saveTikzBox[3][]{%
  \expandafter\newsavebox\csname tikz@box@#2\endcsname%
  \expandafter\savebox\csname tikz@box@#2\endcsname{%
    \tikz[#1]{#3}%
  }%
}

\newcommand\useTikzBox[1]{\expandafter\usebox\csname tikz@box@#1\endcsname}

\newcommand\defTikzBox[3][]{%
  \saveTikzBox[#1]{#2}{#3}
  \expandafter\newcommand\csname#2\endcsname{\useTikzBox{#2}}
}
\makeatother

\defTikzBox{diagRiiiLeftUParVVW}{%
  \useasboundingbox (-.75,-1) rectangle (.75,1);
  \node[inner sep=2] (LU) at (-.25,.4) {};
  \node[inner sep=3] (RM) at (.25,0) {};
  \node[inner sep=2] (LD) at (-.25,-.4) {};
  \draw[red,very thick] (-.6,-1) to[out=90,in=-90] (LD.west) to[out=90,in=-90] (-.5,0) to[out=90,in=-90] (LU.west) to[out=90,in=-90] (-.6,1);
  \draw[red,very thick] (0,-1) to[out=90,in=-90] (LD.east) to[out=90,in=180] (RM.south) to[out=0,in=-90] (.6,-1);
  \draw[red,very thick] (RM.north) to[out=180,in=-90] (LU.east) to[out=90,in=-90] (0,1);
  \draw[red,very thick] (RM.north) to[out=0,in=-90] (.6,1);
}

\defTikzBox{diagRiiiLeftUParVWW}{%
  \useasboundingbox (-.75,-1) rectangle (.75,1);
  \node[inner sep=2] (LU) at (-.25,.4) {};
  \node[inner sep=3] (RM) at (.25,0) {};
  \node[inner sep=3] (LD) at (-.25,-.4) {};
  \draw[red,very thick] (-.6,-1) to[out=90,in=180] (LD.south) to[out=0,in=90] (0,-1);
  \draw[red,very thick] (LD.north) to[out=180,in=-90] (-.5,0) to[out=90,in=-90] (LU.west) to[out=90,in=-90] (-.6,1);
  \draw[red,very thick] (LD.north) to[out=0,in=180] (RM.south) to[out=0,in=90] (.6,-1);
  \draw[red,very thick] (RM.north) to[out=180,in=-90] (LU.east) to[out=90,in=-90] (0,1);
  \draw[red,very thick] (RM.north) to[out=0,in=-90] (.6,1);
}

\defTikzBox{diagRiiiLeftUParWWW}{%
  \useasboundingbox (-.75,-1) rectangle (.75,1);
  \node[inner sep=3] (LU) at (-.25,.4) {};
  \node[inner sep=3] (RM) at (.25,0) {};
  \node[inner sep=3] (LD) at (-.25,-.4) {};
  \draw[red,very thick] (-.6,-1) to[out=90,in=180] (LD.south) to[out=0,in=90] (0,-1);
  \draw[red,very thick] (LU.north) to[out=180,in=-90] (-.6,1);
  \draw[red,very thick] (LU.north) to[out=0,in=-90] (0,1);
  \draw[red,very thick] (.6,-1) to[out=90,in=0] (RM.south) to[out=180,in=0] (LD.north) to[out=180,in=-90] (-.5,0) to[out=90,in=180] (LU.south) to[out=0,in=180] (RM.north) to[out=0,in=-90] (.6,1);
}

\defTikzBox{diagRiiiLeftUParWVW}{%
  \useasboundingbox (-.75,-1) rectangle (.75,1);
  \node[inner sep=3] (LU) at (-.25,.4) {};
  \node[inner sep=3] (RM) at (.25,0) {};
  \node[inner sep=2] (LD) at (-.25,-.4) {};
  \draw[red,very thick] (-.6,-1) to[out=90,in=-90] (LD.west) to[out=90,in=-90] (-.5,0) to[out=90,in=180] (LU.south) to[out=0,in=180] (RM.north) to[out=0,in=-90] (.6,1);
  \draw[red,very thick] (LU.north) to[out=180,in=-90] (-.6,1);
  \draw[red,very thick] (LU.north) to[out=0,in=-90] (0,1);
  \draw[red,very thick] (0,-1) to[out=90,in=-90] (LD.east) to[out=90,in=180] (RM.south) to[out=0,in=90] (.6,-1);
}

\defTikzBox{diagRiiiLeftUParWVV}{%
  \useasboundingbox (-.75,-1) rectangle (.75,1);
  \node[inner sep=3] (LU) at (-.25,.4) {};
  \node[inner sep=2] (RM) at (.25,0) {};
  \node[inner sep=2] (LD) at (-.25,-.4) {};
  \draw[red,very thick] (-.6,-1) to[out=90,in=-90] (LD.west) to[out=90,in=-90] (-.5,0) to[out=90,in=180] (LU.south) to[out=0,in=90] (RM.west) to[out=-90,in=90] (LD.east) to[out=-90,in=90] (0,-1);
  \draw[red,very thick] (LU.north) to[out=180,in=-90] (-.6,1);
  \draw[red,very thick] (LU.north) to[out=0,in=-90] (0,1);
  \draw[red,very thick] (.6,-1) to[out=90,in=-90] (RM.east) to[out=90,in=-90] (.6,1);
}

\defTikzBox{diagRiiiLeftUParWWV}{%
  \useasboundingbox (-.75,-1) rectangle (.75,1);
  \node[inner sep=3] (LU) at (-.25,.4) {};
  \node[inner sep=2] (RM) at (.25,0) {};
  \node[inner sep=3] (LD) at (-.25,-.4) {};
  \draw[red,very thick] (-.6,-1) to[out=90,in=180] (LD.south) to[out=0,in=90] (0,-1);
  \draw[red,very thick] (LD.north) to[out=180,in=-90] (-.5,0) to[out=90,in=180] (LU.south) to[out=0,in=90] (RM.west) to[out=-90,in=0] cycle;
  \draw[red,very thick] (LU.north) to[out=180,in=-90] (-.6,1);
  \draw[red,very thick] (LU.north) to[out=0,in=-90] (0,1);
  \draw[red,very thick] (.6,-1) to[out=90,in=-90] (RM.east) to[out=90,in=-90] (.6,1);
}

\newcommand\diagRiiiRightUParWWV{\reflectbox{\diagRiiiLeftUParWWV}}

\defTikzBox[baseline=-.5ex]{diagCrossPos}{%
  \node[circle,inner sep=2] (C) at (0,0) {};
  \draw[red,very thick] (.5,-.6) -- (.1,-.12) (-.1,.12) -- (-.5,.6);
  \draw[red,very thick] (-.5,-.6) -- (.5,.6);
}

\defTikzBox[baseline=-.5ex]{diagCrossPosUp}{%
  \node[circle,inner sep=2] (C) at (0,0) {};
  \draw[red,very thick,-stealth] (.5,-.6) -- (.1,-.12) (-.1,.12) -- (-.5,.6);
  \draw[red,very thick,-stealth] (-.5,-.6) -- (.5,.6);
}

\newcommand\diagCrossNegUp{\reflectbox{\diagCrossPosUp}}

\defTikzBox[baseline=-.5ex]{diagSingUp}{%
  \draw[red,very thick,-stealth] (-.5,-.6) -- (.5,.6);
  \draw[red,very thick,-stealth] (.5,-.6) -- (-.5,.6);
  \fill[blue] (0,0) circle (.1);
}

\defTikzBox[baseline=-.5ex]{diagSmoothV}{%
  \draw[red,very thick] (-.5,-.6) .. controls(0,0) .. (-.5,.6);
  \draw[red,very thick] (.5,-.6) .. controls(0,0) .. (.5,.6);
}

\defTikzBox[baseline=-.5ex]{diagSmoothUp}{%
  \draw[red,very thick,-stealth] (-.5,-.6) .. controls(0,0) .. (-.5,.6);
  \draw[red,very thick,-stealth] (.5,-.6) .. controls(0,0) .. (.5,.6);
}

\defTikzBox[baseline=-.5ex]{diagSmoothH}{%
  \draw[red,very thick] (-.5,-.6) .. controls(0,0) .. (.5,-.6);
  \draw[red,very thick] (-.5,.6) .. controls(0,0) .. (.5,.6);
}

\defTikzBox{diagRiiParUp}{%
  \useasboundingbox (-.6,-.75) rectangle (.6,.75);
  \draw[red,very thick,-stealth] (.5,-.75) to[out=135,in=225] (.5,.75);
  \draw[red,very thick,-stealth] (-.5,-.75) to[out=45,in=-45] (-.5,.75);
}

\defTikzBox{diagRiiLeftUp}{%
  \useasboundingbox (-.6,-.75) rectangle (.6,.75);
  \node[inner sep=2,circle] (U) at (0,.35) {};
  \node[inner sep=2,circle] (D) at (0,-.35) {};
  \draw[red,very thick,-stealth] (.5,-.75) -- (D) to[out=150,in=210] (U) -- (.5,.75);
  \draw[red,very thick,-stealth] (-.5,-.75) -- (D.center) to[out=30,in=-30] (U.center) -- (-.5,.75);
}

\newcommand\diagRiiLeftUpWith[2]{{%
  \ooalign{%
    \diagRiiLeftUp\crcr%
    \hss\raisebox{1.3ex}[0pt][0pt]{\rlap{\hskip.6em$\scriptstyle#1$}}\hss\crcr%
    \hss\raisebox{-1.3ex}[0pt][0pt]{\rlap{\hskip.6em$\scriptstyle#2$}}\hss\crcr%
  }%
}}

\defTikzBox{diagRiiRightUp}{%
  \useasboundingbox (-.6,-.75) rectangle (.6,.75);
  \node[inner sep=2,circle] (U) at (0,.35) {};
  \node[inner sep=2,circle] (D) at (0,-.35) {};
  \draw[red,very thick,-stealth] (.5,-.75) -- (D.center) to[out=150,in=210] (U.center) -- (.5,.75);
  \draw[red,very thick,-stealth] (-.5,-.75) -- (D) to[out=30,in=-30] (U) -- (-.5,.75);
}

\newcommand\diagRiiRightUpWith[2]{{%
  \ooalign{%
    \diagRiiRightUp\crcr%
    \hss\raisebox{1.3ex}[0pt][0pt]{\rlap{\hskip.6em$\scriptstyle#1$}}\hss\crcr%
    \hss\raisebox{-1.3ex}[0pt][0pt]{\rlap{\hskip.6em$\scriptstyle#2$}}\hss\crcr%
  }%
}}

\defTikzBox{diagRvFTwPos}{%
  \useasboundingbox (-.6,-.75) rectangle (.6,.75);
  \node[inner sep=2,circle] (U) at (0,.35) {};
  \node[inner sep=2,circle] (D) at (0,-.35) {};
  \draw[red,very thick,-stealth] (.5,-.75) -- (D) to[out=150,in=210] (U.center) -- (.5,.75);
  \draw[red,very thick,-stealth] (-.5,-.75) -- (D.center) to[out=30,in=-30] (U) -- (-.5,.75);
}

\newcommand\diagRvFTwPosWith[2]{{%
  \ooalign{%
    \diagRvFTwPos\crcr%
    \hss\raisebox{1.3ex}[0pt][0pt]{\rlap{\hskip.6em$\scriptstyle#1$}}\hss\crcr%
    \hss\raisebox{-1.3ex}[0pt][0pt]{\rlap{\hskip.6em$\scriptstyle#2$}}\hss\crcr%
  }%
}}

\defTikzBox{diagRvFTwNeg}{%
  \useasboundingbox (-.6,-.75) rectangle (.6,.75);
  \node[inner sep=2,circle] (U) at (0,.35) {};
  \node[inner sep=2,circle] (D) at (0,-.35) {};
  \draw[red,very thick,-stealth] (.5,-.75) -- (D.center) to[out=150,in=210] (U) -- (.5,.75);
  \draw[red,very thick,-stealth] (-.5,-.75) -- (D) to[out=30,in=-30] (U.center) -- (-.5,.75);
}

\newcommand\diagRvFTwNegWith[2]{{%
  \ooalign{%
    \diagRvFTwNeg\crcr%
    \hss\raisebox{1.3ex}[0pt][0pt]{\rlap{\hskip.6em$\scriptstyle#1$}}\hss\crcr%
    \hss\raisebox{-1.3ex}[0pt][0pt]{\rlap{\hskip.6em$\scriptstyle#2$}}\hss\crcr%
  }%
}}

\defTikzBox{diagTriviii}{%
  \useasboundingbox (-.75,-1) rectangle (.75,1);
  \draw[red,very thick] (-.6,-1) to[bend right] (-.6,1);
  \draw[red,very thick] (0,-1) -- (0,1);
  \draw[red,very thick] (.6,-1) to[bend left] (.6,1);
}

\defTikzBox{diagRiiiLeftMLR}{%
  \useasboundingbox (-.75,-1) rectangle (.75,1);
  \node[inner sep=2,circle] (LU) at (-.25,.4) {};
  \node[inner sep=2,circle] (RM) at (.25,0) {};
  \node[inner sep=2,circle] (LD) at (-.25,-.4) {};
  \draw[red,very thick,-stealth] (-.6,-1) to[out=90,in=210] (LD) (LD) to[out=30,in=225] (RM) (RM) to[out=45,in=-90] (.6,1);
  \draw[red,very thick,-stealth] (0,-1) to[out=90,in=-60] (LD.center) to[out=120,in=240] (LU.center) to[out=60,in=-90] (0,1);
  \draw[red,very thick,-stealth] (.6,-1) to[out=90,in=-45] (RM.center) to[out=135,in=-30] (LU) (LU) to[out=150,in=-90] (-.6,1);
}

\newcommand\diagRiiiLeftMLRWith[3]{%
  {%
    \ooalign{%
      \diagRiiiLeftMLR\crcr%
      \hss\raisebox{1.3ex}[0pt][0pt]{\llap{$\scriptstyle#1$\hskip.8em}}\hss\crcr%
      \hss\raisebox{-1.3ex}[0pt][0pt]{\llap{$\scriptstyle#2$\hskip.8em}}\hss\crcr%
      \hss\rlap{\hskip.8em$\scriptstyle#3$}\hss%
    }%
  }%
}

\newcommand\diagRiiiRightMRL{\reflectbox{\diagRiiiLeftMLR}}

\newcommand\diagRiiiRightMRLWith[3]{%
  {%
    \ooalign{%
      \diagRiiiRightMRL\crcr%
      \hss\raisebox{1.3ex}[0pt][0pt]{\rlap{\hskip.8em$\scriptstyle#1$}}\hss\crcr%
      \hss\raisebox{-1.3ex}[0pt][0pt]{\rlap{\hskip.8em$\scriptstyle#2$}}\hss\crcr%
      \hss\llap{$\scriptstyle#3$\hskip.8em}\hss%
    }%
  }%
}

\defTikzBox{diagRiiiLeftMRL}{%
  \useasboundingbox (-.75,-1) rectangle (.75,1);
  \node[inner sep=2,circle] (LU) at (-.25,.4) {};
  \node[inner sep=2,circle] (RM) at (.25,0) {};
  \node[inner sep=2,circle] (LD) at (-.25,-.4) {};
  \draw[red,very thick,-stealth] (-.6,-1) to[out=90,in=210] (LD) (LD) to[out=30,in=225] (RM.center) to[out=45,in=-90] (.6,1);
  \draw[red,very thick,-stealth] (0,-1) to[out=90,in=-60] (LD.center) to[out=120,in=240] (LU.center) to[out=60,in=-90] (0,1);
  \draw[red,very thick,-stealth] (.6,-1) to[out=90,in=-45] (RM) (RM) to[out=135,in=-30] (LU) (LU) to[out=150,in=-90] (-.6,1);
}

\newcommand\diagRiiiLeftMRLWith[3]{%
  {%
    \ooalign{%
      \diagRiiiLeftMRL\crcr%
      \hss\raisebox{1.3ex}[0pt][0pt]{\llap{$\scriptstyle#1$\hskip.8em}}\hss\crcr%
      \hss\raisebox{-1.3ex}[0pt][0pt]{\llap{$\scriptstyle#2$\hskip.8em}}\hss\crcr%
      \hss\rlap{\hskip.8em$\scriptstyle#3$}\hss%
    }%
  }%
}

\newcommand\diagRiiiRightMLR{\reflectbox{\diagRiiiLeftMRL}}

\newcommand\diagRiiiRightMLRWith[3]{%
  {%
    \ooalign{%
      \diagRiiiRightMLR\crcr%
      \hss\raisebox{1.3ex}[0pt][0pt]{\rlap{\hskip.8em$\scriptstyle#1$}}\hss\crcr%
      \hss\raisebox{-1.3ex}[0pt][0pt]{\rlap{\hskip.8em$\scriptstyle#2$}}\hss\crcr%
      \hss\llap{$\scriptstyle#3$\hskip.8em}\hss%
    }%
  }%
}

\defTikzBox{diagRiiiLeftLRM}{%
  \useasboundingbox (-.75,-1) rectangle (.75,1);
  \node[inner sep=2,circle] (LU) at (-.25,.4) {};
  \node[inner sep=2,circle] (RM) at (.25,0) {};
  \node[inner sep=2,circle] (LD) at (-.25,-.4) {};
  \draw[red,very thick,-stealth] (-.6,-1) to[out=90,in=210] (LD.center) to[out=30,in=225] (RM) (RM) to[out=45,in=-90] (.6,1);
  \draw[red,very thick,-stealth] (0,-1) to[out=90,in=-60] (LD) (LD) to[out=120,in=240] (LU) (LU) to[out=60,in=-90] (0,1);
  \draw[red,very thick,-stealth] (.6,-1) to[out=90,in=-45] (RM.center) to[out=135,in=-30] (LU.center) to[out=150,in=-90] (-.6,1);
}

\newcommand\diagRiiiLeftLRMWith[3]{%
  {%
    \ooalign{%
      \diagRiiiLeftLRM\crcr%
      \hss\raisebox{1.3ex}[0pt][0pt]{\llap{$\scriptstyle#1$\hskip.8em}}\hss\crcr%
      \hss\raisebox{-1.3ex}[0pt][0pt]{\llap{$\scriptstyle#2$\hskip.8em}}\hss\crcr%
      \hss\rlap{\hskip.8em$\scriptstyle#3$}\hss%
    }%
  }%
}

\newcommand\diagRiiiRightRLM{\reflectbox{\diagRiiiLeftLRM}}

\newcommand\diagRiiiRightRLMWith[3]{%
  {%
    \ooalign{%
      \diagRiiiRightRLM\crcr%
      \hss\raisebox{1.3ex}[0pt][0pt]{\rlap{\hskip.8em$\scriptstyle#1$}}\hss\crcr%
      \hss\raisebox{-1.3ex}[0pt][0pt]{\rlap{\hskip.8em$\scriptstyle#2$}}\hss\crcr%
      \hss\llap{$\scriptstyle#3$\hskip.8em}\hss%
    }%
  }%
}

\defTikzBox{diagRiiiLeftRLM}{%
  \useasboundingbox (-.75,-1) rectangle (.75,1);
  \node[inner sep=2,circle] (LU) at (-.25,.4) {};
  \node[inner sep=2,circle] (RM) at (.25,0) {};
  \node[inner sep=2,circle] (LD) at (-.25,-.4) {};
  \draw[red,very thick,-stealth] (-.6,-1) to[out=90,in=210] (LD.center) to[out=30,in=225] (RM.center) to[out=45,in=-90] (.6,1);
  \draw[red,very thick,-stealth] (0,-1) to[out=90,in=-60] (LD) (LD) to[out=120,in=240] (LU) (LU) to[out=60,in=-90] (0,1);
  \draw[red,very thick,-stealth] (.6,-1) to[out=90,in=-45] (RM) (RM) to[out=135,in=-30] (LU.center) to[out=150,in=-90] (-.6,1);
}

\newcommand\diagRiiiLeftRLMWith[3]{%
  {%
    \vphantom{\diagRiiiLeftRLM}%
    \ooalign{%
      \diagRiiiLeftRLM\crcr%
      \hss\raisebox{1.3ex}[0pt][0pt]{\llap{$\scriptstyle#1$\hskip.8em}}\hss\crcr%
      \hss\raisebox{-1.3ex}[0pt][0pt]{\llap{$\scriptstyle#2$\hskip.8em}}\hss\crcr%
      \hss\rlap{\hskip.8em$\scriptstyle#3$}\hss%
    }%
  }%
}

\newcommand\diagRiiiRightLRM{\reflectbox{\diagRiiiLeftRLM}}

\newcommand\diagRiiiRightLRMWith[3]{%
  {%
    \ooalign{%
      \diagRiiiRightLRM\crcr%
      \hss\raisebox{1.3ex}[0pt][0pt]{\rlap{\hskip.8em$\scriptstyle#1$}}\hss\crcr%
      \hss\raisebox{-1.3ex}[0pt][0pt]{\rlap{\hskip.8em$\scriptstyle#2$}}\hss\crcr%
      \hss\llap{$\scriptstyle#3$\hskip.8em}\hss%
    }%
  }%
}

\defTikzBox{diagRiiiLeftLMR}{%
  \useasboundingbox (-.75,-1) rectangle (.75,1);
  \node[inner sep=2,circle] (LU) at (-.25,.4) {};
  \node[inner sep=2,circle] (RM) at (.25,0) {};
  \node[inner sep=2,circle] (LD) at (-.25,-.4) {};
  \draw[red,very thick,-stealth] (-.6,-1) to[out=90,in=210] (LD) (LD) to[out=30,in=225] (RM) (RM) to[out=45,in=-90] (.6,1);
  \draw[red,very thick,-stealth] (0,-1) to[out=90,in=-60] (LD.center) to[out=120,in=240] (LU) (LU) to[out=60,in=-90] (0,1);
  \draw[red,very thick,-stealth] (.6,-1) to[out=90,in=-45] (RM.center) to[out=135,in=-30] (LU.center) to[out=150,in=-90] (-.6,1);
}

\newcommand\diagRiiiLeftLMRWith[3]{%
  {%
    \ooalign{%
      \diagRiiiLeftLMR\crcr%
      \hss\raisebox{1.3ex}[0pt][0pt]{\llap{$\scriptstyle#1$\hskip.8em}}\hss\crcr%
      \hss\raisebox{-1.3ex}[0pt][0pt]{\llap{$\scriptstyle#2$\hskip.8em}}\hss\crcr%
      \hss\rlap{\hskip.8em$\scriptstyle#3$}\hss%
    }%
  }%
}

\newcommand\diagRiiiRightRML{\reflectbox{\diagRiiiLeftLMR}}

\newcommand\diagRiiiRightRMLWith[3]{%
  {%
    \ooalign{%
      \diagRiiiRightRML\crcr%
      \hss\raisebox{1.3ex}[0pt][0pt]{\rlap{\hskip.8em$\scriptstyle#1$}}\hss\crcr%
      \hss\raisebox{-1.3ex}[0pt][0pt]{\rlap{\hskip.8em$\scriptstyle#2$}}\hss\crcr%
      \hss\llap{$\scriptstyle#3$\hskip.8em}\hss%
    }%
  }%
}

\defTikzBox{diagDeltaBrDLeft}{%
  \useasboundingbox (-.75,-1) rectangle (.75,1);
  \node[inner sep=2,circle] (LU) at (-.25,.4) {};
  \node[inner sep=2,circle] (RM) at (.25,0) {};
  \node[inner sep=2,circle] (LD) at (-.25,-.4) {};
  \draw[red,very thick,-stealth] (-.6,-1) to[out=90,in=210] (LD) (LD) to[out=30,in=225] (RM.center) to[out=45,in=-90] (.6,1);
  \draw[red,very thick,-stealth] (0,-1) to[out=90,in=-60] (LD.center) to[out=120,in=240] (LU) (LU) to[out=60,in=-90] (0,1);
  \draw[red,very thick,-stealth] (.6,-1) to[out=90,in=-45] (RM) (RM) to[out=135,in=-30] (LU.center) to[out=150,in=-90] (-.6,1);
}

\newcommand\diagDeltaBrDLeftWith[3]{%
  {%
    \ooalign{%
      \diagDeltaBrDLeft\crcr%
      \hss\raisebox{1.3ex}[0pt][0pt]{\llap{$\scriptstyle#1$\hskip.8em}}\hss\crcr%
      \hss\raisebox{-1.3ex}[0pt][0pt]{\llap{$\scriptstyle#2$\hskip.8em}}\hss\crcr%
      \hss\rlap{\hskip.8em$\scriptstyle#3$}\hss%
    }%
  }%
}

\newcommand\diagDeltaBrDRight{\reflectbox{\diagDeltaBrDLeft}}

\newcommand\diagDeltaBrDRightWith[3]{%
  {%
    \ooalign{%
      \diagDeltaBrDRight\crcr%
      \hss\raisebox{1.3ex}[0pt][0pt]{\rlap{\hskip.8em$\scriptstyle#1$}}\hss\crcr%
      \hss\raisebox{-1.3ex}[0pt][0pt]{\rlap{\hskip.8em$\scriptstyle#2$}}\hss\crcr%
      \hss\llap{$\scriptstyle#3$\hskip.8em}\hss%
    }%
  }%
}

\defTikzBox{diagRiiiLeftUParVVV}{%
  \useasboundingbox (-.75,-1) rectangle (.75,1);
  \node[inner sep=2] (LU) at (-.25,.4) {};
  \node[inner sep=2] (RM) at (.25,0) {};
  \node[inner sep=2] (LD) at (-.25,-.4) {};
  \draw[red,very thick] (-.6,-1) to[out=90,in=-90] (LD.west) to[out=90,in=-90] (-.5,0) to[out=90,in=-90] (LU.west) to[out=90,in=-90] (-.6,1);
  \draw[red,very thick] (0,-1) to[out=90,in=-90] (LD.east) to[out=90,in=-90] (RM.west) to[out=90,in=-90] (LU.east) to[out=90,in=-90] (0,1);
  \draw[red,very thick] (.6,-1) to[out=90,in=-90] (RM.east) to[out=90,in=-90] (.6,1);
}

\newcommand\diagRiiiRightUParVVV{\reflectbox{\diagRiiiLeftUParVVV}}

\defTikzBox[baseline=-.5ex]{diagFiSing}{%
  \node[circle,inner sep=1,fill=blue] (C) at (0,0) {};
  \draw[red,very thick,-stealth] (-.5,-.6) to[out=60,in=225,looseness=.5] (C) to[out=45,in=90,looseness=1.5] (.6,0) to[out=-90,in=-45,looseness=1.5] (C) to[out=135,in=-60,looseness=.5] (-.5,.6);
}

\defTikzBox[baseline=-.5ex]{diagFiPos}{%
  \node[circle,inner sep=2] (C) at (0,0) {};
  \draw[red,very thick,-stealth] (-.5,-.6) to[out=60,in=225,looseness=.5] (C.south west) -- (C.north east) to[out=45,in=90,looseness=1.5] (.6,0) to[out=-90,in=-45,looseness=1.5] (C) (C) to[out=135,in=-60,looseness=.5] (-.5,.6);
}

\defTikzBox[baseline=-.5ex]{diagFiNeg}{%
  \node[circle,inner sep=2] (C) at (0,0) {};
  \draw[red,very thick,-stealth] (-.5,-.6) to[out=60,in=225,looseness=.5] (C) (C) to[out=45,in=90,looseness=1.5] (.6,0) to[out=-90,in=-45,looseness=1.5] (C.south east) -- (C.north west) to[out=135,in=-60,looseness=.5] (-.5,.6);
}

\defTikzBox[baseline=-.5ex]{diagFiCrit}{%
  \draw[red,very thick,-stealth] (-.5,-1) to[out=60,in=-150] (0,0) to[out=150,in=-60] (-.5,1);
}

\defTikzBox[baseline=-.5ex]{diagFiNil}{%
  \draw[red,very thick,-stealth] (-.5,-.6) to[out=45,in=-45] (-.5,.6);
}

\defTikzBox[baseline=-.5ex]{diagFourTLND}{%
  \useasboundingbox (-.75,-1) rectangle (.75,1);
  \node[inner sep=2,circle] (LU) at (-.25,.4) {};
  \node[inner sep=0,circle] (RM) at (.25,0) {};
  \node[inner sep=0,circle] (LD) at (-.25,-.4) {};
  \draw[red,very thick,-stealth] (-.6,-1) to[out=90,in=210] (LD.center) to[out=30,in=225] (RM.center) to[out=45,in=-90] (.6,1);
  \draw[red,very thick,-stealth] (0,-1) to[out=90,in=-60] (LD.center) to[out=120,in=240] (LU) (LU) to[out=60,in=-90] (0,1);
  \draw[red,very thick,-stealth] (.6,-1) to[out=90,in=-45] (RM.center) to[out=135,in=-30] (LU.center) to[out=150,in=-90] (-.6,1);
  \fill[blue] (RM) circle (.1);
  \fill[blue] (LD) circle (.1);
}

\newcommand\diagFourTRPD{\reflectbox{\diagFourTLND}}

\defTikzBox[baseline=-.5ex]{diagFourTLDN}{%
  \useasboundingbox (-.75,-1) rectangle (.75,1);
  \node[inner sep=0,circle] (LU) at (-.25,.4) {};
  \node[inner sep=0,circle] (RM) at (.25,0) {};
  \node[inner sep=2,circle] (LD) at (-.25,-.4) {};
  \draw[red,very thick,-stealth] (-.6,-1) to[out=90,in=210] (LD) (LD) to[out=30,in=225] (RM.center) to[out=45,in=-90] (.6,1);
  \draw[red,very thick,-stealth] (0,-1) to[out=90,in=-60] (LD.center) to[out=120,in=240] (LU.center) to[out=60,in=-90] (0,1);
  \draw[red,very thick,-stealth] (.6,-1) to[out=90,in=-45] (RM.center) to[out=135,in=-30] (LU.center) to[out=150,in=-90] (-.6,1);
  \fill[blue] (RM) circle (.1);
  \fill[blue] (LU) circle (.1);
}

\newcommand\diagFourTRDP{\reflectbox{\diagFourTLDN}}

\defTikzBox[baseline=-.5ex]{diagDeltaLeftPosUUU}{%
  \useasboundingbox (-.75,-1) rectangle (.75,1);
  \node[inner sep=2,circle] (LU) at (-.25,.4) {};
  \node[inner sep=2,circle] (RM) at (.25,0) {};
  \node[inner sep=2,circle] (LD) at (-.25,-.4) {};
  \draw[red,very thick,-stealth] (-.6,-1) to[out=90,in=210] (LD) (LD) to[out=30,in=225] (RM.center) to[out=45,in=-90] (.6,1);
  \draw[red,very thick,-stealth] (0,-1) to[out=90,in=-60] (LD.center) to[out=120,in=240] (LU) (LU) to[out=60,in=-90] (0,1);
  \draw[red,very thick,-stealth] (.6,-1) to[out=90,in=-45] (RM) (RM) to[out=135,in=-30] (LU.center) to[out=150,in=-90] (-.6,1);
}

\defTikzBox[baseline=-.5ex]{diagDeltaLeftNegUUU}{%
  \useasboundingbox (-.75,-1) rectangle (.75,1);
  \node[inner sep=2,circle] (LU) at (-.25,.4) {};
  \node[inner sep=2,circle] (RM) at (.25,0) {};
  \node[inner sep=2,circle] (LD) at (-.25,-.4) {};
  \draw[red,very thick,-stealth] (-.6,-1) to[out=90,in=210] (LD.center) to[out=30,in=225] (RM) (RM) to[out=45,in=-90] (.6,1);
  \draw[red,very thick,-stealth] (0,-1) to[out=90,in=-60] (LD) (LD) to[out=120,in=240] (LU.center) to[out=60,in=-90] (0,1);
  \draw[red,very thick,-stealth] (.6,-1) to[out=90,in=-45] (RM.center) to[out=135,in=-30] (LU) (LU) to[out=150,in=-90] (-.6,1);
}

\defTikzBox[baseline=-.5ex]{diagCrossSingRivOL}{%
  \useasboundingbox (-.75,-1) rectangle (.75,1);
  \node[inner sep=2,circle] (LU) at (-.25,.4) {};
  \node[inner sep=2,circle] (RM) at (.25,0) {};
  \node[inner sep=2,circle] (LD) at (-.25,-.4) {};
  \draw[red,very thick,-stealth] (-.6,-1) to[out=90,in=210] (LD) (LD) to[out=30,in=225] (RM.center) to[out=45,in=-90] (.6,1);
  \draw[red,very thick,-stealth] (0,-1) to[out=90,in=-60] (LD.center) to[out=120,in=240] (LU.center) to[out=60,in=-90] (0,1);
  \draw[red,very thick,-stealth] (.6,-1) to[out=90,in=-45] (RM.center) to[out=135,in=-30] (LU) (LU) to[out=150,in=-90] (-.6,1);
  \fill[blue] (RM) circle (.1);
}

\newcommand\diagCrossSingRivOLWith[2]{%
  {\ooalign{\diagCrossSingRivOL\crcr\hss\raisebox{8pt}{\llap{$\scriptstyle#1$\hskip.75em}}\hss\crcr\hss\raisebox{-7pt}{\llap{$\scriptstyle#2$\hskip.75em}}\hss}}
}

\newcommand\diagCrossSingRivOR{\reflectbox{\diagCrossSingRivOL}}

\newcommand\diagCrossSingRivORWith[2]{%
  {\ooalign{\diagCrossSingRivOR\crcr\hss\raisebox{8pt}{\rlap{\hskip.75em$\scriptstyle#1$}}\hss\crcr\hss\raisebox{-7pt}{\rlap{\hskip.75em$\scriptstyle#2$}}\hss}}
}

\defTikzBox[baseline=-.5ex]{diagCrossSingRivUL}{%
  \useasboundingbox (-.75,-1) rectangle (.75,1);
  \node[inner sep=2,circle] (LU) at (-.25,.4) {};
  \node[inner sep=2,circle] (RM) at (.25,0) {};
  \node[inner sep=2,circle] (LD) at (-.25,-.4) {};
  \draw[red,very thick,-stealth] (-.6,-1) to[out=90,in=210] (LD.center) to[out=30,in=225] (RM.center) to[out=45,in=-90] (.6,1);
  \draw[red,very thick,-stealth] (0,-1) to[out=90,in=-60] (LD) (LD) to[out=120,in=240] (LU) (LU) to[out=60,in=-90] (0,1);
  \draw[red,very thick,-stealth] (.6,-1) to[out=90,in=-45] (RM.center) to[out=135,in=-30] (LU.center) to[out=150,in=-90] (-.6,1);
  \fill[blue] (RM) circle (.1);
}

\newcommand\diagCrossSingRivULWith[2]{%
  {\ooalign{\diagCrossSingRivUL\crcr\hss\raisebox{8pt}{\llap{$\scriptstyle#1$\hskip.75em}}\hss\crcr\hss\raisebox{-7pt}{\llap{$\scriptstyle#2$\hskip.75em}}\hss}}
}

\newcommand\diagCrossSingRivUR{\reflectbox{\diagCrossSingRivUL}}

\newcommand\diagCrossSingRivURWith[2]{%
  {\ooalign{\diagCrossSingRivUR\crcr\hss\raisebox{8pt}{\rlap{\hskip.75em$\scriptstyle#1$}}\hss\crcr\hss\raisebox{-7pt}{\rlap{\hskip.75em$\scriptstyle#2$}}\hss}}
}

\defTikzBox[baseline=-.5ex]{diagCrossSingRivLNN}{%
  \useasboundingbox (-.75,-1) rectangle (.75,1);
  \node[inner sep=2,circle] (LU) at (-.25,.4) {};
  \node[inner sep=2,circle] (RM) at (.25,0) {};
  \node[inner sep=2,circle] (LD) at (-.25,-.4) {};
  \draw[red,very thick,-stealth] (-.6,-1) to[out=90,in=210] (LD) (LD) to[out=30,in=225] (RM.center) to[out=45,in=-90] (.6,1);
  \draw[red,very thick,-stealth] (0,-1) to[out=90,in=-60] (LD.center) to[out=120,in=240] (LU) (LU) to[out=60,in=-90] (0,1);
  \draw[red,very thick,-stealth] (.6,-1) to[out=90,in=-45] (RM.center) to[out=135,in=-30] (LU.center) to[out=150,in=-90] (-.6,1);
  \fill[blue] (RM) circle (.1);
}

\newcommand\diagCrossSingRivLNNWith[2]{%
  {\ooalign{\diagCrossSingRivLNN\crcr\hss\raisebox{8pt}{\llap{$\scriptstyle#1$\hskip.75em}}\hss\crcr\hss\raisebox{-7pt}{\llap{$\scriptstyle#2$\hskip.75em}}\hss}}
}

\newcommand\diagCrossSingRivRPP{\reflectbox{\diagCrossSingRivLNN}}

\newcommand\diagCrossSingRivRPPWith[2]{%
  {\ooalign{\diagCrossSingRivRPP\crcr\hss\raisebox{8pt}{\rlap{\hskip.75em$\scriptstyle#1$}}\hss\crcr\hss\raisebox{-7pt}{\rlap{\hskip.75em$\scriptstyle#2$}}\hss}}
}

\defTikzBox[baseline=-.5ex]{diagCrossVRivOL}{%
  \useasboundingbox (-.75,-1) rectangle (.75,1);
  \node[inner sep=2,circle] (LU) at (-.25,.4) {};
  \node[inner sep=3,circle] (RM) at (.25,0) {};
  \node[inner sep=2,circle] (LD) at (-.25,-.4) {};
  \draw[red,very thick] (-.6,-1) to[out=90,in=210] (LD) (LD) to[out=30,in=225] (RM.south west) to[out=45,in=-45] (RM.north west) to[out=135,in=-30] (LU) (LU) to[out=150,in=-90] (-.6,1);
  \draw[red,very thick] (0,-1) to[out=90,in=-60] (LD.center) to[out=120,in=240] (LU.center) to[out=60,in=-90] (0,1);
  \draw[red,very thick] (.6,-1) to[out=90,in=-45] (RM.south east) to[out=135,in=225] (RM.north east) to[out=45,in=-90] (.6,1);
}

\newcommand\diagCrossVRivOR{\reflectbox{\diagCrossVRivOL}}

\defTikzBox[baseline=-.5ex]{diagCrossVRivUL}{%
  \useasboundingbox (-.75,-1) rectangle (.75,1);
  \node[inner sep=2,circle] (LU) at (-.25,.4) {};
  \node[inner sep=3,circle] (RM) at (.25,0) {};
  \node[inner sep=2,circle] (LD) at (-.25,-.4) {};
  \draw[red,very thick] (-.6,-1) to[out=90,in=210] (LD.center) to[out=30,in=225] (RM.south west) to[out=45,in=-45] (RM.north west) to[out=135,in=-30] (LU.center) to[out=150,in=-90] (-.6,1);
  \draw[red,very thick] (0,-1) to[out=90,in=-60] (LD) (LD) to[out=120,in=240] (LU) (LU) to[out=60,in=-90] (0,1);
  \draw[red,very thick] (.6,-1) to[out=90,in=-45] (RM.south east) to[out=135,in=225] (RM.north east) to[out=45,in=-90] (.6,1);
}

\newcommand\diagCrossVRivUR{\reflectbox{\diagCrossVRivUL}}

\defTikzBox[baseline=-.5ex]{diagCrossVRivLNN}{%
  \useasboundingbox (-.75,-1) rectangle (.75,1);
  \node[inner sep=2,circle] (LU) at (-.25,.4) {};
  \node[inner sep=3,circle] (RM) at (.25,0) {};
  \node[inner sep=2,circle] (LD) at (-.25,-.4) {};
  \draw[red,very thick] (-.6,-1) to[out=90,in=210] (LD) (LD) to[out=30,in=225] (RM.south west) to[out=45,in=-45] (RM.north west) to[out=135,in=-30] (LU.center) to[out=150,in=-90] (-.6,1);
  \draw[red,very thick] (0,-1) to[out=90,in=-60] (LD.center) to[out=120,in=240] (LU) (LU) to[out=60,in=-90] (0,1);
  \draw[red,very thick] (.6,-1) to[out=90,in=-45] (RM.south east) to[out=135,in=225] (RM.north east) to[out=45,in=-90] (.6,1);
}

\defTikzBox[baseline=-.5ex]{diagCrossHRivOL}{%
  \useasboundingbox (-.75,-1) rectangle (.75,1);
  \node[inner sep=2,circle] (LU) at (-.25,.4) {};
  \node[inner sep=3,circle] (RM) at (.25,0) {};
  \node[inner sep=2,circle] (LD) at (-.25,-.4) {};
  \draw[red,very thick] (-.6,-1) to[out=90,in=210] (LD) (LD) to[out=30,in=225] (RM.south west) to[out=45,in=135] (RM.south east) to[out=-45,in=90] (.6,-1);
  \draw[red,very thick] (.6,1) to[out=-90,in=45] (RM.north east) to[out=225,in=-45] (RM.north west) to[out=135,in=-30] (LU) (LU) to[out=150,in=-90] (-.6,1);
  \draw[red,very thick] (0,-1) to[out=90,in=-60] (LD.center) to[out=120,in=240] (LU.center) to[out=60,in=-90] (0,1);
}

\newcommand\diagCrossHRivOR{\reflectbox{\diagCrossHRivOL}}

\defTikzBox[baseline=-.5ex]{diagCrossHRivUL}{%
  \useasboundingbox (-.75,-1) rectangle (.75,1);
  \node[inner sep=2,circle] (LU) at (-.25,.4) {};
  \node[inner sep=3,circle] (RM) at (.25,0) {};
  \node[inner sep=2,circle] (LD) at (-.25,-.4) {};
  \draw[red,very thick] (-.6,-1) to[out=90,in=210] (LD.center) to[out=30,in=225] (RM.south west) to[out=45,in=135] (RM.south east) to[out=-45,in=90] (.6,-1);
  \draw[red,very thick] (.6,1) to[out=-90,in=45] (RM.north east) to[out=225,in=-45] (RM.north west) to[out=135,in=-30] (LU.center) to[out=150,in=-90] (-.6,1);
  \draw[red,very thick] (0,-1) to[out=90,in=-60] (LD) (LD) to[out=120,in=240] (LU) (LU) to[out=60,in=-90] (0,1);
}

\newcommand\diagCrossHRivUR{\reflectbox{\diagCrossHRivUL}}

\defTikzBox[baseline=-.5ex]{diagCrossHRivLNN}{%
  \useasboundingbox (-.75,-1) rectangle (.75,1);
  \node[inner sep=2,circle] (LU) at (-.25,.4) {};
  \node[inner sep=3,circle] (RM) at (.25,0) {};
  \node[inner sep=2,circle] (LD) at (-.25,-.4) {};
  \draw[red,very thick] (-.6,-1) to[out=90,in=210] (LD) (LD) to[out=30,in=225] (RM.south west) to[out=45,in=135] (RM.south east) to[out=-45,in=90] (.6,-1);
  \draw[red,very thick] (.6,1) to[out=-90,in=45] (RM.north east) to[out=225,in=-45] (RM.north west) to[out=135,in=-30] (LU.center) to[out=150,in=-90] (-.6,1);
  \draw[red,very thick] (0,-1) to[out=90,in=-60] (LD.center) to[out=120,in=240] (LU) (LU) to[out=60,in=-90] (0,1);
}

\defTikzBox[baseline=-.5ex]{diagRvSingU}{%
  \useasboundingbox (-.75,-1) rectangle (.75,1);
  \draw[red,very thick,-stealth] (.5,-1) .. controls(-.5,-.3)and(-.5,.3) .. (.5,1);
  \fill[white] (0,-.6) circle (.15);
  \draw[red,very thick,-stealth] (-.5,-1) .. controls(.5,-.3)and(.5,.3) .. (-.5,1);
  \fill[blue] (0,.6) circle (.1);
}

\defTikzBox[baseline=-.5ex]{diagRvSingD}{%
  \useasboundingbox (-.75,-1) rectangle (.75,1);
  \draw[red,very thick,-stealth] (-.5,-1) .. controls(.5,-.3)and(.5,.3) .. (-.5,1);
  \fill[white] (0,.6) circle (.15);
  \draw[red,very thick,-stealth] (.5,-1) .. controls(-.5,-.3)and(-.5,.3) .. (.5,1);
  \fill[blue] (0,-.6) circle (.1);
}

\defTikzBox[xlen=.4pt,ylen=-.4pt]{BordDeltaN}{%
  \draw[red,very thick]  (47.54,16.43) to[quadratic={(39.43,16.43)}] (33.46,17.1);
  \draw[red,very thick,dotted]  (33.46,17.1) to[quadratic={(24.69,18.08)}] (20.5,20.5) to[quadratic={(18.67,21.56)}] (18.67,22.34);
  \draw[red,very thick] (18.67,22.34) to[quadratic={(18.67,24.57)}] (33.46,24.57);
  \draw[red,very thick,dotted]  (-32.46,16.43) to[quadratic={(-17.67,16.43)}] (-17.67,18.66);
  \draw[red,very thick] (-17.67,18.66) to[quadratic={(-17.67,19.44)}] (-19.5,20.5) to[quadratic={(-26.54,24.57)}] (-46.54,24.57);
  \draw[black]  (18.66,22.34) .. controls (18.66,-4.324)and(-17.66,-8.01) .. (-17.66,18.66);
  \draw[black] (47.54,16.43) -- (47.54,-23.57);
  \draw[black] (33.46,24.57) -- (33.46,-15.43);
  \draw[black,thick,dotted] (-32.46,16.43) -- (-32.46,-15.43);
  \draw[black] (-32.46,-15.43) -- (-32.46,-23.57);
  \draw[black] (-46.54,24.57) -- (-46.54,-15.43);
  \draw[red,very thick] (47.54,-23.57) -- (-32.46,-23.57);
  \draw[red,very thick] (33.46,-15.43) -- (-46.54,-15.43);
}

\defTikzBox[xlen=.4pt,ylen=-.4pt]{BordDeltaP}{%
  \draw[red,very thick] (47.54,16.43) -- (33.46,16.43);
  \draw[red,very thick,dotted] (33.46,16.43) -- (-32.46,16.43);
  \draw[red,very thick] (33.46,24.57) -- (-46.54,24.57);
  \draw[black] (18.66,-17.66) .. controls (18.66,9.01)and(-17.66,5.324) .. (-17.66,-21.34);
  \draw[black] (47.54,16.43) -- (47.54,-23.57);
  \draw[black] (33.46,24.57) -- (33.46,-15.43);
  \draw[black,thick,dotted] (-32.46,16.43) -- (-32.46,-16.1);
  \draw[black] (-32.46,-16.1) -- (-32.46,-23.57);
  \draw[black] (-46.54,24.57) -- (-46.54,-15.43);
  \draw[red,very thick] (47.54,-23.57) to[quadratic={(27.54,-23.57)}] (20.5,-19.5) to[quadratic={(13.46,-15.43)}] (33.46,-15.43);
  \draw[red,very thick]  (-32.46,-23.57) to[quadratic={(-12.46,-23.57)}] (-19.5,-19.5) to[quadratic={(-26.54,-15.43)}] (-46.54,-15.43);
}

\defTikzBox[xlen=.4pt,ylen=-.4pt]{BordPhiFst}{%
  \draw[red,very thick] (47.54,36.43) -- (33.46,36.43);
  \draw[red,very thick,dotted]  (33.46,36.43) -- (-32.46,36.43);
  \draw[red,very thick] (33.46,44.57) -- (-46.54,44.57);
  \draw[black] (18.66,2.343) .. controls (18.66,29.01)and(-17.66,25.32) .. (-17.66,-1.343) .. controls (-17.66,-28.01)and(18.66,-24.32) .. (18.66,2.343);
  \draw[black] (47.54,36.43) -- (47.54,-43.57);
  \draw[black] (33.46,44.57) -- (33.46,-35.43);
  \draw[black,thick,dotted] (-32.46,36.43) -- (-32.46,-35.43);
  \draw[black] (-32.46,-35.43) -- (-32.46,-43.57);
  \draw[black] (-46.54,44.57) -- (-46.54,-35.43);
  \draw[red,very thick] (47.54,-43.57) -- (-32.46,-43.57);
  \draw[red,very thick] (33.46,-35.43) -- (-46.54,-35.43);
}

\defTikzBox[xlen=.4pt,ylen=-.4pt]{BordPhiSnd}{%
  \draw[red,very thick] (47.54,36.43) -- (33.46,36.43);
  \draw[red,very thick,dotted]  (33.46,36.43) -- (-32.46,36.43);
  \draw[red,very thick] (33.46,44.57) -- (-46.54,44.57);
  \draw[black] (23.92,32.57) .. controls+(0,-10)and+(0,15) .. (50.92,4.57) .. controls+(0,-15)and+(0,10) .. (23.92,-23.43);
  \draw[black] (23.92,10.57) .. controls+(0,5)and+(0,7) .. (37.92,4.57) .. controls+(0,-7)and+(0,-5) .. (23.92,-1.43);
  \draw[black] (47.54,36.43) -- (47.54,13);
  \draw[black,thick,dotted] (47.54,13) -- (47.54,-3.7);
  \draw[black] (47.54,-3.7) -- (47.54,-43.57);
  \draw[black] (33.46,44.57) -- (33.46,22.9);
  \draw[black,thick,dotted] (33.46,22.9) -- (33.46,11.1);
  \draw[black] (33.46,11.1) -- (33.46,-1.8);
  \draw[black,thick,dotted] (33.46,-1.8) -- (33.46,-13.7);
  \draw[black] (33.46,-13.7) -- (33.46,-35.43);
  \draw[black,thick,dotted] (-32.46,36.43) -- (-32.46,-35.43);
  \draw[black] (-32.46,-35.43) -- (-32.46,-43.57);
  \draw[black] (-46.54,44.57) -- (-46.54,-35.43);
  \draw[red,very thick] (47.54,-43.57) -- (-32.46,-43.57);
  \draw[red,very thick] (33.46,-35.43) -- (-46.54,-35.43);
}

\defTikzBox[xlen=.25pt,ylen=-.4pt]{BordEquivOO}{%
  \draw[red,very thick] (-133.6,48.63) to[quadratic={(-113.6,48.63)}] (-106.5,44.57) to[quadratic={(-104.69,43.51)}] (-104.69,42.73);
  \draw[red,very thick,dotted] (-104.69,42.73) to[quadratic={(-104.69,40.5)}] (-119.5,40.5);
  \draw[red,very thick,dotted] (-105.4,32.37) -- (-102.7,32.37);
  \draw[red,very thick] (-102.7,32.37) -- (-64.5,32.37);
  \draw[red,very thick,dotted] (-64.5,32.37) -- (-25.41,32.37) to[quadratic={(-5.41,32.37)}] (-12.46,36.43) to[quadratic={(-19.5,40.5)}] (-39.5,40.5) to[quadratic={(-59.5,40.5)}] (-66.54,44.57) to [quadratic={(-68.38,45.63)}] (-68.38,46.41);
  \draw[red,very thick] (-68.38,46.41) to[quadratic={(-68.38,48.63)}] (-53.59,48.63) -- (106.4,48.63);
  \draw[red,very thick] (120.5,40.5) -- (106.4,40.5);
  \draw[red,very thick,dotted](106.4,40.5) -- (40.5,40.5) to[quadratic={(25.71,40.5)}] (25.71,38.28) to[quadratic={(25.71,37.49)}] (27.54,36.43) to[quadratic={(34.59,32.37)}] (54.59,32.37) -- (120.5,32.37);
  \draw[red,very thick]  (120.5,32.37) -- (134.6,32.37);
  \draw[black,thick,dotted] (25.71,38.28) .. controls (25.71,2.07)and(91.28,-1.36) .. (103.69,-31.37);
  \draw[black] (103.69,-31.37) .. controls (104.98,-34.51)and(105.7,-37.93) .. (105.7,-41.72);
  \draw[black] (-104.7,42.72) .. controls (-104.7,2.72)and(-24.71,2.72) .. (-24.71,-37.28);
  \draw[black] (-68.38,46.41) .. controls (-68.38,6.41)and(11.62,6.41) .. (11.62,-33.59);
  \draw[black,thick,dotted] (-10.62,34.59) .. controls (-10.62,-0.084)and(49.49,-8.04) .. (65.5,-32.7);
  \draw[black] (65.5,-32.7) .. controls (67.96,-36.49)and(69.38,-40.68) .. (69.38,-45.41);
  \draw[black] (-133.6,48.63) -- (-133.6,-31.37);
  \draw[black,thick,dotted] (-119.5,40.5) -- (-119.5,-31.37);
  \draw[black] (-119.5,-31.37) -- (-119.5,-39.5);
  \draw[black,thick,dotted] (-105.4,32.37) -- (-105.4,-39.5);
  \draw[black] (-105.4,-39.5) -- (-105.4,-47.63);
  \draw[black] (106.4,48.63) -- (106.4,-31.37);
  \draw[black] (120.5,40.5) -- (120.5,-39.5);
  \draw[black] (134.6,32.37) -- (134.6,-47.63);
  \draw[red,very thick] (-133.6,-31.37) -- (-53.59,-31.37) to[quadratic={(-33.59,-31.37)}] (-26.54,-35.43) to[quadratic={(-19.5,-39.5)}] (-39.5,-39.5) -- (-119.5,-39.5);
  \draw[red,very thick] (-105.4,-47.63) -- (54.59,-47.63) to[quadratic={(74.59,-47.63)}] (67.54,-43.57) to[quadratic={(60.5,-39.5)}] (40.5,-39.5) to[quadratic={(20.5,-39.5)}] (13.46,-35.43) to[quadratic={(6.41,-31.37)}] (26.41,-31.37) -- (106.4,-31.37);
  \draw[red,very thick]  (120.5,-39.5) to[quadratic={(100.5,-39.5)}] (107.5,-43.57) to[quadratic={(114.6,-47.63)}] (134.6,-47.63);
}

\defTikzBox[xlen=.25pt,ylen=-.4pt]{BordEquivOl}{%
  \draw[red,very thick] (-133.6,48.63) -- (106.4,48.63);
  \draw[red,very thick,dotted]  (-119.5,40.5) -- (-39.5,40.5) to[quadratic={(-19.5,40.5)}] (-12.46,36.43) to[quadratic={(-5.41,32.37)}] (-25.41,32.37) -- (-105.4,32.37);
  \draw[red,very thick] (120.5,40.5) -- (106.4,40.5);
  \draw[red,very thick,dotted] (106.4,40.5) -- (40.5,40.5) to[quadratic={(25.71,40.5)}] (25.71,38.28) to[quadratic={(25.71,37.49)}] (27.54,36.43) to[quadratic={(34.59,32.37)}] (54.59,32.37) -- (120.5,32.37);
  \draw[red,very thick] (120.5,32.37) -- (134.6,32.37);
  \draw[black,thick,dotted] (-10.62,34.59) .. controls(-10.62,7.7)and(-74.88,4.91) .. (-88.21,-31.67);
  \draw[black] (-88.21,-31.67) .. controls(-89.76,-35.57)and(-90.62,-40.22) .. (-90.62,-45.41);
  \draw[black,thick,dotted] (25.71,38.28) .. controls (25.71,10.51)and(94.25,8.45) .. (104.44,-31.37);
  \draw[black] (104.44,-31.37) .. controls (105.26,-34.57)and(105.7,-38) .. (105.7,-41.72);
  \draw[black] (-54.29,-41.72) .. controls (-54.29,-37.97)and(-53.59,-34.52) .. (-52.31,-31.37);
  \draw[black,thick,dotted] (-52.31,-31.37) .. controls(-38.21,3.12)and(46.81,1.75) .. (65.67,-31.37);
  \draw[black] (65.67,-31.37) .. controls (68.05,-35.55)and(69.38,-40.23) .. (69.38,-45.41);
  \draw[black] (-24.71,-37.28) .. controls (-24.71,-17.28)and(11.62,-13.59) .. (11.62,-33.59);
  \draw[black] (-133.6,48.63) -- (-133.6,-31.37);
  \draw[black,thick,dotted] (-119.5,40.5) -- (-119.5,-31.37);
  \draw[black] (-119.5,-31.37) -- (-119.5,-39.5);
  \draw[black,thick,dotted] (-105.4,32.37) -- (-105.4,-40.17);
  \draw[black] (-105.4,-40.17) -- (-105.4,-47.63);
  \draw[black] (106.4,48.63) -- (106.4,-31.37);
  \draw[black] (120.5,40.5) -- (120.5,-39.5);
  \draw[black] (134.6,32.37) -- (134.6,-47.63);
  \draw[red,very thick] (-133.6,-31.37) -- (-53.59,-31.37) to[quadratic={(-33.59,-31.37)}] (-26.54,-35.43) to[quadratic={(-19.5,-39.5)}] (-39.5,-39.5) to[quadratic={(-59.5,-39.5)}] (-52.46,-43.57) to[quadratic={(-45.41,-47.63)}] (-25.41,-47.63) -- (54.59,-47.63) to[quadratic={(74.59,-47.63)}] (67.54,-43.57) to[quadratic={(60.5,-39.5)}] (40.5,-39.5) to[quadratic={(20.5,-39.5)}] (13.46,-35.43) to[quadratic={(6.41,-31.37)}] (26.41,-31.37) -- (106.4,-31.37);
  \draw[red,very thick]  (-119.5,-39.5) to[quadratic={(-99.5,-39.5)}] (-92.46,-43.57) to[quadratic={(-85.41,-47.63)}] (-105.4,-47.63);
  \draw[red,very thick]  (120.5,-39.5) to[quadratic={(100.5,-39.5)}] (107.5,-43.57) to[quadratic={(114.6,-47.63)}] (134.6,-47.63);
}

\defTikzBox[xlen=.25pt,ylen=-.4pt]{BordEquivlO}{
  \draw[red,very thick] (-133.6,48.63) to[quadratic={(-113.6,48.63)}] (-106.5,44.57) to[quadratic={(-104.69,43.51)}] (-104.69,42.73);
  \draw[red,very thick,dotted] (-104.69,42.73) to[quadratic={(-104.69,40.5)}] (-119.5,40.5);
  \draw[red,very thick,dotted] (-105.4,32.37) -- (-103.5,32.37);
  \draw[red,very thick] (-103.5,32.37) -- (-64.7,32.37);
  \draw[red,very thick,dotted] (-64.7,32.37) -- (-25.41,32.37) to[quadratic={(-5.41,32.37)}] (-12.46,36.43) to[quadratic={(-19.5,40.5)}] (-39.5,40.5) to[quadratic={(-59.5,40.5)}] (-66.54,44.57) to [quadratic={(-68.38,45.63)}] (-68.38,46.41);
  \draw[red,very thick] (-68.38,46.41) to[quadratic={(-68.38,48.63)}] (-53.59,48.63) -- (26.41,48.63) to[quadratic={(46.41,48.63)}] (53.46,44.57) to[quadratic={(55.29,43.51)}] (55.29,42.73);
  \draw[red,very thick,dotted] (55.29,42.73) to[quadratic={(55.29,40.5)}] (40.5,40.5) to[quadratic={(20.5,40.5)}] (27.54,36.43) to[quadratic={(34.59,32.37)}] (54.59,32.37) -- (53,32.37);
  \draw[red,very thick] (53,32.37) -- (89.2,32.37);
  \draw[red,very thick,dotted] (89.2,32.37) -- (120.5,32.37);
  \draw[red,very thick] (120.5,32.37) -- (134.6,32.37);
  \draw[red,very thick] (106.4,48.63) to[quadratic={(91.62,48.63)}] (91.62,46.41);
  \draw[red,very thick,dotted] (91.62,46.41) to[quadratic={(91.62,45.62)}] (93.46,44.57) to[quadratic={(97.64,42.15)}] (106.4,41.17);
  \draw[red,very thick] (106.4,41.17) to[quadratic={(112.4,40.5)}] (120.5,40.5);
  \draw[black,thick,dotted] (-10.62,34.59) .. controls (-10.62,14.59)and(25.71,18.28) .. (25.71,38.28);
  \draw[black] (-104.7,42.72) .. controls (-104.7,-7.28)and(-24.71,-7.28) .. (-24.71,-37.28);
  \draw[black] (-68.38,46.41) .. controls (-68.38,-3.59)and(55.29,-7.28) .. (55.29,42.72);
  \draw[black] (91.62,46.41) .. controls (91.62,-3.59)and(11.62,-3.59) .. (11.62,-33.59);
  \draw[black] (-133.6,48.63) -- (-133.6,-31.37);
  \draw[black,thick,dotted] (-119.5,40.5) -- (-119.5,-31.37);
  \draw[black] (-119.5,-31.37) -- (-119.5,-39.5);
  \draw[black,thick,dotted] (-105.4,32.37) -- (-105.4,-39.5);
  \draw[black] (-105.4,-39.5) -- (-105.4,-47.63);
  \draw[black] (106.4,48.63) -- (106.4,-31.37);
  \draw[black] (120.5,40.5) -- (120.5,-39.5);
  \draw[black] (134.6,32.37) -- (134.6,-47.63);
  \draw[red,very thick] (-133.6,-31.37) -- (-53.59,-31.37) to[quadratic={(-33.59,-31.37)}] (-26.54,-35.43) to[quadratic={(-19.5,-39.5)}] (-39.5,-39.5) -- (-119.5,-39.5);
  \draw[red,very thick] (-105.4,-47.63) -- (134.6,-47.63);
  \draw[red,very thick] (106.4,-31.37) -- (26.41,-31.37) to[quadratic={(6.41,-31.37)}] (13.46,-35.43) to[quadratic={(20.5,-39.5)}] (40.5,-39.5) -- (120.5,-39.5);
}

\defTikzBox[xlen=.25pt,ylen=-.4pt]{BordEquivll}{%
  \draw[red,very thick]  (-133.6,48.63) -- (26.41,48.63) to[quadratic={(46.41,48.63)}] (53.46,44.57) to[quadratic={(55.29,43.51)}] (55.29,42.73);
  \draw[red,very thick,dotted] (55.29,42.73) to[quadratic={(55.29,40.5)}] (40.5,40.5) to[quadratic={(20.5,40.5)}] (27.54,36.43) to[quadratic={(34.59,32.37)}] (54.59,32.37) -- (53,32.37);
  \draw[red,very thick] (53,32.37) -- (88,32.37);
  \draw[red,very thick,dotted] (88,32.37) -- (120.5,32.37);
  \draw[red,very thick] (120.5,32.37) -- (134.6,32.37);
  \draw[red,very thick,dotted] (-119.5,40.5) -- (-39.5,40.5) to[quadratic={(-19.5,40.5)}] (-12.46,36.43) to [quadratic={(-10.62,35.37)}] (-10.62,34.59);
  \draw[red,very thick,dotted] (-10.62,34.59) to[quadratic={(-10.62,32.37)}] (-25.41,32.37) -- (-105.4,32.37);
  \draw[red,very thick] (106.4,48.63) to[quadratic={(91.62,48.63)}] (91.62,46.41);
  \draw[red,very thick,dotted] (91.62,46.41) to[quadratic={(91.62,45.62)}] (93.46,44.57) to[quadratic={(97.64,42.15)}] (106.4,41.17);
  \draw[red,very thick] (106.4,41.17) to[quadratic={(112.4,40.5)}] (120.5,40.5);
  \draw[black,thick,dotted] (-10.62,34.59) .. controls(-10.62,-0.08)and(-70.73,-4.7) .. (-86.74,-31.37);
  \draw[black] (-86.74,-31.37) .. controls(-89.2,-35.47)and(-90.62,-40.08) .. (-90.62,-45.41);
  \draw[black,thick,dotted] (25.71,38.28) .. controls(25.71,2.36)and(-38.81,-1.30) .. (-51.96,-30.65);
  \draw[black] (-51.96,-30.65) .. controls(-53.46,-33.98)and(-54.29,-37.64) .. (-54.29,-41.72);
  \draw[black] (55.29,42.72) .. controls (55.29,2.72)and(-24.71,2.72) .. (-24.71,-37.28);
  \draw[black] (91.62,46.41) .. controls (91.62,6.41)and(11.62,6.41) .. (11.62,-33.59);
  \draw[black] (-133.6,48.63) -- (-133.6,-31.37);
  \draw[black,thick,dotted] (-119.5,40.5) -- (-119.5,-31.37);
  \draw[black] (-119.5,-31.37) -- (-119.5,-39.5);
  \draw[black,thick,dotted] (-105.4,32.37) -- (-105.4,-40.17);
  \draw[black] (-105.4,-40.17) -- (-105.4,-47.63);
  \draw[black] (106.4,48.63) -- (106.4,-31.37);
  \draw[black] (120.5,40.5) -- (120.5,-39.5);
  \draw[black] (134.6,32.37) -- (134.6,-47.63);
  \draw[red,very thick] (-133.6,-31.37) -- (-53.59,-31.37) to[quadratic={(-33.59,-31.37)}] (-26.54,-35.43) to[quadratic={(-19.5,-39.5)}] (-39.5,-39.5) to[quadratic={(-59.5,-39.5)}] (-52.46,-43.57) to[quadratic={(-45.41,-47.63)}] (-25.41,-47.63) -- (134.6,-47.63);
  \draw[red,very thick] (-119.5,-39.5) to[quadratic={(-99.5,-39.5)}] (-92.46,-43.57) to[quadratic={(-85.41,-47.63)}] (-105.4,-47.63);
  \draw[red,very thick] (106.4,-31.37) -- (26.41,-31.37) to[quadratic={(6.41,-31.37)}] (13.46,-35.43) to[quadratic={(20.5,-39.5)}] (40.5,-39.5) -- (120.5,-39.5);
}

\defTikzBox[xlen=.25pt,ylen=-.4pt]{BordTriId}{%
  \draw[red,very thick] (-133.6,48.63) -- (106.4,48.63);
  \draw[red,very thick,dotted] (-119.5,40.5) -- (106.4,40.5);
  \draw[red,very thick] (106.4,40.5) -- (120.5,40.5);
  \draw[red,very thick,dotted] (-105.4,32.37) -- (120.5,32.37);
  \draw[red,very thick] (120.5,32.37) -- (134.6,32.37);
  \draw[black] (-133.6,48.63) -- (-133.6,-31.37);
  \draw[black,thick,dotted] (-119.5,40.5) -- (-119.5,-31.37);
  \draw[black] (-119.5,-31.37) -- (-119.5,-39.5);
  \draw[black,thick,dotted] (-105.4,32.37) -- (-105.4,-39.5);
  \draw[black] (-105.4,-39.5) -- (-105.4,-47.63);
  \draw[black] (106.4,48.63) -- (106.4,-31.37);
  \draw[black] (120.5,40.5) -- (120.5,-39.5);
  \draw[black] (134.6,32.37) -- (134.6,-47.63);
  \draw[red,very thick] (-133.6,-31.37) -- (106.4,-31.37);
  \draw[red,very thick] (-119.5,-39.5) -- (120.5,-39.5);
  \draw[red,very thick] (-105.4,-47.63) -- (134.6,-47.63);
}

\defTikzBox[xlen=.25pt,ylen=-.4pt]{BordRTwoLCounit}{%
  \draw[red,very thick] (-133.6,48.63) to[quadratic={(-113.6,48.63)}] (-106.5,44.57) to[quadratic={(-104.62,43.48)}] (-104.69,42.68);
  \draw[red,very thick,dotted] (-104.69,42.68) to[quadratic={(-104.86,40.5)}] (-119.5,40.5);
  \draw[red,very thick,dotted] (-105.4,32.37) -- (-103,32.37);
  \draw[red,very thick] (-103,32.37) -- (-62.6,32.37);
  \draw[red,very thick,dotted] (-62.6,32.37) -- (52.2,32.37);
  \draw[red,very thick] (52.2,32.37) -- (89.6,32.37);
  \draw[red,very thick,dotted] (89.6,32.37) -- (120.5,32.37);
  \draw[red,very thick] (120.5,32.37) -- (134.6,32.37);
  \draw[red,very thick] (-68.38,46.41) to[quadratic={(-68.38,48.63)}] (-53.59,48.63) -- (26.41,48.63) to[quadratic={(46.41,48.63)}] (53.46,44.57) to[quadratic={(55.29,43.51)}] (55.29,42.73);
  \draw[red,very thick,dotted] (55.29,42.73) to[quadratic={(55.29,40.5)}] (40.5,40.5) -- (-39.5,40.5) to[quadratic={(-59.5,40.5)}] (-66.54,44.57) to[quadratic={(-68.38,45.63)}] (-68.38,46.41);
  \draw[red,very thick]  (106.4,48.63) to[quadratic={(91.62,48.63)}] (91.62,46.41);
  \draw[red,very thick,dotted] (91.62,46.41) to[quadratic={(91.62,45.63)}] (93.46,44.57) to[quadratic={(97.64,42.15)}] (106.4,41.17);
  \draw[red,very thick]  (106.4,41.17) to[quadratic={(112.4,40.5)}] (120.5,40.5);
  \draw[black] (-68.38,46.41) .. controls (-68.38,6.41)and(55.29,2.72) .. (55.29,42.72);
  \draw[black] (-104.7,42.72) .. controls (-104.7,-27.28)and(91.62,-33.59) .. (91.62,46.41);
  \draw[black] (-133.6,48.63) -- (-133.6,-31.37);
  \draw[black,thick,dotted] (-119.5,40.5) -- (-119.5,-31.37);
  \draw[black] (-119.5,-31.37) -- (-119.5,-39.5);
  \draw[black,thick,dotted] (-105.4,32.37) -- (-105.4,-39.5);
  \draw[black] (-105.4,-39.5) -- (-105.4,-47.63);
  \draw[black] (106.4,48.63) -- (106.4,-31.37);
  \draw[black] (120.5,40.5) -- (120.5,-39.5);
  \draw[black] (134.6,32.37) -- (134.6,-47.63);
  \draw[red,very thick] (-133.6,-31.37) -- (106.4,-31.37);
  \draw[red,very thick] (-119.5,-39.5) -- (120.5,-39.5);
  \draw[red,very thick] (-105.4,-47.63) -- (134.6,-47.63);
}

\defTikzBox[xlen=.25pt,ylen=-.4pt]{BordRTwoBarLUnit}{%
\draw[red,very thick] (-133.6,48.63) -- (106.4,48.63);
\draw[red,very thick,dotted] (-119.5,40.5) -- (106.4,40.5);
\draw[red,very thick] (106.4,40.5) -- (120.5,40.5);
\draw[red,very thick,dotted]  (-105.4,32.37) -- (120.5,32.37);
\draw[red,very thick]  (120.5,32.37) -- (134.6,32.37);
\draw[black] (-104.7,-37.28) .. controls (-104.7,42.72)and(91.62,56.41) .. (91.62,-33.59);
\draw[black] (-68.38,-33.59) .. controls (-68.38,6.41)and(55.29,2.72) .. (55.29,-37.28);
\draw[black] (-133.6,48.63) -- (-133.6,-31.37);
\draw[black,thick,dotted] (-119.5,40.5) -- (-119.5,-32.03);
\draw[black] (-119.5,-32.03) -- (-119.5,-39.5);
\draw[black,thick,dotted] (-105.4,32.37) -- (-105.4,-38.14);
\draw[black] (-105.4,-38.14) -- (-105.4,-47.63);
\draw[black] (106.4,48.63) -- (106.4,-31.37);
\draw[black] (120.5,40.5) -- (120.5,-39.5);
\draw[black] (134.6,32.37) -- (134.6,-47.63);
\draw[red,very thick] (-133.6,-31.37) to[quadratic={(-113.6,-31.37)}] (-106.5,-35.43) to[quadratic={(-99.5,-39.5)}] (-119.5,-39.5);
\draw[red,very thick] (-105.4,-47.63) -- (134.6,-47.63);
\draw[red,very thick] (-66.54,-35.43) to[quadratic={(-73.59,-31.37)}] (-53.59,-31.37) -- (26.41,-31.37) to[quadratic={(46.41,-31.37)}] (53.46,-35.43) to[quadratic={(60.5,-39.5)}] (40.5,-39.5) -- (-39.5,-39.5) to[quadratic={(-59.5,-39.5)}] (-66.54,-35.43);
\draw[red,very thick] (106.4,-31.37) to[quadratic={(86.41,-31.37)}] (93.46,-35.43) to[quadratic={(100.5,-39.5)}] (120.5,-39.5);
}

\defTikzBox[xlen=.25pt,ylen=-.4pt]{BordRTwoRCounit}{%
  \draw[red,very thick] (-133.6,48.63) -- (106.4,48.63);
  \draw[red,very thick,dotted] (-119.5,40.5) to[quadratic={(-99.5,40.5)}] (-92.46,36.43) to[quadratic={(-85.41,32.37)}] (-105.4,32.37);
  \draw[red,very thick,dotted] (-52.46,36.43) to[quadratic={(-59.5,40.5)}] (-39.5,40.5) -- (40.5,40.5) to[quadratic={(60.5,40.5)}] (67.54,36.43) to[quadratic={(74.59,32.37)}] (54.59,32.37) -- (-25.41,32.37) to[quadratic={(-45.41,32.37)}] (-52.46,36.43);
  \draw[red,very thick] (120.5,40.5) to[quadratic={(108.9,40.5)}] (106.4,39.14);
  \draw[red,very thick,dotted]  (106.4,39.14) to[quadratic={(105,38.38)}] (106.4,37.2) to[quadratic={(106.8,36.84)}] (107.5,36.43) to[quadratic={(111.7,34.01)}] (120.5,33.03);
  \draw[red,very thick] (120.5,33.03) to[quadratic={(126.5,32.37)}] (134.6,32.37);
  \draw[black,thick,dotted] (-54.29,38.28) .. controls (-54.29,-1.72)and(69.38,-5.41) .. (69.38,34.59);
  \draw[black,thick,dotted] (-90.62,34.59) .. controls (-90.62,-45.41)and(105.7,-41.72) .. (105.7,38.28);
  \draw[black] (-133.6,48.63) -- (-133.6,-31.37);
  \draw[black,thick,dotted] (-119.5,40.5) -- (-119.5,-31.37);
  \draw[black] (-119.5,-31.37) -- (-119.5,-39.5);
  \draw[black,thick,dotted] (-105.4,32.37) -- (-105.4,-39.5);
  \draw[black] (-105.4,-39.5) -- (-105.4,-47.63);
  \draw[black] (106.4,48.63) -- (106.4,-31.37);
  \draw[black] (120.5,40.5) -- (120.5,-39.5);
  \draw[black] (134.6,32.37) -- (134.6,-47.63);
  \draw[red,very thick] (-133.6,-31.37) -- (106.4,-31.37);
  \draw[red,very thick] (-119.5,-39.5) -- (120.5,-39.5);
  \draw[red,very thick] (-105.4,-47.63) -- (134.6,-47.63);
}

\defTikzBox[xlen=.25pt,ylen=-.4pt]{BordRTwoBarRUnit}{%
  \draw[red,very thick] (-133.6,48.63) -- (106.4,48.63);
  \draw[red,very thick,dotted] (-119.5,40.5) -- (106.4,40.5);
  \draw[red,very thick] (106.4,40.5) -- (120.5,40.5);
  \draw[red,very thick,dotted] (-105.4,32.37) -- (120.5,32.37);
  \draw[red,very thick] (120.5,32.37) -- (134.6,32.37);
  \draw[black] (-90.62,-45.41) .. controls (-90.62,-40.45)and(-89.64,-35.77) .. (-88.44,-31.37);
  \draw[black,thick,dotted] (-88.44,-31.37) .. controls (-68.02,31.62)and(89.38,33.3) .. (104.53,-31.37);
  \draw[black] (104.53,-31.37) .. controls (105.40,-35.64)and(105.7,-38.61) .. (105.7,-41.72);
  \draw[black] (-54.29,-41.72) .. controls (-54.29,-37.87)and(-53.14,-34.42) .. (-51.07,-31.37);
  \draw[black,thick,dotted] (-51.07,-31.37) .. controls (-34.48,-6.98)and(41.36,-8.11) .. (63.38,-31.37);
  \draw[black] (63.38,-31.37) ..controls (67.18,-35.39)and(69.38,-40.08) .. (69.38,-45.41);
  \draw[black] (-133.6,48.63) -- (-133.6,-31.37);
  \draw[black,thick,dotted] (-119.5,40.5) -- (-119.5,-31.37);
  \draw[black] (-119.5,-31.37) -- (-119.5,-39.5);
  \draw[black,thick,dotted] (-105.4,32.37) -- (-105.4,-40.17);
  \draw[black] (-105.4,-40.17) -- (-105.4,-47.63);
  \draw[black] (106.4,48.63) -- (106.4,-31.37);
  \draw[black] (120.5,40.5) -- (120.5,-39.5);
  \draw[black] (134.6,32.37) -- (134.6,-47.63);
  \draw[red,very thick] (-133.6,-31.37) -- (106.4,-31.37);
  \draw[red,very thick] (-119.5,-39.5) to[quadratic={(-99.5,-39.5)}] (-92.46,-43.57) to[quadratic={(-85.41,-47.63)}] (-105.4,-47.63);
  \draw[red,very thick] (-52.46,-43.57) to[quadratic={(-59.5,-39.5)}] (-39.5,-39.5) -- (40.5,-39.5) to[quadratic={(60.5,-39.5)}] (67.54,-43.57) to[quadratic={(74.59,-47.63)}] (54.59,-47.63) -- (-25.41,-47.63) to[quadratic={(-45.41,-47.63)}] (-52.46,-43.57);
  \draw[red,very thick] (120.5,-39.5) to[quadratic={(100.5,-39.5)}] (107.5,-43.57) to[quadratic={(114.6,-47.63)}] (134.6,-47.63);
}

\defTikzBox[xlen=.25pt,ylen=-.4pt]{BordRTwoLCounitRUnit}{%
  \draw[red,very thick] (-133.6,48.63) to[quadratic={(-113.6,48.63)}] (-106.5,44.57) to[quadratic={(-104.62,43.48)}] (-104.69,42.68);
  \draw[red,very thick,dotted] (-104.69,42.68) to[quadratic={(-104.86,40.5)}] (-119.5,40.5);
  \draw[red,very thick,dotted] (-105.4,32.37) -- (-103,32.37);
  \draw[red,very thick] (-103,32.37) -- (-62.6,32.37);
  \draw[red,very thick,dotted] (-62.6,32.37) -- (52.2,32.37);
  \draw[red,very thick] (52.2,32.37) -- (89.6,32.37);
  \draw[red,very thick,dotted] (89.6,32.37) -- (120.5,32.37);
  \draw[red,very thick] (120.5,32.37) -- (134.6,32.37);
  \draw[red,very thick] (-68.38,46.41) to[quadratic={(-68.38,48.63)}] (-53.59,48.63) -- (26.41,48.63) to[quadratic={(46.41,48.63)}] (53.46,44.57) to[quadratic={(55.29,43.51)}] (55.29,42.73);
  \draw[red,very thick,dotted] (55.29,42.73) to[quadratic={(55.29,40.5)}] (40.5,40.5) -- (-39.5,40.5) to[quadratic={(-59.5,40.5)}] (-66.54,44.57) to[quadratic={(-68.38,45.63)}] (-68.38,46.41);
  \draw[red,very thick]  (106.4,48.63) to[quadratic={(91.62,48.63)}] (91.62,46.41);
  \draw[red,very thick,dotted] (91.62,46.41) to[quadratic={(91.62,45.63)}] (93.46,44.57) to[quadratic={(97.64,42.15)}] (106.4,41.17);
  \draw[red,very thick]  (106.4,41.17) to[quadratic={(112.4,40.5)}] (120.5,40.5);
  \draw[black] (-68.38,46.41) .. controls+(0,-35)and+(0,-35) .. (55.29,42.72);
  \draw[black] (-104.7,42.72) .. controls+(0,-50)and+(0,-50) .. (91.62,46.41);
  \draw[black] (-90.62,-45.41) .. controls(-90.62,-40.23)and(-88.51,-35.54) .. (-84.74,-31.36);
  \draw[black,thick,dotted] (-84.74,-31.36) .. controls(-54.79,1.83)and(80.17,3.46) .. (102.55,-31.29);
  \draw[black] (102.55,-31.29) .. controls(104.6,-34.46)and(105.7,-37.94) .. (105.7,-41.72);
  \draw[black] (-54.29,-41.72) .. controls(-54.29,-37.78)and(-52.73,-34.34) .. (-49.95,-31.37);
  \draw[black,thick,dotted] (-49.95,-31.37) .. controls(-31.88,-12.05)and(37.53,-13.02) .. (61.32,-31.37);
  \draw[black] (61.32,-31.37) .. controls(66.38,-35.27)and(69.38,-39.96) .. (69.38,-45.41);
  \draw[black] (-133.6,48.63) -- (-133.6,-31.37);
  \draw[black,thick,dotted] (-119.5,40.5) -- (-119.5,-31.37);
  \draw[black] (-119.5,-31.37) -- (-119.5,-39.5);
  \draw[black,thick,dotted] (-105.4,32.37) -- (-105.4,-40.17);
  \draw[black] (-105.4,-40.17) -- (-105.4,-47.63);
  \draw[black] (106.4,48.63) -- (106.4,-31.37);
  \draw[black] (120.5,40.5) -- (120.5,-39.5);
  \draw[black] (134.6,32.37) -- (134.6,-47.63);
  \draw[red,very thick] (-133.6,-31.37) -- (106.4,-31.37);
  \draw[red,very thick] (-119.5,-39.5) to[quadratic={(-99.5,-39.5)}] (-92.46,-43.57) to[quadratic={(-85.41,-47.63)}] (-105.4,-47.63);
  \draw[red,very thick] (-52.46,-43.57) to[quadratic={(-59.5,-39.5)}] (-39.5,-39.5) -- (40.5,-39.5) to[quadratic={(60.5,-39.5)}] (67.54,-43.57) to[quadratic={(74.59,-47.63)}] (54.59,-47.63) -- (-25.41,-47.63) to[quadratic={(-45.41,-47.63)}] (-52.46,-43.57);
  \draw[red,very thick] (120.5,-39.5) to[quadratic={(100.5,-39.5)}] (107.5,-43.57) to[quadratic={(114.6,-47.63)}] (134.6,-47.63);
}

\defTikzBox[xlen=.25pt,ylen=-.4pt]{BordRightDeltaPll}{%
  \draw[red,very thick] (-133.6,28.63) -- (26.41,28.63) to[quadratic={(46.41,28.63)}] (53.46,24.57) to[quadratic={(55.29,23.51)}] (55.29,22.73);
  \draw[red,very thick,dotted] (55.29,22.73) to[quadratic={(55.29,20.5)}] (40.5,20.5) -- (-119.5,20.5);
  \draw[red,very thick,dotted] (-105.4,12.37) -- (55.29,12.37);
  \draw[red,very thick] (55.29,12.37) -- (91.62,12.37);
  \draw[red,very thick,dotted] (91.62,12.37) -- (120.5,12.37);
  \draw[red,very thick] (120.5,12.37) -- (134.6,12.37);
  \draw[red,very thick] (106.4,28.63) to[quadratic={(91.62,28.63)}] (91.62,26.41);
  \draw[red,very thick,dotted] (91.62,26.41) to[quadratic={(91.62,25.62)}] (93.46,24.57) to[quadratic={(97.64,22.15)}] (106.4,21.17);
  \draw[red,very thick] (106.4,21.17) to[quadratic={(112.4,20.5)}] (120.5,20.5);
  \draw[black] (-10.62,-25.41) .. controls (-10.62,-19.62)and(-8.907,-14.91) .. (-6.225,-11.37);
  \draw[black,thick,dotted] (-6.225,-11.37) .. controls (1.502,-1.144)and(17.27,-0.5424) .. (23.3,-11.37);
  \draw[black] (23.3,-11.37) .. controls (24.81,-14.09)and(25.71,-17.53) .. (25.71,-21.72);
  \draw[black] (55.29,22.72) -- (55.29,-17.28);
  \draw[black] (91.62,26.41) -- (91.62,-13.59);
  \draw[black] (-133.6,28.63) -- (-133.6,-11.37);
  \draw[black,thick,dotted] (-119.5,20.5) -- (-119.5,-11.37);
  \draw[black] (-119.5,-11.37) -- (-119.5,-19.5);
  \draw[black,thick,dotted] (-105.4,12.37) -- (-105.4,-19.5);
  \draw[black] (-105.4,-19.5) -- (-105.4,-27.63);
  \draw[black] (106.4,28.63) -- (106.4,-11.37);
  \draw[black] (120.5,20.5) -- (120.5,-19.5);
  \draw[black] (134.6,12.37) -- (134.6,-27.63);
  \draw[red,very thick] (-133.6,-11.37) -- (26.41,-11.37) to[quadratic={(46.41,-11.37)}] (53.46,-15.43) to[quadratic={(60.5,-19.5)}] (40.5,-19.5) to[quadratic={(20.5,-19.5)}] (27.54,-23.57) to[quadratic={(34.59,-27.63)}] (54.59,-27.63) -- (134.6,-27.63);
  \draw[red,very thick] (-119.5,-19.5) -- (-39.5,-19.5) to[quadratic={(-19.5,-19.5)}] (-12.46,-23.57) to[quadratic={(-5.41,-27.63)}] (-25.41,-27.63) -- (-105.4,-27.63);
  \draw[red,very thick] (106.4,-11.37) to[quadratic={(86.41,-11.37)}] (93.46,-15.43) to[quadratic={(100.5,-19.5)}] (120.5,-19.5);
}

\defTikzBox[xlen=.25pt,ylen=-.4pt]{BordRightPhiFstOl}{%
  \draw[red,very thick] (-133.6,48.63) -- (106.4,48.63);
  \draw[red,very thick,dotted] (-119.5,40.5) -- (106.4,40.5);
  \draw[red,very thick] (106.4,40.5) -- (120.5,40.5);
  \draw[red,very thick,dotted] (-105.4,32.37) -- (120.5,32.37);
  \draw[red,very thick] (120.5,32.37) -- (134.6,32.37);
  \draw[black,thick,dotted] (-10.62,-5.41) .. controls (-10.62,21.26)and(25.71,24.94) .. (25.71,-1.724) .. controls (25.71,-28.39)and(-10.62,-32.08) .. (-10.62,-5.41);
  \draw[black] (-133.6,48.63) -- (-133.6,-31.37);
  \draw[black,thick,dotted] (-119.5,40.5) -- (-119.5,-31.37);
  \draw[black] (-119.5,-31.37) -- (-119.5,-39.5);
  \draw[black,thick,dotted] (-105.4,32.37) -- (-105.4,-39.5);
  \draw[black] (-105.4,-39.5) -- (-105.4,-47.63);
  \draw[black] (106.4,48.63) -- (106.4,-31.37);
  \draw[black] (120.5,40.5) -- (120.5,-39.5);
  \draw[black] (134.6,32.37) -- (134.6,-47.63);
  \draw[red,very thick] (-133.6,-31.37) -- (106.4,-31.37);
  \draw[red,very thick] (-119.5,-39.5) -- (120.5,-39.5);
  \draw[red,very thick] (-105.4,-47.63) -- (134.6,-47.63);
}

\defTikzBox[xlen=.25pt,ylen=-.4pt]{BordRightPhiSndOl}{%
  \draw[red,very thick] (-133.6,48.63) -- (106.4,48.63);
  \draw[red,very thick,dotted] (-119.5,40.5) -- (106.4,40.5);
  \draw[red,very thick] (106.4,40.5) -- (120.5,40.5);
  \draw[red,very thick,dotted] (-105.4,32.37) -- (120.5,32.37);
  \draw[red,very thick] (120.5,32.37) -- (134.6,32.37);
  \draw[black,thick,dotted] (69.38,28.68) .. controls(69.38,23.84)and(85.89,20.17) .. (106.4,15.94);
  \draw[black] (106.4,15.94) .. controls(121.17,11.48)and(140,6.4) .. (140,-1.35) .. controls(140,-8.73)and(122.92,-13.69) .. (106.4,-18.4);
  \draw[black,thick,dotted] (106.4,-18.4) .. controls(87.61,-22.49)and(69.38,-26.28) .. (69.38,-31.37);
  \draw[black,thick,dotted] (69.38,3.65) .. controls+(0,10)and+(0,5) .. (100,-1.35) .. controls+(0,-5)and+(0,-10) .. (69.38,-6.35);
  \draw[black] (-133.6,48.63) -- (-133.6,-31.37);
  \draw[black,thick,dotted] (-119.5,40.5) -- (-119.5,-31.37);
  \draw[black] (-119.5,-31.37) -- (-119.5,-39.5);
  \draw[black,thick,dotted] (-105.4,32.37) -- (-105.4,-39.5);
  \draw[black] (-105.4,-39.5) -- (-105.4,-47.63);
  \draw[black] (106.4,48.63) -- (106.4,-31.37);
  \draw[black] (120.5,40.5) -- (120.5,11.3);
  \draw[black,thick,dotted] (120.5,11.3) -- (120.5,-14.4);
  \draw[black] (120.5,-14.4) -- (120.5,-39.5);
  \draw[black] (134.6,32.37) -- (134.6,5.7);
  \draw[black,thick,dotted] (134.6,5.7) -- (134.6,-8);
  \draw[black] (134.6,-8) -- (134.6,-47.63);
  \draw[red,very thick] (-133.6,-31.37) -- (106.4,-31.37);
  \draw[red,very thick] (-119.5,-39.5) -- (120.5,-39.5);
  \draw[red,very thick] (-105.4,-47.63) -- (134.6,-47.63);
}

\defTikzBox[xlen=.25pt,ylen=-.4pt]{BordLeftDeltaNOO}{%
  \draw[red,very thick]  (-133.6,28.63) -- (-53.59,28.63) to[quadratic={(-33.59,28.63)}] (-26.54,24.57) to[quadratic={(-24.71,23.51)}] (-24.71,22.72);
  \draw[red,very thick,dotted] (-24.71,22.72) to[quadratic={(-24.71,20.5)}] (-39.5,20.5) -- (-119.5,20.5);
  \draw[red,very thick,dotted] (-105.4,12.37) -- (-22.3,12.37);
  \draw[red,very thick] (-22.3,12.37) -- (7.225,12.37);
  \draw[red,very thick,dotted] (7.225,12.37) -- (54.59,12.37) to[quadratic={(74.59,12.37)}] (67.54,16.43) to[quadratic={(60.5,20.5)}] (40.5,20.5) to[quadratic={(20.5,20.5)}] (13.46,24.57) to[quadratic={(11.62,25.63)}] (11.62,26.41);
  \draw[red,very thick] (11.62,26.41) to[quadratic={(11.62,28.63)}] (26.41,28.63) -- (106.4,28.63);
  \draw[red,very thick] (120.5,20.5) to[quadratic={(108.9,20.5)}] (106.4,19.14);
  \draw[red,very thick,dotted] (106.4,19.14) to[quadratic={(105,18.38)}] (106.4,17.2) to[quadratic={(106.8,16.84)}] (107.5,16.43) to[quadratic={(111.7,14.01)}] (120.5,13.03);
  \draw[red,very thick] (120.5,13.03) to[quadratic={(126.5,12.37)}] (134.6,12.37);
  \draw[black] (-24.71,22.72) .. controls (-24.71,-3.942)and(11.62,-0.2563) .. (11.62,26.41);
  \draw[black,thick,dotted] (69.38,14.59) -- (69.38,-11.37);
  \draw[black] (69.38,-11.37) -- (69.38,-25.41);
  \draw[black,thick,dotted] (105.7,18.28) -- (105.7,-11.37);
  \draw[black] (105.7,-11.37) -- (105.7,-21.72);
  \draw[black] (-133.6,28.63) -- (-133.6,-11.37);
  \draw[black,thick,dotted] (-119.5,20.5) -- (-119.5,-11.37);
  \draw[black] (-119.5,-11.37) -- (-119.5,-19.5);
  \draw[black,thick,dotted] (-105.4,12.37) -- (-105.4,-19.5);
  \draw[black] (-105.4,-19.5) -- (-105.4,-27.63);
  \draw[black] (106.4,28.63) -- (106.4,-11.37);
  \draw[black] (120.5,20.5) -- (120.5,-19.5);
  \draw[black] (134.6,12.37) -- (134.6,-27.63);
  \draw[red,very thick] (-133.6,-11.37) -- (106.4,-11.37);
  \draw[red,very thick] (-119.5,-19.5) -- (40.5,-19.5) to[quadratic={(60.5,-19.5)}] (67.54,-23.57) to[quadratic={(74.59,-27.63)}] (54.59,-27.63) -- (-105.4,-27.63);
  \draw[red,very thick] (120.5,-19.5) to[quadratic={(100.5,-19.5)}] (107.5,-23.57) to[quadratic={(114.6,-27.63)}] (134.6,-27.63);
}

\defTikzBox[xlen=.25pt,ylen=-.4pt]{BordLeftPhiFstlO}{%
  \draw[red,very thick] (-133.6,48.63) -- (106.4,48.63);
  \draw[red,very thick,dotted] (-119.5,40.5) -- (106.4,40.5);
  \draw[red,very thick] (106.4,40.5) -- (120.5,40.5);
  \draw[red,very thick,dotted] (-105.4,32.37) -- (120.5,32.37);
  \draw[red,very thick] (120.5,32.37) -- (134.6,32.37);
  \draw[black] (-24.71,2.724) .. controls (-24.71,29.39)and(11.62,33.08) .. (11.62,6.41) .. controls (11.62,-20.26)and(-24.71,-23.94) .. (-24.71,2.724);
  \draw[black] (-133.6,48.63) -- (-133.6,-31.37);
  \draw[black,thick,dotted] (-119.5,40.5) -- (-119.5,-31.37);
  \draw[black] (-119.5,-31.37) -- (-119.5,-39.5);
  \draw[black,thick,dotted] (-105.4,32.37) -- (-105.4,-39.5);
  \draw[black] (-105.4,-39.5) -- (-105.4,-47.63);
  \draw[black] (106.4,48.63) -- (106.4,-31.37);
  \draw[black] (120.5,40.5) -- (120.5,-39.5);
  \draw[black] (134.6,32.37) -- (134.6,-47.63);
  \draw[red,very thick] (-133.6,-31.37) -- (106.4,-31.37);
  \draw[red,very thick] (-119.5,-39.5) -- (120.5,-39.5);
  \draw[red,very thick] (-105.4,-47.63) -- (134.6,-47.63);
}

\defTikzBox[xlen=.25pt,ylen=-.4pt]{BordLeftPhiSndlO}{%
  \draw[red,very thick] (-133.6,48.63) -- (106.4,48.63);
  \draw[red,very thick,dotted] (-119.5,40.5) -- (106.4,40.5);
  \draw[red,very thick] (106.4,40.5) -- (120.5,40.5);
  \draw[red,very thick,dotted] (-105.4,32.37) -- (120.5,32.37);
  \draw[red,very thick] (120.5,32.37) -- (134.6,32.37);
  \draw[black] (10,32.57) .. controls+(0,-10)and+(0,15) .. (46,4.57) .. controls+(0,-15)and+(0,10) .. (10,-23.43);
  \draw[black] (10,10.57) .. controls+(0,5)and+(0,7) .. (25,4.57) .. controls+(0,-7)and+(0,-5) .. (10,-1.43);
  \draw[black] (-133.6,48.63) -- (-133.6,-31.37);
  \draw[black,thick,dotted] (-119.5,40.5) -- (-119.5,-31.37);
  \draw[black] (-119.5,-31.37) -- (-119.5,-39.5);
  \draw[black,thick,dotted] (-105.4,32.37) -- (-105.4,-39.5);
  \draw[black] (-105.4,-39.5) -- (-105.4,-47.63);
  \draw[black] (106.4,48.63) -- (106.4,-31.37);
  \draw[black] (120.5,40.5) -- (120.5,-39.5);
  \draw[black] (134.6,32.37) -- (134.6,-47.63);
  \draw[red,very thick] (-133.6,-31.37) -- (106.4,-31.37);
  \draw[red,very thick] (-119.5,-39.5) -- (120.5,-39.5);
  \draw[red,very thick] (-105.4,-47.63) -- (134.6,-47.63);
}

\defTikzBox[xlen=.25pt,ylen=-.4pt]{BordCompFFst}{%
  \draw[red,very thick]  (-133.6,48.63) -- (26.41,48.63) to[quadratic={(46.41,48.63)}] (53.46,44.57) to[quadratic={(55.29,43.51)}] (55.29,42.73);
  \draw[red,very thick,dotted] (55.29,42.73) to[quadratic={(55.29,40.5)}] (40.5,40.5) to[quadratic={(20.5,40.5)}] (27.54,36.43) to[quadratic={(34.59,32.37)}] (54.59,32.37) -- (56.1,32.37);
  \draw[red,very thick] (56.1,32.37) -- (89,32.37);
  \draw[red,very thick,dotted] (89,32.37) -- (120.5,32.37);
  \draw[red,very thick] (120.5,32.37) -- (134.6,32.37);
  \draw[red,very thick,dotted] (-119.5,40.5) -- (-39.5,40.5) to[quadratic={(-19.5,40.5)}] (-12.46,36.43) to [quadratic={(-10.62,35.37)}] (-10.62,34.59);
  \draw[red,very thick,dotted] (-10.62,34.59) to[quadratic={(-10.62,32.37)}] (-25.41,32.37) -- (-105.4,32.37);
  \draw[red,very thick] (106.4,48.63) to[quadratic={(91.62,48.63)}] (91.62,46.41);
  \draw[red,very thick,dotted] (91.62,46.41) to[quadratic={(91.62,45.62)}] (93.46,44.57) to[quadratic={(97.64,42.15)}] (106.4,41.17);
  \draw[red,very thick] (106.4,41.17) to[quadratic={(112.4,40.5)}] (120.5,40.5);
  \draw[black,thick,dotted] (-10.62,34.59) .. controls(-10.62,-0.08)and(-70.73,-4.7) .. (-86.74,-31.37);
  \draw[black] (-86.74,-31.37) .. controls(-89.2,-35.47)and(-90.62,-40.08) .. (-90.62,-45.41);
  \draw[black,thick,dotted] (25.71,38.28) .. controls(25.71,2.07)and(91.28,-1.36) .. (103.69,-31.37);
  \draw[black] (103.69,-31.37) .. controls(104.98,-34.51)and(105.7,-37.93) .. (105.7,-41.72);
  \draw[black] (55.29,42.72) .. controls+(0,-35)and+(0,-35) .. (91.62,46.41);
  \draw[black] (-54.29,-41.72) .. controls (-54.29,-37.87)and(-53.14,-34.42) .. (-51.07,-31.37);
  \draw[black,thick,dotted] (-51.07,-31.37) .. controls (-34.48,-6.98)and(41.36,-8.11) .. (63.38,-31.37);
  \draw[black] (63.38,-31.37) ..controls (67.18,-35.39)and(69.38,-40.08) .. (69.38,-45.41);
  \draw[black] (-133.6,48.63) -- (-133.6,-31.37);
  \draw[black,thick,dotted] (-119.5,40.5) -- (-119.5,-31.37);
  \draw[black] (-119.5,-31.37) -- (-119.5,-39.5);
  \draw[black,thick,dotted] (-105.4,32.37) -- (-105.4,-40.17);
  \draw[black] (-105.4,-40.17) -- (-105.4,-47.63);
  \draw[black] (106.4,48.63) -- (106.4,-31.37);
  \draw[black] (120.5,40.5) -- (120.5,-39.5);
  \draw[black] (134.6,32.37) -- (134.6,-47.63);
  \draw[red,very thick] (-133.6,-31.37) -- (106.4,-31.37);
  \draw[red,very thick] (-119.5,-39.5) to[quadratic={(-99.5,-39.5)}] (-92.46,-43.57) to[quadratic={(-85.41,-47.63)}] (-105.4,-47.63);
  \draw[red,very thick] (-52.46,-43.57) to[quadratic={(-59.5,-39.5)}] (-39.5,-39.5) -- (40.5,-39.5) to[quadratic={(60.5,-39.5)}] (67.54,-43.57) to[quadratic={(74.59,-47.63)}] (54.59,-47.63) -- (-25.41,-47.63) to[quadratic={(-45.41,-47.63)}] (-52.46,-43.57);
  \draw[red,very thick] (120.5,-39.5) to[quadratic={(100.5,-39.5)}] (107.5,-43.57) to[quadratic={(114.6,-47.63)}] (134.6,-47.63);
}

\defTikzBox[xlen=.25pt,ylen=-.4pt]{BordCompFSnd}{%
  \draw[red,very thick] (-133.6,48.63) -- (106.4,48.63);
  \draw[red,very thick,dotted]  (-119.5,40.5) -- (-39.5,40.5) to[quadratic={(-19.5,40.5)}] (-12.46,36.43) to[quadratic={(-5.41,32.37)}] (-25.41,32.37) -- (-105.4,32.37);
  \draw[red,very thick] (120.5,40.5) -- (106.4,40.5);
  \draw[red,very thick,dotted] (106.4,40.5) -- (40.5,40.5) to[quadratic={(25.71,40.5)}] (25.71,38.28) to[quadratic={(25.71,37.49)}] (27.54,36.43) to[quadratic={(34.59,32.37)}] (54.59,32.37) -- (120.5,32.37);
  \draw[red,very thick] (120.5,32.37) -- (134.6,32.37);
  \draw[black,thick,dotted] (-10.62,34.59) .. controls(-10.62,-0.08)and(49.49,-4.7) .. (65.5,-31.37);
  \draw[black] (65.5,-31.37) .. controls(67.96,-35.47)and(69.38,-40.08) .. (69.38,-45.41);
  \draw[black,thick,dotted] (25.71,38.28) .. controls(25.71,2.07)and(91.28,-1.36) .. (103.69,-31.37);
  \draw[black] (103.69,-31.37) .. controls(104.98,-34.51)and(105.7,-37.93) .. (105.7,-41.72);
  \draw[black] (-133.6,48.63) -- (-133.6,-31.37);
  \draw[black,thick,dotted] (-119.5,40.5) -- (-119.5,-31.37);
  \draw[black] (-119.5,-31.37) -- (-119.5,-39.5);
  \draw[black,thick,dotted] (-105.4,32.37) -- (-105.4,-39.5);
  \draw[black] (-105.4,-39.5) -- (-105.4,-47.63);
  \draw[black] (106.4,48.63) -- (106.4,-31.37);
  \draw[black] (120.5,40.5) -- (120.5,-39.5);
  \draw[black] (134.6,32.37) -- (134.6,-47.63);
  \draw[red,very thick] (-133.6,-31.37) -- (106.4,-31.37);
  \draw[red,very thick] (-119.5,-39.5) -- (40.5,-39.5) to[quadratic={(60.5,-39.5)}] (67.54,-43.57) to[quadratic={(74.59,-47.63)}] (54.59,-47.63) -- (-105.4,-47.63);
  \draw[red,very thick] (120.5,-39.5) to[quadratic={(100.5,-39.5)}] (107.5,-43.57) to[quadratic={(114.6,-47.63)}] (134.6,-47.63);
}

\defTikzBox[xlen=.25pt,ylen=-.4pt]{BordCompFFstVar}{%
  \draw[red,very thick] (-133.6,48.63) to[quadratic={(-113.6,48.63)}] (-106.5,44.57) to[quadratic={(-104.69,43.51)}] (-104.69,42.73);
  \draw[red,very thick,dotted] (-104.69,42.73) to[quadratic={(-104.69,40.5)}] (-119.5,40.5);
  \draw[red,very thick]  (-105.4,32.37) -- (-71,32.37);
  \draw[red,very thick,dotted] (-71,32.37) -- (-25.41,32.37) to[quadratic={(-5.41,32.37)}] (-12.46,36.43) to[quadratic={(-19.5,40.5)}] (-39.5,40.5) to[quadratic={(-59.5,40.5)}] (-66.54,44.57) to[quadratic={(-68.38,45.63)}] (-68.38,46.41);
  \draw[red,very thick] (-68.38,46.41) to[quadratic={(-68.38,48.63)}] (-53.59,48.63) -- (106.4,48.63);
  \draw[red,very thick] (120.5,40.5) -- (106.4,40.5);
  \draw[red,very thick,dotted](106.4,40.5) -- (40.5,40.5) to[quadratic={(25.71,40.5)}] (25.71,38.28) to[quadratic={(25.71,37.49)}] (27.54,36.43) to[quadratic={(34.59,32.37)}] (54.59,32.37) -- (120.5,32.37);
  \draw[red,very thick]  (120.5,32.37) -- (134.6,32.37);
  \draw[black,thick,dotted] (-10.62,34.59) .. controls(-10.62,-0.08)and(-70.73,-4.7) .. (-86.74,-31.37);
  \draw[black] (-86.74,-31.37) .. controls(-89.2,-35.47)and(-90.62,-40.08) .. (-90.62,-45.41);
  \draw[black,thick,dotted] (25.71,38.28) .. controls(25.71,2.07)and(91.28,-1.36) .. (103.69,-31.37);
  \draw[black] (103.69,-31.37) .. controls(104.98,-34.51)and(105.7,-37.93) .. (105.7,-41.72);
  \draw[black] (-104.71,42.72) .. controls+(0,-35)and+(0,-35) .. (-68.38,46.41);
  \draw[black] (-54.29,-41.72) .. controls (-54.29,-37.87)and(-53.14,-34.42) .. (-51.07,-31.37);
  \draw[black,thick,dotted] (-51.07,-31.37) .. controls (-34.48,-6.98)and(41.36,-8.11) .. (63.38,-31.37);
  \draw[black] (63.38,-31.37) ..controls (67.18,-35.39)and(69.38,-40.08) .. (69.38,-45.41);
  \draw[black] (-133.6,48.63) -- (-133.6,-31.37);
  \draw[black,thick,dotted] (-119.5,40.5) -- (-119.5,-31.37);
  \draw[black] (-119.5,-31.37) -- (-119.5,-39.5);
  \draw[black,thick,dotted] (-105.4,32.37) -- (-105.4,-40.17);
  \draw[black] (-105.4,-40.17) -- (-105.4,-47.63);
  \draw[black] (106.4,48.63) -- (106.4,-31.37);
  \draw[black] (120.5,40.5) -- (120.5,-39.5);
  \draw[black] (134.6,32.37) -- (134.6,-47.63);
  \draw[red,very thick] (-133.6,-31.37) -- (106.4,-31.37);
  \draw[red,very thick] (-119.5,-39.5) to[quadratic={(-99.5,-39.5)}] (-92.46,-43.57) to[quadratic={(-85.41,-47.63)}] (-105.4,-47.63);
  \draw[red,very thick] (-52.46,-43.57) to[quadratic={(-59.5,-39.5)}] (-39.5,-39.5) -- (40.5,-39.5) to[quadratic={(60.5,-39.5)}] (67.54,-43.57) to[quadratic={(74.59,-47.63)}] (54.59,-47.63) -- (-25.41,-47.63) to[quadratic={(-45.41,-47.63)}] (-52.46,-43.57);
  \draw[red,very thick] (120.5,-39.5) to[quadratic={(100.5,-39.5)}] (107.5,-43.57) to[quadratic={(114.6,-47.63)}] (134.6,-47.63);
}

\defTikzBox[xlen=.25pt,ylen=-.4pt]{BordCompFSndVar}{%
  \draw[red,very thick] (-133.6,48.63) -- (106.4,48.63);
  \draw[red,very thick,dotted]  (-119.5,40.5) -- (-39.5,40.5) to[quadratic={(-19.5,40.5)}] (-12.46,36.43) to[quadratic={(-5.41,32.37)}] (-25.41,32.37) -- (-105.4,32.37);
  \draw[red,very thick] (120.5,40.5) -- (106.4,40.5);
  \draw[red,very thick,dotted] (106.4,40.5) -- (40.5,40.5) to[quadratic={(25.71,40.5)}] (25.71,38.28) to[quadratic={(25.71,37.49)}] (27.54,36.43) to[quadratic={(34.59,32.37)}] (54.59,32.37) -- (120.5,32.37);
  \draw[red,very thick] (120.5,32.37) -- (134.6,32.37);
  \draw[black,thick,dotted] (-10.62,34.59) .. controls(-10.62,-0.08)and(-70.73,-4.7) .. (-86.74,-31.37);
  \draw[black] (-86.74,-31.37) .. controls(-89.2,-35.47)and(-90.62,-40.08) .. (-90.62,-45.41);
  \draw[black,thick,dotted] (25.71,38.28) ..controls(25.71,2.07)and(-39.86,-1.36).. (-52.28,-31.37);
  \draw[black] (-52.28,-31.37) ..controls(-53.57,-34.51)and(-54.29,-37.93).. (-54.29,-41.72);
  \draw[black] (-133.6,48.63) -- (-133.6,-31.37);
  \draw[black,thick,dotted] (-119.5,40.5) -- (-119.5,-31.37);
  \draw[black] (-119.5,-31.37) -- (-119.5,-39.5);
  \draw[black,thick,dotted] (-105.4,32.37) -- (-105.4,-39.5);
  \draw[black] (-105.4,-39.5) -- (-105.4,-47.63);
  \draw[black] (106.4,48.63) -- (106.4,-31.37);
  \draw[black] (120.5,40.5) -- (120.5,-39.5);
  \draw[black] (134.6,32.37) -- (134.6,-47.63);
  \draw[red,very thick] (-133.6,-31.37) -- (106.4,-31.37);
  \draw[red,very thick] (-119.5,-39.5) to[quadratic={(-99.5,-39.5)}] (-92.46,-43.57) to[quadratic={(-85.41,-47.63)}] (-105.4,-47.63);
  \draw[red,very thick] (120.5,-39.5) -- (-39.5,-39.5) to[quadratic={(-59.5,-39.5)}] (-52.46,-43.57) to[quadratic={(-45.41,-47.63)}] (-25.41,-47.63) -- (134.6,-47.63);
}

\defTikzBox[xlen=.25pt,ylen=-.4pt]{BordCompGFst}{%
  \draw[red,very thick] (-133.6,48.63) -- (26.41,48.63) to[quadratic={(46.41,48.63)}] (53.46,44.57) to[quadratic={(55.29,43.51)}] (55.29,42.73);
  \draw[red,very thick,dotted] (55.29,42.73) to[quadratic={(55.29,40.5)}] (40.5,40.5) -- (-119.5,40.5);
  \draw[red,very thick,dotted] (-105.4,32.37) -- (55.7,32.37);
  \draw[red,very thick] (55.7,32.37) -- (90,32.37);
  \draw[red,very thick,dotted] (90,32.37) -- (120.5,32.37);
  \draw[red,very thick] (120.5,32.37) -- (134.6,32.37);
  \draw[red,very thick] (106.4,48.63) to[quadratic={(91.62,48.63)}] (91.62,46.41);
  \draw[red,very thick,dotted] (91.62,46.41) to[quadratic={(91.62,45.63)}] (93.46,44.57) to[quadratic={(97.64,42.15)}] (106.4,41.17);
  \draw[red,very thick] (106.4,41.17) to[quadratic={(112.4,40.5)}] (120.5,40.5);
  \draw[black] (55.29,42.72) .. controls+(0,-40)and+(0,-40) .. (91.62,46.41);
  \draw[black] (-133.6,48.63) -- (-133.6,-31.37);
  \draw[black,thick,dotted] (-119.5,40.5) -- (-119.5,-32.03);
  \draw[black] (-119.5,-32.03) -- (-119.5,-39.5);
  \draw[black,thick,dotted] (-105.4,32.37) -- (-105.4,-38.14);
  \draw[black] (-105.4,-38.14) -- (-105.4,-47.63);
  \draw[black] (106.4,48.63) -- (106.4,-31.37);
  \draw[black] (120.5,40.5) -- (120.5,-39.5);
  \draw[black] (134.6,32.37) -- (134.6,-47.63);
  \draw[red,very thick] (-133.6,-31.37) -- (106.4,-31.37);
  \draw[red,very thick] (-119.5,-39.5) -- (120.5,-39.5);
  \draw[red,very thick] (-105.4,-47.63) -- (134.6,-47.63);
}

\defTikzBox[xlen=.25pt,ylen=-.4pt]{BordCompGSnd}{%
  \draw[red,very thick] (-133.6,48.63) -- (26.41,48.63) to[quadratic={(46.41,48.63)}] (53.46,44.57) to[quadratic={(55.29,43.51)}] (55.29,42.73);
  \draw[red,very thick,dotted] (55.29,42.73) to[quadratic={(55.29,40.5)}] (40.5,40.5) -- (-119.5,40.5);
  \draw[red,very thick,dotted] (-105.4,32.37) -- (55.7,32.37);
  \draw[red,very thick] (55.7,32.37) -- (90,32.37);
  \draw[red,very thick,dotted] (90,32.37) -- (120.5,32.37);
  \draw[red,very thick] (120.5,32.37) -- (134.6,32.37);
  \draw[red,very thick] (106.4,48.63) to[quadratic={(91.62,48.63)}] (91.62,46.41);
  \draw[red,very thick,dotted] (91.62,46.41) to[quadratic={(91.62,45.63)}] (93.46,44.57) to[quadratic={(97.64,42.15)}] (106.4,41.17);
  \draw[red,very thick] (106.4,41.17) to[quadratic={(112.4,40.5)}] (120.5,40.5);
  \draw[black] (-90.62,-45.41) .. controls(-90.62,-40.32)and(-89.21,-35.64) .. (-86.63,-31.37);
  \draw[black,thick,dotted] (-86.63,-31.37) .. controls(-60.62,11.66)and(84.21,13.19) .. (103.56,-31.37);
  \draw[black] (103.56,-31.37) .. controls(104.96,-34.58)and(105.7,-38.03) .. (105.7,-41.72);
  \draw[black] (-54.29,-41.72) .. controls (-54.29,-37.87)and(-53.14,-34.42) .. (-51.07,-31.37);
  \draw[black,thick,dotted] (-51.07,-31.37) .. controls (-34.48,-6.98)and(41.36,-8.11) .. (63.38,-31.37);
  \draw[black] (63.38,-31.37) ..controls (67.18,-35.39)and(69.38,-40.08) .. (69.38,-45.41);
  \draw[black] (55.29,42.72) .. controls+(0,-40)and+(0,-40) .. (91.62,46.41);
  \draw[black] (-133.6,48.63) -- (-133.6,-31.37);
  \draw[black,thick,dotted] (-119.5,40.5) -- (-119.5,-32.03);
  \draw[black] (-119.5,-32.03) -- (-119.5,-39.5);
  \draw[black,thick,dotted] (-105.4,32.37) -- (-105.4,-38.14);
  \draw[black] (-105.4,-38.14) -- (-105.4,-47.63);
  \draw[black] (106.4,48.63) -- (106.4,-31.37);
  \draw[black] (120.5,40.5) -- (120.5,-39.5);
  \draw[black] (134.6,32.37) -- (134.6,-47.63);
  \draw[red,very thick] (-133.6,-31.37) -- (106.4,-31.37);
  \draw[red,very thick] (-119.5,-39.5) to[quadratic={(-99.5,-39.5)}] (-92.46,-43.57) to[quadratic={(-85.41,-47.63)}] (-105.4,-47.63);
  \draw[red,very thick] (-52.46,-43.57) to[quadratic={(-59.5,-39.5)}] (-39.5,-39.5) -- (40.5,-39.5) to[quadratic={(60.5,-39.5)}] (67.54,-43.57) to[quadratic={(74.59,-47.63)}] (54.59,-47.63) -- (-25.41,-47.63) to[quadratic={(-45.41,-47.63)}] (-52.46,-43.57);
  \draw[red,very thick] (120.5,-39.5) to[quadratic={(100.5,-39.5)}] (107.5,-43.57) to[quadratic={(114.6,-47.63)}] (134.6,-47.63);
}

\defTikzBox[xlen=.25pt,ylen=-.4pt]{BordCompGTrd}{%
  \draw[red,very thick] (-133.6,48.63) -- (106.4,48.63);
  \draw[red,very thick,dotted] (-119.5,40.5) -- (106.4,40.5);
  \draw[red,very thick] (106.4,40.5) -- (120.5,40.5);
  \draw[red,very thick,dotted] (-105.4,32.37) -- (120.5,32.37);
  \draw[red,very thick] (120.5,32.37) -- (134.6,32.37);
  \draw[black] (69.38,-45.41) .. controls(69.38,-40.08)and(70.03,-35.59) .. (71.14,-31.37);
  \draw[black,thick,dotted] (71.14,-31.37) ..controls(77.61,-8.11)and(99.88,-6.99) .. (104.75,-31.37);
  \draw[black] (104.75,-31.37) .. controls(105.36,-34.42)and(105.7,-37.87) .. (105.7,-41.72);
  \draw[black] (-133.6,48.63) -- (-133.6,-31.37);
  \draw[black,thick,dotted] (-119.5,40.5) -- (-119.5,-31.37);
  \draw[black] (-119.5,-31.37) -- (-119.5,-39.5);
  \draw[black,thick,dotted] (-105.4,32.37) -- (-105.4,-39.5);
  \draw[black] (-105.4,-39.5) -- (-105.4,-47.63);
  \draw[black] (106.4,48.63) -- (106.4,-31.37);
  \draw[black] (120.5,40.5) -- (120.5,-39.5);
  \draw[black] (134.6,32.37) -- (134.6,-47.63);
  \draw[red,very thick] (-133.6,-31.37) -- (106.4,-31.37);
  \draw[red,very thick] (-119.5,-39.5) -- (40.5,-39.5) to[quadratic={(60.5,-39.5)}] (67.54,-43.57) to[quadratic={(74.59,-47.63)}] (54.59,-47.63) -- (-105.4,-47.63);
  \draw[red,very thick] (120.5,-39.5) to[quadratic={(100.5,-39.5)}] (107.5,-43.57) to[quadratic={(114.6,-47.63)}] (134.6,-47.63);
}

\defTikzBox[xlen=.25pt,ylen=-.4pt]{BordCompGFth}{%
  \draw[red,very thick] (-133.6,48.63) to[quadratic={(-113.6,48.63)}] (-106.5,44.57) to[quadratic={(-104.62,43.48)}] (-104.69,42.68);
  \draw[red,very thick,dotted] (-104.69,42.68) to[quadratic={(-104.86,40.5)}] (-119.5,40.5);
  \draw[red,very thick,dotted] (-105.4,32.37) -- (-103,32.37);
  \draw[red,very thick] (-103,32.37) -- (-62.6,32.37);
  \draw[red,very thick,dotted] (-62.6,32.37) -- (52.2,32.37);
  \draw[red,very thick] (52.2,32.37) -- (89.6,32.37);
  \draw[red,very thick,dotted] (89.6,32.37) -- (120.5,32.37);
  \draw[red,very thick] (120.5,32.37) -- (134.6,32.37);
  \draw[red,very thick] (-68.38,46.41) to[quadratic={(-68.38,48.63)}] (-53.59,48.63) -- (26.41,48.63) to[quadratic={(46.41,48.63)}] (53.46,44.57) to[quadratic={(55.29,43.51)}] (55.29,42.73);
  \draw[red,very thick,dotted] (55.29,42.73) to[quadratic={(55.29,40.5)}] (40.5,40.5) -- (-39.5,40.5) to[quadratic={(-59.5,40.5)}] (-66.54,44.57) to[quadratic={(-68.38,45.63)}] (-68.38,46.41);
  \draw[red,very thick]  (106.4,48.63) to[quadratic={(91.62,48.63)}] (91.62,46.41);
  \draw[red,very thick,dotted] (91.62,46.41) to[quadratic={(91.62,45.63)}] (93.46,44.57) to[quadratic={(97.64,42.15)}] (106.4,41.17);
  \draw[red,very thick]  (106.4,41.17) to[quadratic={(112.4,40.5)}] (120.5,40.5);
  \draw[black] (-68.38,46.41) .. controls (-68.38,6.41)and(55.29,2.72) .. (55.29,42.72);
  \draw[black] (-104.7,42.72) .. controls+(0,-60)and+(0,-60) .. (91.62,46.41);
  \draw[black] (69.38,-45.41) .. controls(69.38,-40.08)and(70.03,-35.59) .. (71.14,-31.37);
  \draw[black,thick,dotted] (71.14,-31.37) ..controls(77.61,-8.11)and(99.88,-6.99) .. (104.75,-31.37);
  \draw[black] (104.75,-31.37) .. controls(105.36,-34.42)and(105.7,-37.87) .. (105.7,-41.72);
  \draw[black] (-133.6,48.63) -- (-133.6,-31.37);
  \draw[black,thick,dotted] (-119.5,40.5) -- (-119.5,-31.37);
  \draw[black] (-119.5,-31.37) -- (-119.5,-39.5);
  \draw[black,thick,dotted] (-105.4,32.37) -- (-105.4,-39.5);
  \draw[black] (-105.4,-39.5) -- (-105.4,-47.63);
  \draw[black] (106.4,48.63) -- (106.4,-31.37);
  \draw[black] (120.5,40.5) -- (120.5,-39.5);
  \draw[black] (134.6,32.37) -- (134.6,-47.63);
  \draw[red,very thick] (-133.6,-31.37) -- (106.4,-31.37);
  \draw[red,very thick] (-119.5,-39.5) -- (40.5,-39.5) to[quadratic={(60.5,-39.5)}] (67.54,-43.57) to[quadratic={(74.59,-47.63)}] (54.59,-47.63) -- (-105.4,-47.63);
  \draw[red,very thick] (120.5,-39.5) to[quadratic={(100.5,-39.5)}] (107.5,-43.57) to[quadratic={(114.6,-47.63)}] (134.6,-47.63);
}

\defTikzBox[xlen=.25pt,ylen=-.4pt]{BordCompGFstVar}{%
  \draw[red,very thick] (-133.6,48.63) to[quadratic={(-113.6,48.63)}] (-106.5,44.57) to[quadratic={(-104.69,43.51)}] (-104.69,42.73);
  \draw[red,very thick,dotted] (-104.69,42.73) to[quadratic={(-104.69,40.5)}] (-119.5,40.5);
  \draw[red,very thick] (-105.4,32.37) -- (-70.5,32.37);
  \draw[red,very thick,dotted] (-70.5,32.37) -- (120.5,32.37);
  \draw[red,very thick] (120.5,32.37) -- (134.6,32.37);
  \draw[red,very thick] (106.4,48.63) -- (-53.59,48.63) to[quadratic={(-68.38,48.63)}] (-68.38,46.41);
  \draw[red,very thick,dotted] (-68.38,46.41) to[quadratic={(-68.38,45.63)}] (-66.54,44.57) to[quadratic={(-59.5,40.5)}] (-39.5,40.5) -- (106.4,40.5);
  \draw[red,very thick] (106.4,40.5) -- (120.5,40.5);
  \draw[black] (-104.71,42.72) .. controls+(0,-40)and+(0,-40) .. (-68.38,46.41);
  \draw[black] (-133.6,48.63) -- (-133.6,-31.37);
  \draw[black,thick,dotted] (-119.5,40.5) -- (-119.5,-32.03);
  \draw[black] (-119.5,-32.03) -- (-119.5,-39.5);
  \draw[black,thick,dotted] (-105.4,32.37) -- (-105.4,-38.14);
  \draw[black] (-105.4,-38.14) -- (-105.4,-47.63);
  \draw[black] (106.4,48.63) -- (106.4,-31.37);
  \draw[black] (120.5,40.5) -- (120.5,-39.5);
  \draw[black] (134.6,32.37) -- (134.6,-47.63);
  \draw[red,very thick] (-133.6,-31.37) -- (106.4,-31.37);
  \draw[red,very thick] (-119.5,-39.5) -- (120.5,-39.5);
  \draw[red,very thick] (-105.4,-47.63) -- (134.6,-47.63);
}

\defTikzBox[xlen=.25pt,ylen=-.4pt]{BordCompGSndVar}{%
  \draw[red,very thick] (-133.6,48.63) to[quadratic={(-113.6,48.63)}] (-106.5,44.57) to[quadratic={(-104.69,43.51)}] (-104.69,42.73);
  \draw[red,very thick,dotted] (-104.69,42.73) to[quadratic={(-104.69,40.5)}] (-119.5,40.5);
  \draw[red,very thick] (-105.4,32.37) -- (-70.5,32.37);
  \draw[red,very thick,dotted] (-70.5,32.37) -- (120.5,32.37);
  \draw[red,very thick] (120.5,32.37) -- (134.6,32.37);
  \draw[red,very thick] (106.4,48.63) -- (-53.59,48.63) to[quadratic={(-68.38,48.63)}] (-68.38,46.41);
  \draw[red,very thick,dotted] (-68.38,46.41) to[quadratic={(-68.38,45.63)}] (-66.54,44.57) to[quadratic={(-59.5,40.5)}] (-39.5,40.5) -- (106.4,40.5);
  \draw[red,very thick] (106.4,40.5) -- (120.5,40.5);
  \draw[black] (-90.62,-45.41) .. controls(-90.62,-40.32)and(-89.21,-35.64) .. (-86.63,-31.37);
  \draw[black,thick,dotted] (-86.63,-31.37) .. controls(-60.62,11.66)and(84.21,13.19) .. (103.56,-31.37);
  \draw[black] (103.56,-31.37) .. controls(104.96,-34.58)and(105.7,-38.03) .. (105.7,-41.72);
  \draw[black] (-54.29,-41.72) .. controls (-54.29,-37.87)and(-53.14,-34.42) .. (-51.07,-31.37);
  \draw[black,thick,dotted] (-51.07,-31.37) .. controls (-34.48,-6.98)and(41.36,-8.11) .. (63.38,-31.37);
  \draw[black] (63.38,-31.37) ..controls (67.18,-35.39)and(69.38,-40.08) .. (69.38,-45.41);
  \draw[black] (-104.71,42.72) .. controls+(0,-40)and+(0,-40) .. (-68.38,46.41);
  \draw[black] (-133.6,48.63) -- (-133.6,-31.37);
  \draw[black,thick,dotted] (-119.5,40.5) -- (-119.5,-32.03);
  \draw[black] (-119.5,-32.03) -- (-119.5,-39.5);
  \draw[black,thick,dotted] (-105.4,32.37) -- (-105.4,-38.14);
  \draw[black] (-105.4,-38.14) -- (-105.4,-47.63);
  \draw[black] (106.4,48.63) -- (106.4,-31.37);
  \draw[black] (120.5,40.5) -- (120.5,-39.5);
  \draw[black] (134.6,32.37) -- (134.6,-47.63);
  \draw[red,very thick] (-133.6,-31.37) -- (106.4,-31.37);
  \draw[red,very thick] (-119.5,-39.5) to[quadratic={(-99.5,-39.5)}] (-92.46,-43.57) to[quadratic={(-85.41,-47.63)}] (-105.4,-47.63);
  \draw[red,very thick] (-52.46,-43.57) to[quadratic={(-59.5,-39.5)}] (-39.5,-39.5) -- (40.5,-39.5) to[quadratic={(60.5,-39.5)}] (67.54,-43.57) to[quadratic={(74.59,-47.63)}] (54.59,-47.63) -- (-25.41,-47.63) to[quadratic={(-45.41,-47.63)}] (-52.46,-43.57);
  \draw[red,very thick] (120.5,-39.5) to[quadratic={(100.5,-39.5)}] (107.5,-43.57) to[quadratic={(114.6,-47.63)}] (134.6,-47.63);
}

\defTikzBox[xlen=.25pt,ylen=-.4pt]{BordCompGTrdVar}{%
  \draw[red,very thick] (-133.6,48.63) -- (106.4,48.63);
  \draw[red,very thick,dotted] (-119.5,40.5) -- (106.4,40.5);
  \draw[red,very thick] (106.4,40.5) -- (120.5,40.5);
  \draw[red,very thick,dotted] (-105.4,32.37) -- (120.5,32.37);
  \draw[red,very thick] (120.5,32.37) -- (134.6,32.37);
  \draw[black] (-90.62,-45.41) .. controls(-90.62,-40.08)and(-89.97,-35.59) .. (-88.86,-31.37);
  \draw[black,thick,dotted] (-88.86,-31.37) ..controls(-82.39,-8.11)and(-60.12,-6.99) .. (-55.25,-31.37);
  \draw[black] (-55.25,-31.37) .. controls(-54.64,-34.42)and(-54.3,-37.87) .. (-54.3,-41.72);
  \draw[black] (-133.6,48.63) -- (-133.6,-31.37);
  \draw[black,thick,dotted] (-119.5,40.5) -- (-119.5,-31.37);
  \draw[black] (-119.5,-31.37) -- (-119.5,-39.5);
  \draw[black,thick,dotted] (-105.4,32.37) -- (-105.4,-39.5);
  \draw[black] (-105.4,-39.5) -- (-105.4,-47.63);
  \draw[black] (106.4,48.63) -- (106.4,-31.37);
  \draw[black] (120.5,40.5) -- (120.5,-39.5);
  \draw[black] (134.6,32.37) -- (134.6,-47.63);
  \draw[red,very thick] (-133.6,-31.37) -- (106.4,-31.37);
  \draw[red,very thick] (-119.5,-39.5) to[quadratic={(-99.5,-39.5)}] (-92.46,-43.57) to[quadratic={(-85.41,-47.63)}] (-105.4,-47.63);
  \draw[red,very thick] (120.5,-39.5) -- (-39.5,-39.5) to[quadratic={(-59.5,-39.5)}] (-52.46,-43.57) to[quadratic={(-45.41,-47.63)}] (-25.41,-47.63) -- (134.6,-47.63);
}

\defTikzBox[xlen=.25pt,ylen=-.4pt]{BordCompGFthVar}{%
  \draw[red,very thick] (-133.6,48.63) to[quadratic={(-113.6,48.63)}] (-106.5,44.57) to[quadratic={(-104.62,43.48)}] (-104.69,42.68);
  \draw[red,very thick,dotted] (-104.69,42.68) to[quadratic={(-104.86,40.5)}] (-119.5,40.5);
  \draw[red,very thick,dotted] (-105.4,32.37) -- (-103,32.37);
  \draw[red,very thick] (-103,32.37) -- (-62.6,32.37);
  \draw[red,very thick,dotted] (-62.6,32.37) -- (52.2,32.37);
  \draw[red,very thick] (52.2,32.37) -- (89.6,32.37);
  \draw[red,very thick,dotted] (89.6,32.37) -- (120.5,32.37);
  \draw[red,very thick] (120.5,32.37) -- (134.6,32.37);
  \draw[red,very thick] (-68.38,46.41) to[quadratic={(-68.38,48.63)}] (-53.59,48.63) -- (26.41,48.63) to[quadratic={(46.41,48.63)}] (53.46,44.57) to[quadratic={(55.29,43.51)}] (55.29,42.73);
  \draw[red,very thick,dotted] (55.29,42.73) to[quadratic={(55.29,40.5)}] (40.5,40.5) -- (-39.5,40.5) to[quadratic={(-59.5,40.5)}] (-66.54,44.57) to[quadratic={(-68.38,45.63)}] (-68.38,46.41);
  \draw[red,very thick]  (106.4,48.63) to[quadratic={(91.62,48.63)}] (91.62,46.41);
  \draw[red,very thick,dotted] (91.62,46.41) to[quadratic={(91.62,45.63)}] (93.46,44.57) to[quadratic={(97.64,42.15)}] (106.4,41.17);
  \draw[red,very thick]  (106.4,41.17) to[quadratic={(112.4,40.5)}] (120.5,40.5);
  \draw[black] (-68.38,46.41) .. controls (-68.38,6.41)and(55.29,2.72) .. (55.29,42.72);
  \draw[black] (-104.7,42.72) .. controls+(0,-60)and+(0,-60) .. (91.62,46.41);
  \draw[black] (-90.62,-45.41) .. controls(-90.62,-40.08)and(-89.97,-35.59) .. (-88.86,-31.37);
  \draw[black,thick,dotted] (-88.86,-31.37) ..controls(-82.39,-8.11)and(-60.12,-6.99) .. (-55.25,-31.37);
  \draw[black] (-55.25,-31.37) .. controls(-54.64,-34.42)and(-54.3,-37.87) .. (-54.3,-41.72);
  \draw[black] (-133.6,48.63) -- (-133.6,-31.37);
  \draw[black,thick,dotted] (-119.5,40.5) -- (-119.5,-31.37);
  \draw[black] (-119.5,-31.37) -- (-119.5,-39.5);
  \draw[black,thick,dotted] (-105.4,32.37) -- (-105.4,-39.5);
  \draw[black] (-105.4,-39.5) -- (-105.4,-47.63);
  \draw[black] (106.4,48.63) -- (106.4,-31.37);
  \draw[black] (120.5,40.5) -- (120.5,-39.5);
  \draw[black] (134.6,32.37) -- (134.6,-47.63);
  \draw[red,very thick] (-133.6,-31.37) -- (106.4,-31.37);
  \draw[red,very thick] (-119.5,-39.5) to[quadratic={(-99.5,-39.5)}] (-92.46,-43.57) to[quadratic={(-85.41,-47.63)}] (-105.4,-47.63);
  \draw[red,very thick] (120.5,-39.5) -- (-39.5,-39.5) to[quadratic={(-59.5,-39.5)}] (-52.46,-43.57) to[quadratic={(-45.41,-47.63)}] (-25.41,-47.63) -- (134.6,-47.63);
}

\defTikzBox[xlen=.25pt,ylen=-.4pt]{BordCompHFst}{%
  \draw[red,very thick] (-133.6,48.63) -- (26.41,48.63) to[quadratic={(46.41,48.63)}] (53.46,44.57) to[quadratic={(55.29,43.51)}] (55.29,42.73);
  \draw[red,very thick,dotted] (55.29,42.73) to[quadratic={(55.29,40.5)}] (40.5,40.5) -- (-119.5,40.5);
  \draw[red,very thick,dotted] (-105.4,32.37) -- (53,32.37);
  \draw[red,very thick] (53,32.37) -- (88,32.37);
  \draw[red,very thick,dotted] (88,32.37) -- (120.5,32.37);
  \draw[red,very thick] (120.5,32.37) -- (134.6,32.37);
  \draw[red,very thick] (106.4,48.63) to[quadratic={(91.62,48.63)}] (91.62,46.41);
  \draw[red,very thick,dotted] (91.62,46.41) to[quadratic={(91.62,45.63)}] (93.46,44.57) to[quadratic={(97.64,42.15)}] (106.4,41.17);
  \draw[red,very thick] (106.4,41.17) to[quadratic={(112.4,40.5)}] (120.5,40.5);
  \draw[black] (55.29,42.72) .. controls+(0,-40)and+(0,40) .. (-24.71,-37.28);
  \draw[black] (91.62,46.41) .. controls+(0,-40)and+(0,40) .. (11.62,-33.59);
  \draw[black] (-133.6,48.63) -- (-133.6,-31.37);
  \draw[black,thick,dotted] (-119.5,40.5) -- (-119.5,-31.37);
  \draw[black] (-119.5,-31.37) -- (-119.5,-39.5);
  \draw[black,thick,dotted] (-105.4,32.37) -- (-105.4,-39.5);
  \draw[black] (-105.4,-39.5) -- (-105.4,-47.63);
  \draw[black] (106.4,48.63) -- (106.4,-31.37);
  \draw[black] (120.5,40.5) -- (120.5,-39.5);
  \draw[black] (134.6,32.37) -- (134.6,-47.63);
  \draw[red,very thick] (-133.6,-31.37) -- (-53.59,-31.37) to[quadratic={(-33.59,-31.37)}] (-26.54,-35.43) to[quadratic={(-19.5,-39.5)}] (-39.5,-39.5) -- (-119.5,-39.5);
  \draw[red,very thick] (-105.4,-47.63) -- (134.6,-47.63);
  \draw[red,very thick] (106.4,-31.37) -- (26.41,-31.37) to[quadratic={(6.41,-31.37)}] (13.46,-35.43) to[quadratic={(20.5,-39.5)}] (40.5,-39.5) -- (120.5,-39.5);
}

\defTikzBox[xlen=.25pt,ylen=-.4pt]{BordCompHSnd}{%
  \draw[red,very thick] (-133.6,48.63) to[quadratic={(-113.6,48.63)}] (-106.5,44.57) to[quadratic={(-104.62,43.48)}] (-104.69,42.68);
  \draw[red,very thick,dotted] (-104.69,42.68) to[quadratic={(-104.86,40.5)}] (-119.5,40.5);
  \draw[red,very thick,dotted] (-105.4,32.37) -- (-103,32.37);
  \draw[red,very thick] (-103,32.37) -- (-62.6,32.37);
  \draw[red,very thick,dotted] (-62.6,32.37) -- (52.2,32.37);
  \draw[red,very thick] (52.2,32.37) -- (88.6,32.37);
  \draw[red,very thick,dotted] (88.6,32.37) -- (120.5,32.37);
  \draw[red,very thick] (120.5,32.37) -- (134.6,32.37);
  \draw[red,very thick] (-68.38,46.41) to[quadratic={(-68.38,48.63)}] (-53.59,48.63) -- (26.41,48.63) to[quadratic={(46.41,48.63)}] (53.46,44.57) to[quadratic={(55.29,43.51)}] (55.29,42.73);
  \draw[red,very thick,dotted] (55.29,42.73) to[quadratic={(55.29,40.5)}] (40.5,40.5) -- (-39.5,40.5) to[quadratic={(-59.5,40.5)}] (-66.54,44.57) to[quadratic={(-68.38,45.63)}] (-68.38,46.41);
  \draw[red,very thick]  (106.4,48.63) to[quadratic={(91.62,48.63)}] (91.62,46.41);
  \draw[red,very thick,dotted] (91.62,46.41) to[quadratic={(91.62,45.63)}] (93.46,44.57) to[quadratic={(97.64,42.15)}] (106.4,41.17);
  \draw[red,very thick]  (106.4,41.17) to[quadratic={(112.4,40.5)}] (120.5,40.5);
  \draw[black] (-68.38,46.41) .. controls (-68.38,6.41)and(55.29,2.72) .. (55.29,42.72);
  \draw[black] (91.62,46.41) .. controls+(0,-40)and+(0,40) .. (11.62,-33.59);
  \draw[black] (-104.7,42.72) .. controls+(0,-40)and+(0,40) .. (-24.71,-37.28);
  \draw[black] (105.7,-41.72) .. controls(105.7,-37.78)and(105.24,-34.34) .. (104.43,-31.37);
  \draw[black,thick,dotted] (104.43,-31.37) .. controls(99.12,-12.05)and(78.73,-13.02) .. (71.75,-31.37);
  \draw[black] (71.75,-31.37) .. controls(70.26,-35.27)and(69.38,-39.96) .. (69.38,-45.41);
  \draw[black] (-133.6,48.63) -- (-133.6,-31.37);
  \draw[black,thick,dotted] (-119.5,40.5) -- (-119.5,-31.37);
  \draw[black] (-119.5,-31.37) -- (-119.5,-39.5);
  \draw[black,thick,dotted] (-105.4,32.37) -- (-105.4,-39.5);
  \draw[black] (-105.4,-39.5) -- (-105.4,-47.63);
  \draw[black] (106.4,48.63) -- (106.4,-31.37);
  \draw[black] (120.5,40.5) -- (120.5,-39.5);
  \draw[black] (134.6,32.37) -- (134.6,-47.63);
  \draw[red,very thick] (-133.6,-31.37) -- (-53.59,-31.37) to[quadratic={(-33.59,-31.37)}] (-26.54,-35.43) to[quadratic={(-19.5,-39.5)}] (-39.5,-39.5) -- (-119.5,-39.5);
  \draw[red,very thick] (-105.4,-47.63) -- (54.59,-47.63) to[quadratic={(74.59,-47.63)}] (67.54,-43.57) to[quadratic={(60.5,-39.5)}] (40.5,-39.5) to[quadratic={(20.5,-39.5)}] (13.46,-35.43) to[quadratic={(6.41,-31.37)}] (26.41,-31.37) -- (106.4,-31.37);
  \draw[red,very thick]  (120.5,-39.5) to[quadratic={(100.5,-39.5)}] (107.5,-43.57) to[quadratic={(114.6,-47.63)}] (134.6,-47.63);
}

\defTikzBox[xlen=.25pt,ylen=-.4pt]{BordCompHFstVar}{%
  \draw[red,very thick] (-133.6,48.63) to[quadratic={(-113.6,48.63)}] (-106.5,44.57) to[quadratic={(-104.69,43.51)}] (-104.69,42.73);
  \draw[red,very thick,dotted] (-104.69,42.73) to[quadratic={(-104.69,40.5)}] (-119.5,40.5);
  \draw[red,very thick,dotted] (-105.4,32.37) -- (-104.7,32.37);
  \draw[red,very thick] (-104.7,32.37) -- (-68.38,32.37);
  \draw[red,very thick,dotted] (-68.38,32.37) -- (120.5,32.37);
  \draw[red,very thick] (120.5,32.37) -- (134.6,32.37);
  \draw[red,very thick] (106.4,48.63) -- (-53.59,48.63) to[quadratic={(-68.38,48.63)}] (-68.38,46.41);
  \draw[red,very thick,dotted] (-68.38,46.41) to[quadratic={(-68.38,45.63)}] (-66.54,44.57) to[quadratic={(-59.5,40.5)}] (-39.5,40.5) -- (106.4,40.5);
  \draw[red,very thick] (106.4,40.5) -- (120.5,40.5);
  \draw[black] (-104.71,42.72) .. controls+(0,-40)and+(0,40) .. (-24.71,-37.28);
  \draw[black] (-68.38,46.41) .. controls+(0,-40)and+(0,40) .. (11.62,-33.59);
  \draw[black] (-133.6,48.63) -- (-133.6,-31.37);
  \draw[black,thick,dotted] (-119.5,40.5) -- (-119.5,-31.37);
  \draw[black] (-119.5,-31.37) -- (-119.5,-39.5);
  \draw[black,thick,dotted] (-105.4,32.37) -- (-105.4,-39.5);
  \draw[black] (-105.4,-39.5) -- (-105.4,-47.63);
  \draw[black] (106.4,48.63) -- (106.4,-31.37);
  \draw[black] (120.5,40.5) -- (120.5,-39.5);
  \draw[black] (134.6,32.37) -- (134.6,-47.63);
  \draw[red,very thick] (-133.6,-31.37) -- (-53.59,-31.37) to[quadratic={(-33.59,-31.37)}] (-26.54,-35.43) to[quadratic={(-19.5,-39.5)}] (-39.5,-39.5) -- (-119.5,-39.5);
  \draw[red,very thick] (-105.4,-47.63) -- (134.6,-47.63);
  \draw[red,very thick] (106.4,-31.37) -- (26.41,-31.37) to[quadratic={(6.41,-31.37)}] (13.46,-35.43) to[quadratic={(20.5,-39.5)}] (40.5,-39.5) -- (120.5,-39.5);
}

\defTikzBox[xlen=.25pt,ylen=-.4pt]{BordCompHSndVar}{%
  \draw[red,very thick] (-133.6,48.63) to[quadratic={(-113.6,48.63)}] (-106.5,44.57) to[quadratic={(-104.62,43.48)}] (-104.69,42.68);
  \draw[red,very thick,dotted] (-104.69,42.68) to[quadratic={(-104.86,40.5)}] (-119.5,40.5);
  \draw[red,very thick,dotted] (-105.4,32.37) -- (-103,32.37);
  \draw[red,very thick] (-103,32.37) -- (-62.6,32.37);
  \draw[red,very thick,dotted] (-62.6,32.37) -- (52.2,32.37);
  \draw[red,very thick] (52.2,32.37) -- (88.6,32.37);
  \draw[red,very thick,dotted] (88.6,32.37) -- (120.5,32.37);
  \draw[red,very thick] (120.5,32.37) -- (134.6,32.37);
  \draw[red,very thick] (-68.38,46.41) to[quadratic={(-68.38,48.63)}] (-53.59,48.63) -- (26.41,48.63) to[quadratic={(46.41,48.63)}] (53.46,44.57) to[quadratic={(55.29,43.51)}] (55.29,42.73);
  \draw[red,very thick,dotted] (55.29,42.73) to[quadratic={(55.29,40.5)}] (40.5,40.5) -- (-39.5,40.5) to[quadratic={(-59.5,40.5)}] (-66.54,44.57) to[quadratic={(-68.38,45.63)}] (-68.38,46.41);
  \draw[red,very thick]  (106.4,48.63) to[quadratic={(91.62,48.63)}] (91.62,46.41);
  \draw[red,very thick,dotted] (91.62,46.41) to[quadratic={(91.62,45.63)}] (93.46,44.57) to[quadratic={(97.64,42.15)}] (106.4,41.17);
  \draw[red,very thick]  (106.4,41.17) to[quadratic={(112.4,40.5)}] (120.5,40.5);
  \draw[black] (-68.38,46.41) .. controls (-68.38,6.41)and(55.29,2.72) .. (55.29,42.72);
  \draw[black] (91.62,46.41) .. controls+(0,-40)and+(0,40) .. (11.62,-33.59);
  \draw[black] (-104.7,42.72) .. controls+(0,-40)and+(0,40) .. (-24.71,-37.28);
  \draw[black] (-54.3,-41.72) .. controls(-54.3,-37.78)and(-54.76,-34.34) .. (-55.57,-31.37);
  \draw[black,thick,dotted] (-55.57,-31.37) .. controls(-60.88,-12.05)and(-81.27,-13.02) .. (-88.25,-31.37);
  \draw[black] (-88.25,-31.37) .. controls(-89.74,-35.27)and(-90.62,-39.96) .. (-90.62,-45.41);
  \draw[black] (-133.6,48.63) -- (-133.6,-31.37);
  \draw[black,thick,dotted] (-119.5,40.5) -- (-119.5,-31.37);
  \draw[black] (-119.5,-31.37) -- (-119.5,-39.5);
  \draw[black,thick,dotted] (-105.4,32.37) -- (-105.4,-39.5);
  \draw[black] (-105.4,-39.5) -- (-105.4,-47.63);
  \draw[black] (106.4,48.63) -- (106.4,-31.37);
  \draw[black] (120.5,40.5) -- (120.5,-39.5);
  \draw[black] (134.6,32.37) -- (134.6,-47.63);
  \draw[red,very thick] (-133.6,-31.37) -- (-53.59,-31.37) to[quadratic={(-33.59,-31.37)}] (-26.54,-35.43) to[quadratic={(-19.5,-39.5)}] (-39.5,-39.5) to[quadratic={(-59.5,-39.5)}] (-52.46,-43.57) to[quadratic={(-45.41,-47.63)}] (-25.41,-47.63) -- (134.6,-47.63);
  \draw[red,very thick] (-119.5,-39.5) to[quadratic={(-99.5,-39.5)}] (-92.46,-43.57) to[quadratic={(-85.41,-47.63)}] (-105.4,-47.63);
  \draw[red,very thick] (106.4,-31.37) -- (26.41,-31.37) to[quadratic={(6.41,-31.37)}] (13.46,-35.43) to[quadratic={(20.5,-39.5)}] (40.5,-39.5) -- (120.5,-39.5);
}

\defTikzBox[xlen=.25pt,ylen=-.4pt]{BordPsiMinus}{%
  \draw[red,very thick]  (-133.6,48.63) -- (26.41,48.63) to[quadratic={(46.41,48.63)}] (53.46,44.57) to[quadratic={(55.29,43.51)}] (55.29,42.73);
  \draw[red,very thick,dotted] (55.29,42.73) to[quadratic={(55.29,40.5)}] (40.5,40.5) to[quadratic={(20.5,40.5)}] (27.54,36.43) to[quadratic={(34.59,32.37)}] (54.59,32.37) -- (56.1,32.37);
  \draw[red,very thick] (56.1,32.37) -- (89,32.37);
  \draw[red,very thick,dotted] (89,32.37) -- (120.5,32.37);
  \draw[red,very thick] (120.5,32.37) -- (134.6,32.37);
  \draw[red,very thick,dotted] (-119.5,40.5) -- (-39.5,40.5) to[quadratic={(-19.5,40.5)}] (-12.46,36.43) to [quadratic={(-10.62,35.37)}] (-10.62,34.59);
  \draw[red,very thick,dotted] (-10.62,34.59) to[quadratic={(-10.62,32.37)}] (-25.41,32.37) -- (-105.4,32.37);
  \draw[red,very thick] (106.4,48.63) to[quadratic={(91.62,48.63)}] (91.62,46.41);
  \draw[red,very thick,dotted] (91.62,46.41) to[quadratic={(91.62,45.62)}] (93.46,44.57) to[quadratic={(97.64,42.15)}] (106.4,41.17);
  \draw[red,very thick] (106.4,41.17) to[quadratic={(112.4,40.5)}] (120.5,40.5);
  \draw[black,thick,dotted] (25.71,38.28) .. controls(25.71,2.07)and(91.28,-1.36) .. (103.69,-31.37);
  \draw[black] (103.69,-31.37) .. controls(104.98,-34.51)and(105.7,-37.93) .. (105.7,-41.72);
  \draw[black,thick,dotted] (-10.62,34.59) .. controls(-10.62,-0.08)and(49.49,-4.7) .. (65.5,-31.37);
  \draw[black] (65.5,-31.37) .. controls(67.96,-35.47)and(69.38,-40.08) .. (69.38,-45.41);
  \draw[black] (55.29,42.72) .. controls+(0,-35)and+(0,-35) .. (91.62,46.41);
  \draw[black] (-133.6,48.63) -- (-133.6,-31.37);
  \draw[black,thick,dotted] (-119.5,40.5) -- (-119.5,-31.37);
  \draw[black] (-119.5,-31.37) -- (-119.5,-39.5);
  \draw[black,thick,dotted] (-105.4,32.37) -- (-105.4,-40.17);
  \draw[black] (-105.4,-40.17) -- (-105.4,-47.63);
  \draw[black] (106.4,48.63) -- (106.4,-31.37);
  \draw[black] (120.5,40.5) -- (120.5,-39.5);
  \draw[black] (134.6,32.37) -- (134.6,-47.63);
  \draw[red,very thick] (-133.6,-31.37) -- (106.4,-31.37);
  \draw[red,very thick] (-119.5,-39.5) -- (40.5,-39.5) to[quadratic={(60.5,-39.5)}] (67.54,-43.57) to[quadratic={(74.59,-47.63)}] (54.59,-47.63) -- (-105.4,-47.63);
  \draw[red,very thick] (120.5,-39.5) to[quadratic={(100.5,-39.5)}] (107.5,-43.57) to[quadratic={(114.6,-47.63)}] (134.6,-47.63);
}

\defTikzBox[xlen=.25pt,ylen=-.4pt]{BordPsiPlus}{%
  \draw[red,very thick] (-133.6,48.63) -- (26.41,48.63) to[quadratic={(46.41,48.63)}] (53.46,44.57) to[quadratic={(55.29,43.51)}] (55.29,42.73);
  \draw[red,very thick,dotted] (55.29,42.73) to[quadratic={(55.29,40.5)}] (40.5,40.5) -- (-119.5,40.5);
  \draw[red,very thick,dotted] (-105.4,32.37) -- (53.2,32.37);
  \draw[red,very thick] (53.2,32.37) -- (87.8,32.37);
  \draw[red,very thick,dotted] (87.8,32.37) -- (120.5,32.37);
  \draw[red,very thick] (120.5,32.37) -- (134.6,32.37);
  \draw[red,very thick] (106.4,48.63) to[quadratic={(91.62,48.63)}] (91.62,46.41);
  \draw[red,very thick,dotted] (91.62,46.41) to[quadratic={(91.62,45.63)}] (93.46,44.57) to[quadratic={(97.64,42.15)}] (106.4,41.17);
  \draw[red,very thick] (106.4,41.17) to[quadratic={(112.4,40.5)}] (120.5,40.5);
  \draw[black] (55.29,42.72) .. controls+(0,-40)and+(0,40) .. (-24.71,-37.28);
  \draw[black] (91.62,46.41) .. controls+(0,-40)and+(0,40) .. (11.62,-33.59);
  \draw[black] (105.7,-41.72) .. controls(105.7,-37.78)and(105.24,-34.34) .. (104.43,-31.37);
  \draw[black,thick,dotted] (104.43,-31.37) .. controls(99.12,-12.05)and(78.73,-13.02) .. (71.75,-31.37);
  \draw[black] (71.75,-31.37) .. controls(70.26,-35.27)and(69.38,-39.96) .. (69.38,-45.41);
  \draw[black] (-133.6,48.63) -- (-133.6,-31.37);
  \draw[black,thick,dotted] (-119.5,40.5) -- (-119.5,-31.37);
  \draw[black] (-119.5,-31.37) -- (-119.5,-39.5);
  \draw[black,thick,dotted] (-105.4,32.37) -- (-105.4,-39.5);
  \draw[black] (-105.4,-39.5) -- (-105.4,-47.63);
  \draw[black] (106.4,48.63) -- (106.4,-31.37);
  \draw[black] (120.5,40.5) -- (120.5,-39.5);
  \draw[black] (134.6,32.37) -- (134.6,-47.63);
  \draw[red,very thick] (-133.6,-31.37) -- (-53.59,-31.37) to[quadratic={(-33.59,-31.37)}] (-26.54,-35.43) to[quadratic={(-19.5,-39.5)}] (-39.5,-39.5) -- (-119.5,-39.5);
  \draw[red,very thick] (-105.4,-47.63) -- (54.59,-47.63) to[quadratic={(74.59,-47.63)}] (67.54,-43.57) to[quadratic={(60.5,-39.5)}] (40.5,-39.5) to[quadratic={(20.5,-39.5)}] (13.46,-35.43) to[quadratic={(6.41,-31.37)}] (26.41,-31.37) -- (106.4,-31.37);
  \draw[red,very thick]  (120.5,-39.5) to[quadratic={(100.5,-39.5)}] (107.5,-43.57) to[quadratic={(114.6,-47.63)}] (134.6,-47.63);
}

\defTikzBox[xlen=.25pt,ylen=-.4pt]{BordPsiMinusVar}{%
  \draw[red,very thick] (-133.6,48.63) to[quadratic={(-113.6,48.63)}] (-106.5,44.57) to[quadratic={(-104.69,43.51)}] (-104.69,42.73);
  \draw[red,very thick,dotted] (-104.69,42.73) to[quadratic={(-104.69,40.5)}] (-119.5,40.5);
  \draw[red,very thick]  (-105.4,32.37) -- (-71,32.37);
  \draw[red,very thick,dotted] (-71,32.37) -- (-25.41,32.37) to[quadratic={(-5.41,32.37)}] (-12.46,36.43) to[quadratic={(-19.5,40.5)}] (-39.5,40.5) to[quadratic={(-59.5,40.5)}] (-66.54,44.57) to[quadratic={(-68.38,45.63)}] (-68.38,46.41);
  \draw[red,very thick] (-68.38,46.41) to[quadratic={(-68.38,48.63)}] (-53.59,48.63) -- (106.4,48.63);
  \draw[red,very thick] (120.5,40.5) -- (106.4,40.5);
  \draw[red,very thick,dotted](106.4,40.5) -- (40.5,40.5) to[quadratic={(25.71,40.5)}] (25.71,38.28) to[quadratic={(25.71,37.49)}] (27.54,36.43) to[quadratic={(34.59,32.37)}] (54.59,32.37) -- (120.5,32.37);
  \draw[red,very thick]  (120.5,32.37) -- (134.6,32.37);
  \draw[black,thick,dotted] (-10.62,34.59) .. controls(-10.62,-0.08)and(-70.73,-4.7) .. (-86.74,-31.37);
  \draw[black] (-86.74,-31.37) .. controls(-89.2,-35.47)and(-90.62,-40.08) .. (-90.62,-45.41);
  \draw[black,thick,dotted] (25.71,38.28) ..controls(25.71,2.07)and(-39.86,-1.36).. (-52.28,-31.37);
  \draw[black] (-52.28,-31.37) ..controls(-53.57,-34.51)and(-54.29,-37.93).. (-54.29,-41.72);
  \draw[black] (-104.71,42.72) .. controls+(0,-35)and+(0,-35) .. (-68.38,46.41);
  \draw[black] (-133.6,48.63) -- (-133.6,-31.37);
  \draw[black,thick,dotted] (-119.5,40.5) -- (-119.5,-31.37);
  \draw[black] (-119.5,-31.37) -- (-119.5,-39.5);
  \draw[black,thick,dotted] (-105.4,32.37) -- (-105.4,-39.5);
  \draw[black] (-105.4,-39.5) -- (-105.4,-47.63);
  \draw[black] (106.4,48.63) -- (106.4,-31.37);
  \draw[black] (120.5,40.5) -- (120.5,-39.5);
  \draw[black] (134.6,32.37) -- (134.6,-47.63);
  \draw[red,very thick] (-133.6,-31.37) -- (106.4,-31.37);
  \draw[red,very thick] (-119.5,-39.5) to[quadratic={(-99.5,-39.5)}] (-92.46,-43.57) to[quadratic={(-85.41,-47.63)}] (-105.4,-47.63);
  \draw[red,very thick] (120.5,-39.5) -- (-39.5,-39.5) to[quadratic={(-59.5,-39.5)}] (-52.46,-43.57) to[quadratic={(-45.41,-47.63)}] (-25.41,-47.63) -- (134.6,-47.63);
}

\defTikzBox[xlen=.25pt,ylen=-.4pt]{BordPsiPlusVar}{%
  \draw[red,very thick] (-133.6,48.63) to[quadratic={(-113.6,48.63)}] (-106.5,44.57) to[quadratic={(-104.69,43.51)}] (-104.69,42.73);
  \draw[red,very thick,dotted] (-104.69,42.73) to[quadratic={(-104.69,40.5)}] (-119.5,40.5);
  \draw[red,very thick] (-105.4,32.37) -- (-70.5,32.37);
  \draw[red,very thick,dotted] (-70.5,32.37) -- (120.5,32.37);
  \draw[red,very thick] (120.5,32.37) -- (134.6,32.37);
  \draw[red,very thick] (106.4,48.63) -- (-53.59,48.63) to[quadratic={(-68.38,48.63)}] (-68.38,46.41);
  \draw[red,very thick,dotted] (-68.38,46.41) to[quadratic={(-68.38,45.63)}] (-66.54,44.57) to[quadratic={(-59.5,40.5)}] (-39.5,40.5) -- (106.4,40.5);
  \draw[red,very thick] (106.4,40.5) -- (120.5,40.5);
  \draw[black] (-104.71,42.72) .. controls+(0,-40)and+(0,40) .. (-24.71,-37.28);
  \draw[black] (-68.38,46.41) .. controls+(0,-40)and+(0,40) .. (11.62,-33.59);
  \draw[black] (-54.3,-41.72) .. controls(-54.3,-37.78)and(-54.76,-34.34) .. (-55.57,-31.37);
  \draw[black,thick,dotted] (-55.57,-31.37) .. controls(-60.88,-12.05)and(-81.27,-13.02) .. (-88.25,-31.37);
  \draw[black] (-88.25,-31.37) .. controls(-89.74,-35.27)and(-90.62,-39.96) .. (-90.62,-45.41);
  \draw[black] (-133.6,48.63) -- (-133.6,-31.37);
  \draw[black,thick,dotted] (-119.5,40.5) -- (-119.5,-31.37);
  \draw[black] (-119.5,-31.37) -- (-119.5,-39.5);
  \draw[black,thick,dotted] (-105.4,32.37) -- (-105.4,-39.5);
  \draw[black] (-105.4,-39.5) -- (-105.4,-47.63);
  \draw[black] (106.4,48.63) -- (106.4,-31.37);
  \draw[black] (120.5,40.5) -- (120.5,-39.5);
  \draw[black] (134.6,32.37) -- (134.6,-47.63);
  \draw[red,very thick] (-133.6,-31.37) -- (-53.59,-31.37) to[quadratic={(-33.59,-31.37)}] (-26.54,-35.43) to[quadratic={(-19.5,-39.5)}] (-39.5,-39.5) to[quadratic={(-59.5,-39.5)}] (-52.46,-43.57) to[quadratic={(-45.41,-47.63)}] (-25.41,-47.63) -- (134.6,-47.63);
  \draw[red,very thick] (-119.5,-39.5) to[quadratic={(-99.5,-39.5)}] (-92.46,-43.57) to[quadratic={(-85.41,-47.63)}] (-105.4,-47.63);
  \draw[red,very thick] (106.4,-31.37) -- (26.41,-31.37) to[quadratic={(6.41,-31.37)}] (13.46,-35.43) to[quadratic={(20.5,-39.5)}] (40.5,-39.5) -- (120.5,-39.5);
}

\makeatletter

\DeclareFontFamily{OT1}{lzc}{}
\DeclareFontShape{OT1}{lzc}{m}{it}{<->[1.15]pzcmi}{}
\DeclareMathAlphabet{\mathlzc}{OT1}{lzc}{m}{it}

\DeclareFontFamily{U}{mathx}{\hyphenchar\font45}
\DeclareFontShape{U}{mathx}{m}{n}{
      <5> <6> <7> <8> <9> <10>
      <10.95> <12> <14.4> <17.28> <20.74> <24.88>
      mathx10
      }{}
\DeclareSymbolFont{mathx}{U}{mathx}{m}{n}
\DeclareFontSubstitution{U}{mathx}{m}{n}

\newcommand\RomNum[1]{%
  \expandafter\uppercase\expandafter{\romannumeral#1}}

\newcommand\RMove[1]{\mathrm R_{\mathrm{\RomNum{#1}}}}

\ifdefined\widecheck\else
\DeclareMathAccent{\widecheck}{0}{mathx}{"71}
\fi

\newcommand{\Cone}{\operatorname{Cone}}

\newcommand\mnphantom@impl[2]{%
  \sbox0{\ensuremath{#1#2}}\kern-.6\wd0}

\newcommand\mnphantom[1]{\mathpalette\mnphantom@impl{#1}}

\DeclarePairedDelimiterX\KhOf[1]{\lbrack}{\rbrack}{%
  \expandafter\ifx\delimsize\middle%
  \expandafter\@firstoftwo%
  \else%
  \expandafter\@secondoftwo%
  \fi{\mnphantom{\left\lbrack\vphantom{#1}\right.}}{\mnphantom{\delimsize\lbrack}}%
  \delimsize\lbrack\mathopen{}#1\mathclose{}\delimsize\rbrack%
  \expandafter\ifx\delimsize\middle%
  \expandafter\@firstoftwo%
  \else%
  \expandafter\@secondoftwo%
  \fi{\mnphantom{\left.\vphantom{#1}\right\rbrack}}{\mnphantom{\delimsize\rbrack}}
}

\newcommand\Cob{\mathlzc{Cob_2^{\ell}}}

\makeatother

\theoremstyle{plain}
\newtheorem{mainthm}{Main Theorem}

\newtheorem{theorem}{Theorem}[section]
\newtheorem{proposition}[theorem]{Proposition}

\newtheorem{lemma}[theorem]{Lemma}

\theoremstyle{definition}

\theoremstyle{remark}

\numberwithin{equation}{section}
\numberwithin{figure}{section}

\title{On the four-term relation on Khovanov homology}
\date{February 28, 2023}
\author{Noboru Ito and Jun Yoshida}

\begingroup
\renewcommand\abstract[1]{}
\abstract{%
The goal of this paper is to prove a categorified analogue of Kontsevich's 4T relation on Vassiliev derivatives of Khovanov homology. 
This categorification is described as commutativity of a hexagon-shaped diagram of singular links.
Indeed, unwinding the proof of the 4T relation on classical Vassiliev derivatives, one sees that it follows from a relationship between Reidemeister moves and crossing changes.
It turns out that, in the case of Khovanov homology, this relationship becomes commutativity of a diagram consisting of morphisms corresponding to these moves.
The statements and the proofs are given
in terms of cobordisms following Bar-Natan's formulation.
}
\endgroup

\begin{document}
\maketitle

\tableofcontents

\section{Introduction}
\label{sec:intro}
The goal of this paper is to categorify Kontsevich's $4T$ relation on Vassiliev derivatives of Khovanov homology.  

Every knot invariant is extended to singular knots by the Vassiliev skein relation~(\cite{Birman1993, BirmanLin1993, BarNatan1995}):
\begin{equation}
\label{eq:intro:Vassiliev-skein}
v^{(r+1)}\left(\diagSingUp\right)
= v^{(r)}\left(\diagCrossPosUp\right)
- v^{(r)}\left(\diagCrossNegUp\right)
\quad.
\end{equation}
Shirokova and Webster \cite{ShirokovaWebster2007} studied a categorified Vassiliev skein relation on Khovanov-Rozansky homology by computing $\mathrm{Ext}$ in a derived category.
In the case of Khovanov homology, the authors \cite{ItoYoshida2020,ItoYoshida2021} found an explicit description as in the following theorem.

\begin{theorem}[{\cite[Main~Theorems~A, B]{ItoYoshida2021}}]
\label{theo:genus1}
There is a degree-preserving morphism
\begin{equation}
\label{eq:Phi-hat}
\widehat\Phi:\KhOf*{\diagCrossNegUp}\to\KhOf*{\diagCrossPosUp}
\quad,
\end{equation}
of complexes in the category of cobordisms $\Cob$ in \cite{BarNatan2005} which is compatible with Reidemeister moves in the following sense:
\begin{enumerate}[label=\upshape(\arabic*)]
  \item \label{genusone-r12} the diagrams below commute:
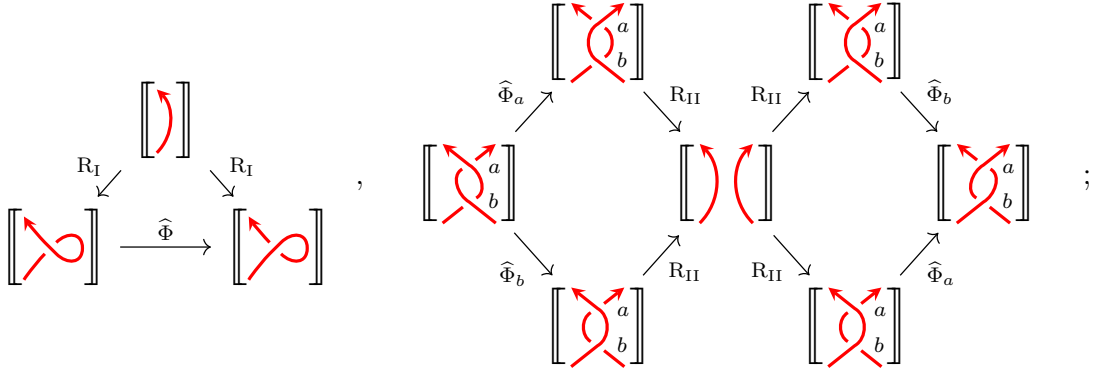
\begin{equation*}
\begin{tikzcd}[column sep=0em, row sep=2ex]
& \KhOf*{\diagFiNil} \ar[dr,"{\RMove1}"] \ar[dl,"{\RMove1}"'] &\\
\KhOf*{\diagFiNeg} \ar[rr,"\widehat\Phi"] && \KhOf*{\diagFiPos}
\end{tikzcd}
\ ,\quad
\begin{tikzcd}[column sep=-.2em,row sep=3ex]
& \KhOf*{\diagRiiRightUpWith{a}{b}} \ar[dr,"{\RMove2}"] && \KhOf*{\diagRiiRightUpWith{a}{b}} \ar[dr,"{\widehat\Phi_b}"] & \\
\KhOf*{\diagRvFTwNegWith{a}{b}} \ar[ur,"{\widehat\Phi_a}"] \ar[dr,"{\widehat\Phi_b}"'] && \KhOf*{\diagRiiParUp} \ar[ur,"{\RMove2}"] \ar[dr,"{\RMove2}"'] && \KhOf*{\diagRvFTwPosWith{a}{b}} \\
& \KhOf*{\diagRiiLeftUpWith{a}{b}} \ar[ur,"{\RMove2}"'] && \KhOf*{\diagRiiLeftUpWith{a}{b}} \ar[ur,"{\widehat\Phi_a}"']
\end{tikzcd}
\quad;
\end{equation*}
\item there are degree $-1$ morphisms $\psi^{}_{\mathsf O}$ and $\psi^{}_{\mathsf U}$ which make the following diagrams chain homotopy commutative:
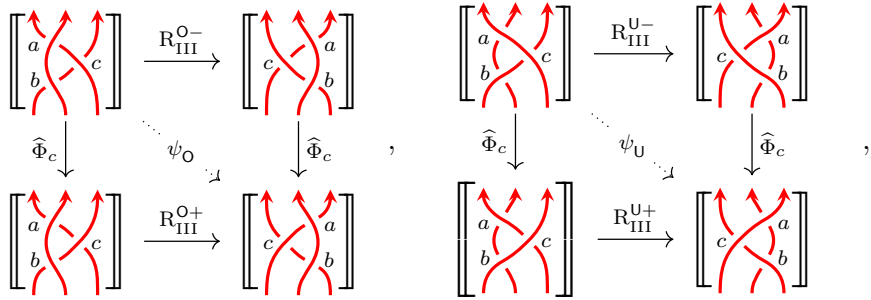
\begin{equation*}
\begin{tikzcd}
\KhOf*{\diagRiiiLeftMLRWith{a}{b}{c}} \ar[r,"{\RMove3^{\mathsf O-}}"] \ar[d,"{\widehat\Phi_c}"'] \ar[dr,dotted,"{\psi^{}_{\mathsf O}}" description] & \KhOf*{\diagRiiiRightMLRWith{a}{b}{c}} \ar[d,"{\widehat\Phi_c}"] \\
\KhOf*{\diagRiiiLeftMRLWith{a}{b}{c}} \ar[r,"{\RMove3^{\mathsf{O}+}}"] & \KhOf*{\diagRiiiRightMRLWith{a}{b}{c}}
\end{tikzcd}
\ ,\quad
\begin{tikzcd}
\KhOf*{\diagRiiiLeftLRMWith{a}{b}{c}} \ar[r,"{\RMove3^{\mathsf U-}}"] \ar[d,"{\widehat\Phi_c}"'] \ar[dr,dotted,"{\psi^{}_{\mathsf U}}" description] & \KhOf*{\diagRiiiRightLRMWith{a}{b}{c}} \ar[d,"{\widehat\Phi_c}"] \\
\KhOf*{\diagRiiiLeftRLMWith{a}{b}{c}} \ar[r,"{\RMove3^{\mathsf{U}+}}"] & \KhOf*{\diagRiiiRightRLMWith{a}{b}{c}}
\end{tikzcd}
\quad,
\end{equation*}
where the morphisms $\RMove3^{\ast\pm}$ are constructed using the ``Kauffman trick'' with respect to the crossings $c$.
\end{enumerate}
\end{theorem}
Taking the mapping cone of the morphism $\widehat\Phi$, we extend Khovanov homology to singular links by
\begin{equation}\label{eq:cat-Vas}
\KhOf*{\diagSingUp}
\cong
\Cone\left(\KhOf*{\diagCrossNegUp}\xrightarrow{\widehat\Phi}\KhOf*{\diagCrossPosUp}\right)
\quad, 
\end{equation}
which is none other than a categorified version of Vassiliev skein relation~\eqref{eq:intro:Vassiliev-skein}.  
Indeed, the compatibility in \cref{theo:genus1} implies the invariance under the following moves: 
\begin{equation}
\label{eq:intro:moveSing}
\diagCrossSingRivOL \xleftrightarrow{\RMove4} \diagCrossSingRivOR
\ ,\quad
\diagCrossSingRivUL \xleftrightarrow{\RMove4'} \diagCrossSingRivUR
\ ,\quad
\diagRvSingU \xleftrightarrow{\RMove5} \diagRvSingD
\quad.
\end{equation}

In the case of polynomial knot invariants, say $v$, its Vassiliev derivatives satisfy the relations
\begin{gather}
v^{(r)}\left(\diagFiSing\right) = 0
\quad, \label{eq:FI-rel}\\
v^{(r)}\Bigl(\scalebox{.9}{\diagFourTRPD}\Bigr)-v^{(r)}\Bigl(\scalebox{.9}{\diagFourTLDN}\Bigr)
=  v^{(r)}\Bigl(\scalebox{.9}{\diagFourTRDP}\Bigr) - v^{(r)}\Bigl(\scalebox{.9}{\diagFourTLND}\Bigr)
\quad, \label{eq:4T-rel}
\end{gather}
which are respectively called the $FI$ and the $4T$ relations.
We next aim at finding categorified analogues of them.
We note that, besides the obvious proof of these relations, they are also understood in view of the homotopy theory of the space $\mathcal M$ of immersions $S^1\to\mathbb R^3$.
Vassiliev~\cite{Vassiliev1990} introduced a stratification $\mathcal M=\bigcup_i\mathcal M_i$ such that
\begin{itemize}
  \item $\mathcal M_0$ consists of smooth embeddings;
  \item $\mathcal M_1$ consists of smooth immersions with exactly one double point;
  \item $\mathcal M_2$ consists of
\begin{enumerate}[label=\upshape(\alph*)]
  \item\label{M2:crit} smooth injections with exactly one critical point and
  \item smooth immersion with exactly two double points;
\end{enumerate}
\item $\mathcal M_3$ contains immersions with triple points.
\end{itemize}
In this point of view, a crossing change is thought of as a ``wall-crossing'' passing through the codimension $1$ staratum $\mathcal M_1$.
Hence, the $FI$ relation~\eqref{eq:FI-rel} corresponds to the monodromy around the stratum $\mathcal M_2$ as in \cref{fig:FI-Vassiliev}.
\begin{figure}[tbp]
\centering
\begin{tikzpicture}
\node[below] (N) at (0,-1) {\diagFiNeg};
\node[above] (P) at (0,1) {\diagFiPos};
\node[right] (Cr) at (1,0) {\diagFiCrit};
\node[right] (Non) at (3,0) {\diagFiNil};
\draw[red] (-2,0) node[left]{\diagFiSing} -- node[below left] {$\mathcal M_1$} (1,0);
\fill[blue] (1,0) circle (.15) node[below]{$\mathcal M_2$};
\draw[thick,dotted] (N) to[bend right] (Non);
\draw[thick,dotted,-stealth] (Non) to[bend right] (P);
\draw[thick,dotted,-stealth] (N) to[bend left] (P);
\end{tikzpicture}
\caption{The $FI$~relation in Vassiliev theory}
\label{fig:FI-Vassiliev}
\end{figure}
Similarly, the $4T$ relation~\eqref{eq:4T-rel} corresponds to the obstruction dual to the hexagon prism surrounding the triple point singularity in $\mathcal M_3$ as in \cref{fig:hexprism}.
\begin{figure}
\centering
\begin{tikzpicture}[xscale=3,yscale=2.5]
\fill[cyan,fill opacity=0.2] (.41,2.04) -- (-.41,.36) -- (-.41,-2.04) -- (.41,-.36) -- cycle;
\fill[magenta,fill opacity=0.2] (-1.31,1.64) -- (1.31,.76) -- (1.31,-1.64) -- (-1.31,-.76) -- cycle;
\fill[orange,fill opacity=0.2] (1.86,.84) -- (-1.14,.84) -- (-1.86,-.84) -- (1.14,-.84) -- cycle;
\draw[red,thick] (0,1.2) -- (0,1);
\draw[red,thick,dotted] (0,1) -- (0,-1.2);
\draw[red,thick] (-1.31,.44) -- (-1,.33) (.71,-.24) -- (1.31,-0.44);
\draw[red,thick,dotted] (-1,.33) -- (0,0) -- (.71,-.24);
\draw[red,thick,dotted] (.41,.84) -- (0,0) (0,0) -- (-.3,-.61);
\draw[red,thick] (-.3,-.61) -- (-.41,-.84);
\fill[violet] (0,0) circle(.05);
\node[circle,inner sep=1,fill=blue] (rLRM) at (1,1) {};
\node[circle,inner sep=1,fill=blue] (lLRM) at (.76,1.61) {};
\node[circle,inner sep=1,fill=blue] (lLMR) at (-.24,1.61) {};
\node[circle,inner sep=1,fill=blue] (lMLR) at (-1,1) {};
\node[circle,inner sep=1,fill=blue] (rMLR) at (-.76,.39) {};
\node[circle,inner sep=1,fill=blue] (rDelta) at (.24,.39) {};
\node[circle,inner sep=1,fill=magenta] (lPau) at (-.71,1.24) {};
\node[circle,inner sep=1,fill=cyan] (lPbu) at (.3,1.61) {};
\node[circle,inner sep=1,fill=magenta] (rPbu) at (.71,.76) {};
\node[circle,inner sep=1,fill=cyan] (rPau) at (-.3,.39) {};
\node[circle,inner sep=1,fill=blue] (rRLM) at (1,-1) {};
\node[circle,inner sep=1,fill=blue] (lRLM) at (.76,-.39) {};
\node[circle,inner sep=1,fill=blue] (lDelta) at (-.24,-.39) {};
\node[circle,inner sep=1,fill=blue] (lMRL) at (-1,-1) {};
\node[circle,inner sep=1,fill=blue] (rMRL) at (-.76,-1.61) {};
\node[circle,inner sep=1,fill=blue] (rRML) at (.24,-1.61) {};
\node[circle,inner sep=1,fill=magenta] (lPad) at (-.71,-.76) {};
\node[circle,inner sep=1,fill=cyan] (lPbd) at (.3,-.39) {};
\node[circle,inner sep=1,fill=magenta] (rPbd) at (.71,-1.24) {};
\node[circle,inner sep=1,fill=cyan] (rPad) at (-.3,-1.61) {};
\draw (lPau) -- (lMLR) -- (rMLR) -- (-.4,.39) (rPau) -- (rDelta) -- (rPbu);
\draw[densely dotted] (-.4,.39) -- (rPau) (.2,1.61) -- (lPbu);
\draw[densely dotted] (lPau) -- (-.52,1.39) (rPbu) -- (.88,.9);
\draw (-.52,1.39) -- (lLMR) (.88,.9) -- (rLRM) -- (lLRM) -- (lPbu) (.2,1.61) -- (lLMR);
\draw[densely dotted] (lMRL) -- (lPad) -- (lDelta) -- (lPbd) -- (lRLM) -- (rRLM) -- (rPbd);
\draw (lMRL) -- (rMRL) -- (-.41,-1.61) (rPad) -- (rRML) -- (rPbd);
\draw[densely dotted] (-.41,-1.61) -- (rPad);;
\path (rLRM) -- node[midway,circle,inner sep=1,fill=orange](rXXM){} (rRLM);
\path (lLRM) -- node[midway,circle,inner sep=1,fill=orange](lXXM){} (lRLM);
\path (lLMR) -- node[midway,circle,inner sep=1,fill=orange](lXMX){} (lDelta);
\path (lMLR) -- node[midway,circle,inner sep=1,fill=orange](lMXX){} (lMRL);
\path (rMLR) -- node[midway,circle,inner sep=1,fill=orange](rMXX){} (rMRL);
\path (rDelta) -- node[midway,circle,inner sep=1,fill=orange](rXMX){} (rRML);
\draw (rLRM) -- (1,.86);
\draw[densely dotted] (1,.86) -- (rXXM) -- (1,-.84);
\draw[densely dotted] (1,-.84) -- (rRLM);
\draw[densely dotted] (lLRM) -- (lXXM) -- (lRLM);
\draw[densely dotted] (lLMR) -- (lXMX) -- (lDelta);
\draw (lMLR) -- (lMXX) (-1,-.84) -- (lMRL);
\draw[densely dotted] (lMXX) -- (-1,-.84);
\draw (rMLR) -- (rMXX) (-.76,-.84) -- (rMRL);
\draw[densely dotted] (rMXX) -- (-.76,-.84);
\draw (rDelta) -- (rXMX) (.24,-.84) -- (rRML);
\draw[densely dotted] (rXMX) -- (.24,-.84);
\node[above] at (lLMR) {\scalebox{.8}{\diagRiiiLeftLMR}};
\node[above] at (lLRM) {\scalebox{.8}{\diagRiiiLeftLRM}};
\node[above right] at (rLRM) {\scalebox{.8}{\diagRiiiRightLRM}};
\node[left=3] at (lMLR) {\scalebox{.8}{\diagRiiiLeftMLR}};
\node[below] at (rRML) {\scalebox{.8}{\diagRiiiRightRML}};
\node[below] at (rMRL) {\scalebox{.8}{\diagRiiiRightMRL}};
\node[below left] at (lMRL) {\scalebox{.8}{\diagRiiiLeftMRL}};
\node[right=3] at (rRLM) {\scalebox{.8}{\diagRiiiRightRLM}};
\end{tikzpicture}
\caption{Hexagon prism diagram}
\label{fig:hexprism}
\end{figure}
In fact, the $4T$ relation \eqref{eq:4T-rel} is proved by taking the alternating sum of the twelve diagrams in \cref{fig:hexprism}.
Kontsevich~\cite{Kontsevich1993} showed that the $FI$ and the $4T$ relation are essential for \emph{finite type invariants}: for the space $V_m$ of the order $m$ Vassiliev invariants, he gave an explicit description of the quotient $V_m/V_{m-1}$ in terms of these relations.

Our goal is to categorify this framework.
\Cref{theo:genus1}-\ref{genusone-r12} implies 
\[
\KhOf*{\diagFiSing} \simeq 0
\quad, 
\]
which is none other than \emph{categorified $FI$ relation}.
The next target is the $4T$ relation.  

\begin{mainthm}\label{mainTHM}
With regard to the morphism $\widehat\Phi$ and the chain homotopies $\psi^{}_{\mathsf O},\psi^{}_{\mathsf U}$ in \cref{theo:genus1}, the following hold.
\begin{enumerate}[label=\upshape(\arabic*)]
\item\label{sub:mainTHM:hexagon} The following diagrams commute strictly:
\begin{equation}
\label{eq:hexagon-cross}
\begin{tikzcd}[column sep={3.1em,between origins}, row sep={5pc,between origins}]
& \KhOf*{\diagRiiiLeftLMRWith{a}{b}{c}} \ar[rr,"\widehat\Phi_b"] \ar[dl,"\widehat\Phi_a"'] && \KhOf*{\diagRiiiLeftLRMWith{a}{b}{c}} \ar[dr,"{\RMove3^{\mathsf U-}}"] & \\
\KhOf*{\diagRiiiLeftMLRWith{a}{b}{c}} \ar[dr,"{\RMove3^{\mathsf O-}}"'] &&&& \KhOf*{\diagRiiiRightLRMWith{a}{b}{c}} \ar[dl,"\widehat\Phi_b"] \\
& \KhOf*{\diagRiiiRightMLRWith{a}{b}{c}} \ar[rr,"\widehat\Phi_a"] && \KhOf*{\diagDeltaBrDRightWith{a}{b}{c}} &
\end{tikzcd}
\,,\;
\begin{tikzcd}[column sep={3.1em,between origins}, row sep={5pc,between origins}]
& \KhOf*{\diagDeltaBrDLeftWith{a}{b}{c}} \ar[rr,"\widehat\Phi_b"] \ar[dl,"\widehat\Phi_a"'] && \KhOf*{\diagRiiiLeftRLMWith{a}{b}{c}} \ar[dr,"{\RMove3^{\mathsf U+}}"] & \\
\KhOf*{\diagRiiiLeftMRLWith{a}{b}{c}} \ar[dr,"{\RMove3^{\mathsf O+}}"'] &&&& \KhOf*{\diagRiiiRightRLMWith{a}{b}{c}} \ar[dl,"\widehat\Phi_b"] \\
& \KhOf*{\diagRiiiRightMRLWith{a}{b}{c}} \ar[rr,"\widehat\Phi_a"] && \KhOf*{\diagRiiiRightRMLWith{a}{b}{c}} &
\end{tikzcd}
\end{equation}
\item\label{sub:mainTHM:hexprism} In the module of morphisms of degree $-1$, the following equation holds:
\begin{equation}\label{eq:hexagon-prism}
\widehat\Phi_a\psi^{}_{\mathsf O}\widehat\Phi_a -  \widehat\Phi_b\psi^{}_{\mathsf U}\widehat\Phi_b = 0
\quad.
\end{equation}
\end{enumerate}
\end{mainthm}

The left hand side of the equation~\eqref{eq:hexagon-prism} is a $-1$-cocycle in a $\mathrm{Hom}$ complex associated with the boundary of the hexagon prism in \cref{fig:hexprism}, or the \emph{$2$-monodromy} around the triple point.
Thus, the part~\ref{sub:mainTHM:hexprism} asserts that it is actually a trivial cycle.
As a consequence, we obtain the following categorified analogue of $4T$-relation:

\begin{mainthm}\label{cat-4T}
There are morphisms
\[
\KhOf*{\diagFourTLND}\to\KhOf*{\diagFourTRDP}
\ ,\quad
\KhOf*{\diagFourTLDN}\to\KhOf*{\diagFourTRPD}
\]
whose mapping cones are chain homotopic to each other.
\end{mainthm}

Indeed, by taking the mapping cones along the vertical edges of the hexagon prism diagram~\eqref{fig:hexprism}, we obtain the following (strictly) commutative diagram:
\begin{equation}
\label{eq:hexagon}
\begin{tikzcd}[column sep={3.5em,between origins}, row sep={5pc,between origins}]
& \KhOf*{\diagCrossSingRivLNNWith{a}{b}} \ar[rr,"\widehat\Phi_b"] \ar[dl,"\widehat\Phi_a"'] && \KhOf*{\diagCrossSingRivULWith{a}{b}} \ar[dr,"{\RMove4^{\mathsf U}}"] & \\
\KhOf*{\diagCrossSingRivOLWith{a}{b}} \ar[dr,"{\RMove4^{\mathsf O}}"'] &&&& \KhOf*{\diagCrossSingRivURWith{a}{b}} \ar[dl,"\widehat\Phi_b"] \\
& \KhOf*{\diagCrossSingRivORWith{a}{b}} \ar[rr,"\widehat\Phi_a"] && \KhOf*{\diagCrossSingRivRPPWith{a}{b}} &
\end{tikzcd}
\quad,
\end{equation}
where $\RMove4^{\mathsf O},\RMove4^{\mathsf U}$ are the chain homotopy equivalences associated with the moves~\eqref{eq:intro:moveSing}.
The two morphisms in \cref{cat-4T} are induced by taking mapping cones of \eqref{eq:hexagon} in two different directions.

We note that the equation~\eqref{eq:hexagon-prism} suggests studying an obstruction associated with a triple point in the stratum $\mathcal M_3$ for general knot homologies.
Indeed, suppose $H(K)$ is a knot homology equipped with a crossing-change morphism $\widehat\Phi$ which is compatible with Reidemeister moves in the sense of Theorem~\ref{theo:genus1} (possibly up to homotopy).
In this case, one can prove that the two hexagon diagrams~\eqref{eq:hexagon-cross} are commutative up to chain homotopies, say $\alpha_+$ and $\alpha_-$ with
\[
\partial(\alpha_-)=\widehat\Phi_a\RMove3^{\mathsf O-}\widehat\Phi_a - \widehat\Phi_b\RMove3^{\mathsf U-}\widehat\Phi_b
\ ,\quad
\partial(\alpha_+)=\widehat\Phi_a\RMove3^{\mathsf O+}\widehat\Phi_a - \widehat\Phi_b\RMove3^{\mathsf U+}\widehat\Phi_b
\quad.
\]
For example, the part~\ref{sub:mainTHM:hexagon} in \cref{mainTHM} asserts that we can take $\alpha_\pm=0$.
They together with the homotopies $\psi_{\mathsf O}$ and $\psi_{\mathsf U}$ in \cref{theo:genus1} yield a $-1$-cocycle
\[
\theta\coloneqq\widehat\Phi_a\psi_{\mathsf O}\widehat\Phi_a - \widehat\Phi_b\psi_{\mathsf U}\widehat\Phi_b + \widehat\Phi_c\alpha_- - \alpha_+\widehat\Phi_c
\]
in the $\mathrm{Hom}$-complex, which corresponds to the $2$-cycle generated by the faces of the hexagon prism~\cref{fig:hexprism}.
The part~\ref{sub:mainTHM:hexprism} in \cref{mainTHM} then implies that we have $\theta=0$ in the case of Khovanov homology.


\section{Review on crossing-change on Khovanov homology}
\label{sec:review}

We begin with a review on Khovanov homology and the genus-one morphism $\widehat\Phi$.

\subsection{Definitions and notations}

We construct Khovanov homology following Bar-Natan's formalism~\cite{BarNatan2005}.
By \emph{cobordisms}, we mean $2$-dimensional cobordisms with corners.
In this paper, we always use the ``left-to-right'' convention to depict them; e.g.~
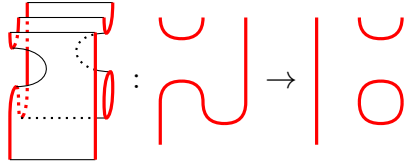
\begin{equation*}
\begin{tikzpicture}[xlen=.5pt,ylen=-.5pt,scale=.8]
\draw[red,very thick] (-73.59,48.63) -- (-33.59,48.63) to[quadratic={(-13.59,48.63)}] (-6.545,44.57) to[quadratic={(-4.71,43.51)}] (-4.71,42.73);
\draw[red,very thick,dotted] (-4.71,42.73) to[quadratic={(-4.71,40.5)}] (-19.5,40.5) to[quadratic={(-39.5,40.5)}] (-32.46,36.43) to[quadratic={(-25.41,32.37)}] (-5.41,32.37) -- (-3.76,32.37);
\draw[red,very thick] (-3.76,32.37) -- (29.86,32.37);
\draw[red,very thick,dotted] (29.86,32.37) -- (60.5,32.37);
\draw[red,very thick] (60.5,32.37) -- (74.59,32.37);
\draw[red,very thick] (46.41,48.63) to[quadratic={(31.62,48.63)}] (31.62,46.41);
\draw[red,very thick,dotted] (31.62,46.41) to[quadratic={(31.62,45.63)}] (33.46,44.57) to[quadratic={(37.64,42.15)}] (46.41,41.17);
\draw[red,very thick] (46.41,41.17) to[quadratic={(52.39,40.5)}] (60.5,40.5);
\draw[black] (-4.71,42.72) .. controls+(0,-40)and+(0,-40) .. (31.62,46.41);
\draw[black,thick,dotted] (-34.29,38.28) -- (-34.29,-31.37);
\draw[black] (-34.29,-31.37) -- (-34.29,-41.72);
\draw[black] (9.38,-45.41) .. controls(9.38,-40.08)and(10.03,-35.4) .. (11.14,-31.37);
\draw[black,thick,dotted] (11.14,-31.37) .. controls(17.61,-8.11)and(39.9,-7) .. (44.76,-31.37);
\draw[black] (44.76,-31.37) .. controls(45.37,-34.42)and(45.71,-37.87) .. (45.71,-41.72);
\draw[black] (-73.59,48.63) -- (-73.59,-31.37);
\draw[black] (46.41,48.63) -- (46.41,-31.37);
\draw[black] (60.5,40.5) -- (60.5,-39.5);
\draw[black] (74.59,32.37) -- (74.59,-47.63);
\draw[red,very thick] (-73.59,-31.37) -- (46.41,-31.37);
\draw[red,very thick] (-32.46,-43.57) to[quadratic={(-39.5,-39.5)}] (-19.5,-39.5) to[quadratic={(0.5,-39.5)}] (7.545,-43.57) to[quadratic={(14.59,-47.63)}] (-5.41,-47.63) to[quadratic={(-25.41,-47.63)}] (-32.46,-43.57);
\draw[red,very thick] (60.5,-39.5) to[quadratic={(40.5,-39.5)}] (47.54,-43.57) to[quadratic={(54.59,-47.63)}] (74.59,-47.63);
\end{tikzpicture}
\;:\;
\begin{tikzpicture}[xlen=.5pt,ylen=-.5pt,scale=.8]
\draw[red,very thick] (-59.5,40.5) -- (-19.5,40.5) to[quadratic={(0.5,40.5)}] (0.5,20.5) to[quadratic={(0.5,0.5)}] (-19.5,0.5) to[quadratic={(-39.5,0.5)}] (-39.5,-19.5) to[quadratic={(-39.5,-39.5)}] (-19.5,-39.5) -- (60.5,-39.5);
\draw[red,very thick] (60.5,40.5) to[quadratic={(40.5,40.5)}] (40.5,20.5) to[quadratic={(40.5,0.5)}] (60.5,0.5);
\end{tikzpicture}
\;\to\;
\begin{tikzpicture}[xlen=.5pt,ylen=-.5pt,scale=.8]
\draw[red,very thick] (-59.5,40.5) -- (60.5,40.5);
\draw[red,very thick] (-39.5,-19.5) to[quadratic={(-39.5,0.5)}] (-19.5,0.5) to[quadratic={(0.5,0.5)}] (0.5,-19.5) to[quadratic={(0.5,-39.5)}] (-19.5,-39.5) to[quadratic={(-39.5,-39.5)}] (-39.5,-19.5);
\draw[red,very thick] (60.5,0.5) to[quadratic={(40.5,0.5)}] (40.5,-19.5) to[quadratic={(40.5,-39.5)}] (60.5,-39.5);
\end{tikzpicture}
\end{equation*}
As an unfortunate consequence, the composition $W_1 \circ W_0$ of cobordisms $W_0$ and $W_1$ is, to the contrary to its notation, represented by gluing the right boundary of $W_0$ to the left boundary of $W_1$.
Hence, we will not omit the composition operator ``$\circ$'' in composing ``pictures'' to clarify the composition order throughout the paper.

We define $\Cob$ to be the additive closure of the category of formal sums of cobordisms modulo the three relations called $S$-, $T$-, and $4\!\mathit{Tu}$-relations (see~\cite{BarNatan2005}).
Although a morphism of $\Cob$ is hence a matrix of cobordisms, we often denote it by the formal sum of all its entries in the case where the domains and codomains are understood from the pictures.
For example, the morphism
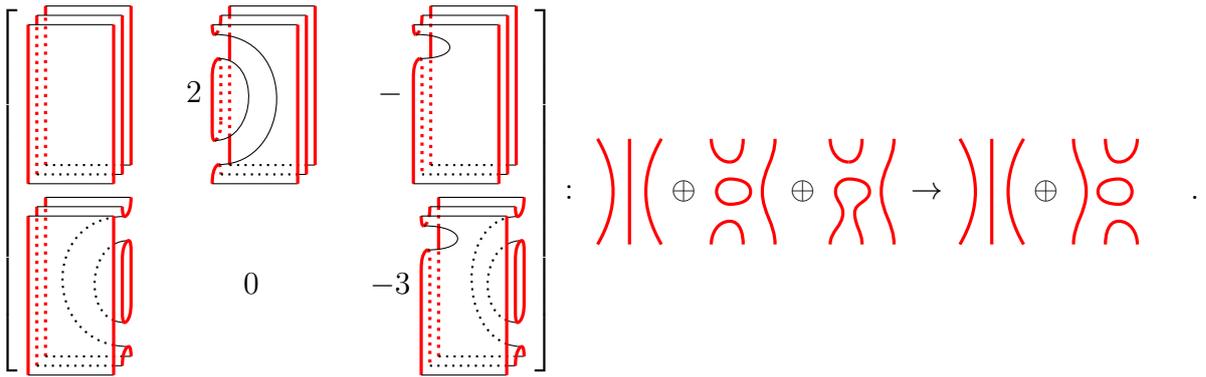
\begin{equation*}
\begin{tikzcd}
\begin{bmatrix}
\;\BordTriId &  2\;\BordRTwoLCounit & -\;\BordCompGFst\; \\[6ex]
\;\BordRTwoBarRUnit & 0 & -3\;\BordCompGSnd\;
\end{bmatrix}
:\:
\diagTriviii \!\oplus \diagRiiiLeftUParWWV \oplus \diagRiiiLeftUParWVV \to \diagTriviii \!\oplus \diagRiiiRightUParWWV
\end{tikzcd}
\quad.  
\end{equation*}
in $\Cob$ is written as the following linear sum:
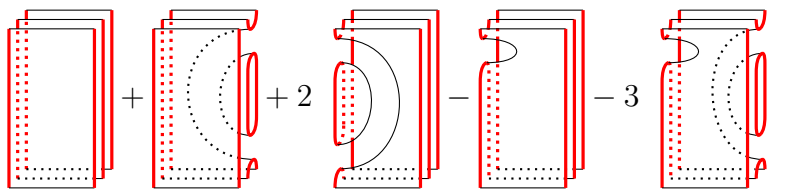
\begin{equation*}
\BordTriId
+\BordRTwoBarRUnit
+ 2~~\BordRTwoLCounit - \BordCompGFst - 3~~\BordCompGSnd
\quad.
\end{equation*}

For each tangle diagram $D$, we construct $\KhOf{D}$ as a complex in  $\Cob$.  We consider the following single saddle operations $\delta_-$ and $\delta_+$:
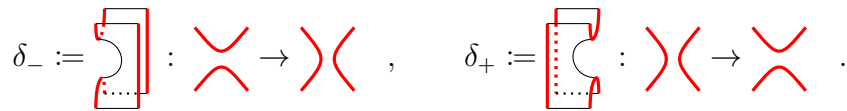
\[
\delta_-\coloneqq\BordDeltaN
\;:\;
\diagSmoothH\to \diagSmoothV
\quad,\qquad
\delta_+\coloneqq\BordDeltaP
\;:\;
\diagSmoothV
\to \diagSmoothH
\quad.
\]
The complexes $\KhOf*{\diagCrossNegUp}$ and $\KhOf*{\diagCrossPosUp}$ are defined as follows:
\begin{equation*}
\begin{tikzcd}[column sep=1.75em,row sep=0ex]
\KhOf*{\diagCrossNegUp}\;:\quad\cdots \ar[r] & 0 \ar[r] & \ar[r] \overset{-1}{\diagSmoothH} \ar[r,"\delta_-"] & \overset{0}{\diagSmoothV} \ar[r] & \overset{1}{\vphantom{\diagSmoothV}0} \ar[r] & 0 \ar[r] & \cdots \\
\KhOf*{\diagCrossPosUp}\;:\quad\cdots \ar[r] & 0 \ar[r]& 0 \ar[r] & \diagSmoothV \ar[r,"-\delta_+"] & \diagSmoothH \ar[r] & 0 \ar[r] & \cdots
\end{tikzcd}
\quad.
\end{equation*}
For general tangle diagram $D$, we define $\KhOf{D}$ by ``stacking up'' the complexes for all the crossings in $D$.
Bar-Natan~\cite{BarNatan2005} proved that the resulting complex is invariant under Reidemeister moves and that it induces Khovanov homology for link diagrams.
We call it the \emph{Khovanov complex} of $D$.

In \cite{ItoYoshida2021}, the following degree-preserving crossing-change is discussed:
\begin{equation}\label{eq:Phi}
\widehat{\Phi}:\KhOf*{\diagCrossNegUp}\to \KhOf*{\diagCrossPosUp}
\quad.
\end{equation}
We first consider the morphism $\Phi$ in the category $\Cob$ given by
\begin{equation}\label{eq:Phi-def}
\Phi\coloneqq \BordPhiFst - \BordPhiSnd:\diagSmoothV \to \diagSmoothV
\quad,
\end{equation}
where annuluses are attached to the identity cobordism in two different ways.
The morphism $\widehat\Phi$ is defined by the following diagram:
\begin{equation}\label{eq:Phihat-def}
\begin{tikzcd}[column sep=1.75em]
\KhOf*{\diagCrossNegUp}\mathrlap{\;:} \ar[d,"\widehat\Phi"'] & \mathllap{\cdots} \ar[r] & 0 \ar[r] \ar[d] & \overset{-1}{\diagSmoothH} \ar[r,"\delta_-"] \ar[d] & \overset{0}{\diagSmoothV} \ar[r] \ar[d,"\Phi"] & \overset{1}{\vphantom{\diagSmoothV}0} \ar[r] \ar[d] & 0 \ar[r] \ar[d] & \cdots \\
\KhOf*{\diagCrossPosUp}\mathrlap{\;:} & \mathllap{\cdots} \ar[r] & 0 \ar[r] & 0 \ar[r] & \diagSmoothV \ar[r,"-\delta_+"] & \diagSmoothH \ar[r] & 0 \ar[r] & \cdots
\end{tikzcd}
\quad.
\end{equation}
In view of the Vassiliev skein relation, we extend Khovanov complex to singular tangle diagrams by setting $\KhOf*{\diagSingUp}$ to be the mapping cone of $\widehat\Phi$; more explicitly, it is the complex as follows:
\[
\KhOf*{\diagSingUp}\;:\quad
\cdots\to 0
\to \overset{-2}{\diagSmoothH}
\xrightarrow{-\delta_-} \overset{-1}{\diagSmoothV}
\xrightarrow{\Phi} \overset{0}{\diagSmoothV}
\xrightarrow{-\delta_+} \overset{1}{\diagSmoothH}
\to 0\to{\scriptstyle\cdots}
\quad.
\]

\begin{theorem}[\cite{ItoYoshida2021}]
\label{main:UKH-sing}
The complex $\KhOf{D}$ is invariant under moves of singular tangles up to chain homotopy equivalences.
\end{theorem}

\subsection{The homotopy equivalence for \texorpdfstring{$\RMove3$}{RIII}}
\label{sec:review:R3}

We next recall how the invariance under $\RMove3$ was proved in \cite{BarNatan2005}, where a technique so-called ``Kauffman trick'' was used.
Indeed, by the definition of Khovanov complex, one has the following identifications:
\begin{equation}
\label{eq:R3O-cones}
\begin{gathered}
\KhOf*{\diagRiiiLeftMLR} \cong \operatorname{Cone}\left(\KhOf*{\diagCrossHRivOL} \xrightarrow{\delta_-} \KhOf*{\diagCrossVRivOL}\right)
\quad,\\
\KhOf*{\diagRiiiRightMLR} \cong \operatorname{Cone}\left(\KhOf*{\diagCrossHRivOR} \xrightarrow{\delta_-} \KhOf*{\diagCrossVRivOR}\right)
\quad,\\
\KhOf*{\diagRiiiLeftMRL} \cong \operatorname{Cone}\left(\KhOf*{\diagCrossVRivOL} \xrightarrow{\delta_+} \KhOf*{\diagCrossHRivOL}\right)[1]
\quad,\\
\KhOf*{\diagRiiiRightMRL} \cong \operatorname{Cone}\left(\KhOf*{\diagCrossVRivOR} \xrightarrow{\delta_+} \KhOf*{\diagCrossHRivOR}\right)[1]
\quad.
\end{gathered}
\end{equation}
Notice that, in each pair of the moves, the domains and the codomains are identical or connected by two Reidemeister moves of type II.
Unwinding the chain homotopy equivalences associated with those moves (see \cite{BarNatan2005}), one sees that we have the commutative squares in the category of complexes in $\Cob$:
\begin{equation}
\label{eq:R3O-conecubes}
\begin{tikzcd}
\KhOf*{\diagCrossHRivOL} \ar[d,"\gamma^{}_{\mathsf O}"'] \ar[r,"\delta_-"] \ar[dr,dotted,"F_{\mathsf O-}" description] & \KhOf*{\diagCrossVRivOL} \ar[d,"-\omega^{}_{\mathsf O}"] \\
\KhOf*{\diagCrossHRivOR} \ar[r,"\delta_-"'] & \KhOf*{\diagCrossVRivOR}
\end{tikzcd}
\ ,\quad
\begin{tikzcd}
\KhOf*{\diagCrossVRivOL} \ar[d,"\omega^{}_{\mathsf O}"'] \ar[r,"\delta_+"] \ar[dr,dotted,"F_{\mathsf O+}" description] & \KhOf*{\diagCrossHRivOL} \ar[d,"\gamma^{}_{\mathsf O}"] \\
\KhOf*{\diagCrossVRivOR} \ar[r,"\delta_+"'] & \KhOf*{\diagCrossHRivOR}
\end{tikzcd}
\quad,
\end{equation}
where $\gamma^{}_{\mathsf O}=\{\gamma_{\mathsf O}^i\}_i$ and $\omega^{}_{\mathsf O}=\{\omega_{\mathsf O}^i\}_i$ are morphisms of complexes which are zero except
\begin{gather}
\label{eq:def-R3O:gamma}
\gamma_{\mathsf O}^{-1}\coloneqq -\BordEquivOO
\ ,\quad
\gamma_{\mathsf O}^0\coloneqq -\BordEquivOl + \BordEquivlO
\ ,\quad
\gamma_{\mathsf O}^1\coloneqq \BordEquivll
\quad,\\[1ex]
\label{eq:def-R3O:omega}
\omega_{\mathsf O}^0\coloneqq
\BordTriId
+\BordRTwoBarRUnit
+\BordRTwoLCounit
+\BordRTwoLCounitRUnit
\quad,
\end{gather}
while $F_{\mathsf O-}=\{F_{\mathsf O-}^i\}_i$ and $F_{\mathsf O+}=\{F_{\mathsf O+}^i\}_i$ are chain homotopies (i.e.~degree $-1$ elements in the $\mathrm{Hom}$ complexes bounding appropriate elements) which are zero except
\begin{equation}
\label{eq:def-R3O:F}
F_{\mathsf O-}^0\coloneqq \BordCompFSnd
\ ,\quad
F_{\mathsf O-}^1\coloneqq \BordCompFFst
\allowdisplaybreaks[3]
\ ,\quad
F_{\mathsf O+}^0\coloneqq \BordCompHSnd
\ ,\quad
F_{\mathsf O+}^1\coloneqq \BordCompHFst
\quad.
\end{equation}
In fact, one can verify that $F_{\mathsf O-}$ and $F_{\mathsf O+}$ satisfy the equations in the $\mathrm{Hom}$-complexes respectively:
\begin{equation}
\label{eq:FO-eqn}
\partial(F_{\mathsf O-}) = \delta_-\gamma^{}_{\mathsf O} - (-\omega^{}_{\mathsf O})\delta_-
\ ,\quad
\partial(F_{\mathsf O+}) = \delta_+\omega^{}_{\mathsf O} - \gamma^{}_{\mathsf O}\delta_+
\quad,
\end{equation}
so that they exhibit \eqref{eq:R3O-conecubes} as homotopy commutative squares.
We then obtain a chain homotopy equivalence
\begin{equation}\label{eq:R3O}
\RMove3^{\mathsf O-}:\KhOf*{\diagRiiiLeftMLR} \xrightarrow\simeq \KhOf*{\diagRiiiRightMLR}
\ ,\quad
\RMove3^{\mathsf O+}:\KhOf*{\diagRiiiLeftMRL} \xrightarrow\simeq \KhOf*{\diagRiiiRightMRL}
\quad.
\end{equation}
More explicitly, under the identifications \eqref{eq:R3O-cones}, the morphisms $\RMove3^{\mathsf O-}$ and $\RMove3^{\mathsf O+}$ in \eqref{eq:R3O} are respectively represented by the matrices below:
\begin{equation}
\label{eq:R3O-matrix}
\RMove3^{\mathsf O-} =
\begin{bmatrix}
-\omega^{}_{\mathsf O} & -F_{\mathsf O-} \\ 0 & \gamma^{}_{\mathsf O}
\end{bmatrix}
\ ,\quad
\RMove3^{\mathsf O+} =
\begin{bmatrix}
\gamma^{}_{\mathsf O} & -F_{\mathsf O_+} \\ 0 & \omega^{}_{\mathsf O}
\end{bmatrix}
\quad.
\end{equation}
Similarly, the chain homotopy equivalences
\begin{equation}\label{eq:R3U}
\RMove3^{\mathsf U-}:\KhOf*{\diagRiiiLeftLRM} \xrightarrow\simeq \KhOf*{\diagRiiiRightLRM}
\ ,\quad
\RMove3^{\mathsf U+}:\KhOf*{\diagRiiiLeftRLM} \xrightarrow\simeq \KhOf*{\diagRiiiRightRLM}
\end{equation}
are obtained by the following homotopy commutative squares:
\begin{equation}
\begin{tikzcd}
\KhOf*{\diagCrossHRivUL} \ar[d,"\gamma_{\mathsf U}"'] \ar[r,"\delta_-"] \ar[dr,dotted,"F_{\mathsf U-}" description] & \KhOf*{\diagCrossVRivUL} \ar[d,"-\omega^{}_{\mathsf U}"] \\
\KhOf*{\diagCrossHRivUR} \ar[r,"\delta_-"'] & \KhOf*{\diagCrossVRivUR}
\end{tikzcd}
\ ,\quad
\begin{tikzcd}
\KhOf*{\diagCrossVRivUL} \ar[d,"\omega^{}_{\mathsf U}"'] \ar[r,"\delta_+"] \ar[dr,dotted,"F_{\mathsf U+}" description] & \KhOf*{\diagCrossHRivUL} \ar[d,"\gamma_{\mathsf U}"] \\
\KhOf*{\diagCrossVRivUR} \ar[r,"\delta_+"'] & \KhOf*{\diagCrossHRivUR}
\end{tikzcd}
\quad,
\end{equation}
where $\gamma_{\mathsf U}$ and $\omega^{}_{\mathsf U}$ are morphisms of chain complexes in $\Cob$ given by
\begin{gather}
\label{eq:def-R3U:gamma}
\gamma_{\mathsf U}^{-1}\coloneqq \BordEquivll
\ ,\quad
\gamma_{\mathsf U}^0\coloneqq -\BordEquivOl + \BordEquivlO
\ ,\quad
\gamma_{\mathsf U}^1\coloneqq -\BordEquivOO
\quad,\\
\label{eq:def-R3U:omega}
\omega_{\mathsf U}^0\coloneqq \BordTriId -\BordRTwoBarRUnit -\BordRTwoLCounit +\BordRTwoLCounitRUnit
\quad,
\end{gather}
while $F_{\mathsf U-}$ and $F_{\mathsf U_+}$ are chain homotopies given by
\begin{gather}
\label{eq:def-R3U:F}
F_{\mathsf U-}^0\coloneqq \BordCompFSndVar
\ ,\quad
F_{\mathsf U-}^1\coloneqq -\BordCompFFstVar
\ ,\quad
F_{\mathsf U+}^0\coloneqq -\BordCompHSndVar
\ , \quad
F_{\mathsf U+}^1\coloneqq \BordCompHFstVar
\quad,
\end{gather}
so that they satisfy the equations
\begin{equation}
\label{eq:FU-eqn}
\partial(F_{\mathsf U-}) = \delta_-\gamma^{}_{\mathsf U} - (-\omega^{}_{\mathsf U})\delta_-
\ ,\quad
\partial(F_{\mathsf U+}) = \delta_+\omega^{}_{\mathsf U} - \gamma^{}_{\mathsf U}\delta_+
\quad.
\end{equation}
The matrix representations of $\RMove3^{\mathsf U-}$ and $\RMove3^{\mathsf U+}$ are then given by
\begin{equation}
\label{eq:R3U-matrix}
\RMove3^{\mathsf U-} =
\begin{bmatrix}
-\omega^{}_{\mathsf U} & -F_{\mathsf U-} \\ 0 & \gamma^{}_{\mathsf U}
\end{bmatrix}
\ ,\quad
\RMove3^{\mathsf U+} =
\begin{bmatrix}
\gamma^{}_{\mathsf U} & -F_{\mathsf U+} \\ 0 & \omega^{}_{\mathsf U}
\end{bmatrix}
\quad.
\end{equation}

\subsection{The compatibility of \texorpdfstring{$\widehat\Phi$}{Φ} with \texorpdfstring{$\RMove3$}{RIII}}
\label{sec:review:R4}

We next discuss the compatibility of the morphism $\widehat\Phi$ in \eqref{eq:Phi-hat} with the chain homotopy equivalences $\RMove3^{\mathsf O\pm}$ and $\RMove3^{\mathsf U\pm}$ in \eqref{eq:R3O} and \eqref{eq:R3U}.
More precisely, we prove the following result.

\begin{proposition}[{\cite[Theorem~4.1]{ItoYoshida2021}}]
\label{prop:R4}
There are degree $-1$ morphisms $\psi_{\mathsf O}$ and $\psi_{\mathsf U}$ which make the following diagrams chain homotopy commutative:
\begin{equation}\label{eq:R4}
\begin{tikzcd}
\KhOf*{\diagRiiiLeftMLRWith{a}{b}{c}} \ar[d,"{\RMove3^{\mathsf O-}}"'] \ar[r,"{\widehat\Phi_c}"] \ar[dr,dotted,"{\psi^{}_{\mathsf O}}" description] & \KhOf*{\diagRiiiLeftMRLWith{a}{b}{c}} \ar[d,"{\RMove3^{\mathsf{O}+}}"] \\
\KhOf*{\diagRiiiRightMLRWith{a}{b}{c}} \ar[r,"{\widehat\Phi_c}"] & \KhOf*{\diagRiiiRightMRLWith{a}{b}{c}}
\end{tikzcd}
\quad,
\begin{tikzcd}
\KhOf*{\diagRiiiLeftLRMWith{a}{b}{c}} \ar[r,"{\widehat\Phi_c}"] \ar[d,"{\RMove3^{\mathsf U-}}"'] \ar[dr,dotted,"{\psi^{}_{\mathsf U}}" description] & \KhOf*{\diagRiiiLeftRLMWith{a}{b}{c}} \ar[d,"{\RMove3^{\mathsf{U}-}}"] \\
\KhOf*{\diagRiiiRightLRMWith{a}{b}{c}} \ar[r,"{\widehat\Phi_c}"] & \KhOf*{\diagRiiiRightRLMWith{a}{b}{c}}
\end{tikzcd}
\quad.
\end{equation}
\end{proposition}

In view of ``Kauffman trick'', we begin with the compatibility of the morphism $\Phi$ in \eqref{eq:Phi}.
By direct computation, one can verify the following.

\begin{lemma}
There is a homotopy commutative square
\begin{equation}
\label{eq:Phi-GO}
\begin{tikzcd}
\KhOf*{\diagCrossVRivOL} \ar[d,"-\omega^{}_{\mathsf O}"'] \ar[r,"\Phi"] \ar[dr,dotted,"G_{\mathsf O}" description] & \KhOf*{\diagCrossVRivOL} \ar[d,"\omega_{\mathsf O}"] \\
\KhOf*{\diagCrossVRivOR} \ar[r,"\Phi"'] & \KhOf*{\diagCrossVRivOR}
\end{tikzcd}
\ ,\quad
\begin{tikzcd}
\KhOf*{\diagCrossVRivUL} \ar[d,"-\omega^{}_{\mathsf U}"'] \ar[r,"\Phi"] \ar[dr,dotted,"G_{\mathsf U}" description] & \KhOf*{\diagCrossVRivUL} \ar[d,"\omega_{\mathsf U}"] \\
\KhOf*{\diagCrossVRivOR} \ar[r,"\Phi"] & \KhOf*{\diagCrossVRivUR}
\end{tikzcd}
\quad,
\end{equation}
where the chain homotopies $G_{\mathsf O}=\{G_{\mathsf O}^i\}_i$ and $G_{\mathsf U}=\{G_{\mathsf U}^i\}_i$ are given by
\begin{gather}
\label{eq:R4:GO}
G_{\mathsf O}^0 \coloneqq -\BordCompGTrd -\BordCompGFth
\ ,\quad
G_{\mathsf O}^1 \coloneqq \BordCompGFst + \BordCompGSnd
\quad,\displaybreak[2]\\[1ex]
\label{eq:R4:GU}
G_{\mathsf U}^0\coloneqq -\BordCompGTrdVar +\BordCompGFthVar
\ ,\quad
G_{\mathsf U}^1\coloneqq \BordCompGFstVar -\BordCompGSndVar
\quad,
\end{gather}
and zero in the other degrees.
More precisely, $G_{\mathsf O}$ and $G_{\mathsf U}$ respectively satisfy the following equations in the $\mathrm{Hom}$-complexes:
\begin{equation}
\label{eq:G-eqn}
\partial(G_{\mathsf O}) = \Phi(-\omega^{}_{\mathsf O}) - \omega^{}_{\mathsf O}\Phi
\ ,\quad
\partial(G_{\mathsf U}) = \Phi(-\omega^{}_{\mathsf U}) - \omega^{}_{\mathsf U}\Phi
\quad.
\end{equation}
\end{lemma}

Now, we have the following diagram of morphisms of complexes in $\Cob$:
\begin{equation}
\label{eq:R4cube-O}
\begin{tikzcd}[row sep=1ex,column sep=1em]
& \KhOf*{\diagCrossHRivOL} \ar[rr,"\delta_-"] \ar[dl] \ar[dd,"\vphantom{X}" description, "\gamma^{}_{\mathsf O}" pos=0.8] && \KhOf*{\diagCrossVRivOL} \ar[rr] \ar[dl,"\Phi"] \ar[dd,"\vphantom{X}" description, "-\omega^{}_{\mathsf O}" pos=0.8] && 0 \ar[dl] \ar[dd] \\
0 \ar[rr] \ar[dd] && \KhOf*{\diagCrossVRivOL} \ar[rr,"-\delta_+" pos=0.8] \ar[dd,"\omega_{\mathsf O}" pos=0.8] && \KhOf*{\diagCrossHRivOL} \ar[dd,"\gamma^{}_{\mathsf O}" pos=0.8] & \\
& \KhOf*{\diagCrossHRivOR} \ar[r,dash] \ar[dl] & {} \ar[r,"\delta_-" pos=0.3] & \KhOf*{\diagCrossVRivOR} \ar[r,dash] \ar[dl,"\Phi"] & {} \ar[r] & 0 \ar[dl] \\
0 \ar[rr] && \KhOf*{\diagCrossVRivOR} \ar[rr,"-\delta_+"] && \KhOf*{\diagCrossHRivOR} &
\end{tikzcd}
\quad,
\end{equation}
where all square faces are strictly commutative or commutative up to chain homotopies $F_{\mathsf O-}$, $G_{\mathsf O}$, and $F_{\mathsf O+}$.
The last ingredients for the chain homotopy $\psi_{\mathsf O}$ in \cref{prop:R4} are the $2$-homotopies bounding the homotopy commutative cubes in \eqref{eq:R4cube-O}; indeed, since every complexes are supported in degrees $-1,0,1$, the cycles of degrees higher than $2$ are trivial for degree reason.
We define the elements
\begin{equation*}
\Psi_{\mathsf O-}\in\mathrm{Hom}^{-2}\biggl(\KhOf*{\diagCrossHRivOL},\KhOf*{\diagCrossVRivOR}\biggr)
\ ,\quad
\Psi_{\mathsf O+}\in\mathrm{Hom}^{-2}\biggl(\KhOf*{\diagCrossVRivOL}, \KhOf*{\diagCrossHRivOR}\biggr)
\end{equation*}
by the following cobordisms (and zeros in the other degrees):
\begin{equation}
\label{eq:def-R4:PsiO}
(\Psi_{\mathsf O-})^1 \coloneqq -\BordPsiMinus
\ ,\quad
(\Psi_{\mathsf O+})^1\coloneqq
-\BordPsiPlus
\quad.
\end{equation}
The direct computation shows that these elements satisfy the following equations in the $\mathrm{Hom}$-complexes (e.g.~see \cite[Lemma~4.1]{ItoYoshida2021}):
\begin{equation}
\label{eq:PsiO-eqn}
\partial(\Psi_{\mathsf O-}) = -\Phi F_{\mathsf O-} -  G_{\mathsf O}\delta_-
\ ,\quad
\partial(\Psi_{\mathsf O+}) = -\delta_+ G_{\mathsf O} + F_{\mathsf O+}\Phi
\quad.
\end{equation}

Similarly, regarding the homotopy commutative diagram
\begin{equation}
\label{eq:R4cube-U}
\begin{tikzcd}[row sep=1ex,column sep=1em]
& \KhOf*{\diagCrossHRivUL} \ar[rr,"\delta_-"] \ar[dl] \ar[dd,"\vphantom{X}" description, "\gamma^{}_{\mathsf U}" pos=0.8] && \KhOf*{\diagCrossVRivUL} \ar[rr] \ar[dl,"\Phi"] \ar[dd,"\vphantom{X}" description, "-\omega^{}_{\mathsf U}" pos=0.8] && 0 \ar[dl] \ar[dd] \\
0 \ar[rr] \ar[dd] && \KhOf*{\diagCrossVRivUL} \ar[rr,"-\delta_+" pos=0.8] \ar[dd,"\omega_{\mathsf U}" pos=0.8] && \KhOf*{\diagCrossHRivUL} \ar[dd,"\gamma^{}_{\mathsf U}" pos=0.8] & \\
& \KhOf*{\diagCrossHRivUR} \ar[r,dash] \ar[dl] & {} \ar[r,"\delta_-" pos=0.3] & \KhOf*{\diagCrossVRivUR} \ar[r,dash] \ar[dl,"\Phi"] & {} \ar[r] & 0 \ar[dl] \\
0 \ar[rr] && \KhOf*{\diagCrossVRivUR} \ar[rr,"-\delta_+"] && \KhOf*{\diagCrossHRivUR} &
\end{tikzcd}
\quad,
\end{equation}
we define the elements
\begin{equation*}
\Psi_{\mathsf U-}\in\mathrm{Hom}^{-2}\biggl(\KhOf*{\diagCrossHRivUL},\KhOf*{\diagCrossVRivUR}\biggr)
\ ,\quad
\Psi_{\mathsf U+}\in\mathrm{Hom}^{-2}\biggl(\KhOf*{\diagCrossVRivUL}, \KhOf*{\diagCrossHRivUR}\biggr)
\end{equation*}
by the following cobordisms (and zeros in the other degrees):
\begin{equation}
\label{eq:def-R4:PsiU}
(\Psi_{\mathsf U-})^1 \coloneqq -\BordPsiMinusVar
\ ,\quad
(\Psi_{\mathsf U+})^1\coloneqq -\BordPsiPlusVar
\quad.
\end{equation}
As in the case of $\Psi_{\mathsf O\pm}$, they satisfy the following equations:
\begin{equation}
\label{eq:PsiU-eqn}
\partial(\Psi_{\mathsf U-}) = -\Phi F_{\mathsf U-} -  G_{\mathsf U}\delta_-
\ ,\quad
\partial(\Psi_{\mathsf U+}) = -\delta_+ G_{\mathsf U} + F_{\mathsf U+}\Phi
\quad.
\end{equation}

\begin{proof}[Proof of \cref{prop:R4}]
In view of \eqref{eq:R3O-cones}, we obtain isomorphisms
\begin{equation}
\label{eq:prf:R4:R3Ocone}
\KhOf*{\diagRiiiLeftMLR}^i \cong \KhOf*{\diagCrossVRivOL}^i \oplus \KhOf*{\diagCrossHRivOL}^{i+1}
\ ,\quad
\KhOf*{\diagRiiiRightMRL}^i \cong \KhOf*{\diagCrossHRivOR}^{i-1} \oplus \KhOf*{\diagCrossVRivOR}^i
\quad,
\end{equation}
with differentials represented respectively by the matrices
\begin{equation}
\label{eq:prf:R4:R3Ocone-diff}
D_1\coloneqq \begin{bmatrix} d & \delta_- \\ 0 & -d \end{bmatrix}
\ ,\quad
D_2\coloneqq \begin{bmatrix} -d & -\delta_+ \\ 0 & d \end{bmatrix}
\quad.
\end{equation}
Similarly, we also have isomorphisms
\begin{equation}
\label{eq:prf:R4:R3Ucone}
\KhOf*{\diagRiiiLeftLRM}^i \cong \KhOf*{\diagCrossVRivUL}^i \oplus \KhOf*{\diagCrossHRivUL}^{i+1}
\ ,\quad
\KhOf*{\diagRiiiRightRLM}^i \cong \KhOf*{\diagCrossHRivUR}^{i-1} \oplus \KhOf*{\diagCrossVRivUR}^i
\quad,
\end{equation}
with the differentials given in the same formulas as \eqref{eq:prf:R4:R3Ocone-diff}.
We then define the elements
\[
\psi_{\mathsf O}^{} \in \mathrm{Hom}^{-1}\biggl(\KhOf*{\diagRiiiLeftMLR}, \KhOf*{\diagRiiiRightMRL}\biggr)
\ ,\quad
\psi_{\mathsf U}^{} \in \mathrm{Hom}^{-1}\biggl(\KhOf*{\diagRiiiLeftLRM}, \KhOf*{\diagRiiiRightRLM}\biggr)
\]
which are represented by the following matrices with respect to the decompositions~\eqref{eq:prf:R4:R3Ocone} and \eqref{eq:prf:R4:R3Ucone}:
\begin{equation}
\label{eq:def-R4:psi}
\psi_{\mathsf O} \coloneqq
\begin{bmatrix} -\Psi_{\mathsf O+} & 0 \\ G_{\mathsf O} & -\Psi_{\mathsf O-} \end{bmatrix}
\ ,\quad
\psi_{\mathsf U} \coloneqq
\begin{bmatrix} -\Psi_{\mathsf U_+} & 0 \\ G_{\mathsf U} & -\Psi_{\mathsf U-} \end{bmatrix}
\quad.
\end{equation}
By virtue of the equations~\eqref{eq:R3O-matrix}, \eqref{eq:G-eqn}, and \eqref{eq:PsiO-eqn}, we have
\begin{equation}
\label{eq:prf:R4:psiO}
\begin{split}
\partial(\psi_{\mathsf O})
&= D_2\psi_{\mathsf O} + \psi_{\mathsf O}D_1
= \begin{bmatrix} -d & -\delta_+ \\ 0 & d\end{bmatrix}
\begin{bmatrix} -\Psi_{\mathsf O+} & 0 \\ G_{\mathsf O} & -\Psi_{\mathsf O-} \end{bmatrix}
+
\begin{bmatrix} -\Psi_{\mathsf O+} & 0 \\ G_{\mathsf O} & -\Psi_{\mathsf O-} \end{bmatrix}
\begin{bmatrix} d & \delta_- \\ 0 & -d\end{bmatrix}
\\
&= \begin{bmatrix}
F_{\mathsf O+}\Phi & \delta_+\Psi_{\mathsf O-} - \Psi_{\mathsf O+}\delta_- \\
-\omega_{\mathsf O}^{}\Phi - \Phi\omega_{\mathsf O}^{} & \Phi F_{\mathsf O-}
\end{bmatrix}
\\
&= 
\begin{bmatrix} 0 & 0 \\ \Phi & 0 \end{bmatrix}
\begin{bmatrix} -\omega_{\mathsf O}^{} & -F_{\mathsf O-} \\ 0 & \gamma_{\mathsf O}^{} \end{bmatrix}
- \begin{bmatrix} \gamma_{\mathsf O}^{} & -F_{\mathsf O+} \\ 0 & \omega_{\mathsf O}^{} \end{bmatrix}
\begin{bmatrix} 0 & 0 \\ \Phi & 0 \end{bmatrix}
+ \begin{bmatrix} 0 & \delta_+\Psi_{\mathsf O-} - \Psi_{\mathsf O+}\delta_- \\ 0 & 0 \end{bmatrix}
\\
&= \widehat\Phi_c\RMove3^{\mathsf O-} - \RMove3^{\mathsf O+}\widehat\Phi_c
+ \begin{bmatrix} 0 & \delta_+\Psi_{\mathsf O-} - \Psi_{\mathsf O+}\delta_- \\ 0 & 0 \end{bmatrix}
\quad.
\end{split}
\end{equation}
Observe that the last term in the right hand side of \eqref{eq:prf:R4:psiO} vanishes; in fact, we have
\[
\delta_+^{-1}\Psi_{\mathsf O-}^1
= -\;\BordPsiPlus\BordLeftDeltaNOO
= -\;\BordRightDeltaPll\BordPsiMinus
= \Psi_{\mathsf O+}^1\delta_-^1
\quad.
\]
Thus, $\psi_{\mathsf O}$ makes the left square in \eqref{eq:R4} commute.
The similar computation shows that $\psi_{\mathsf U}$ makes the right commute, so we obtain the result.
\end{proof}

\section{Proof of Main~Theorem A}
\label{sec:proof}

We now begin the proof of \cref{mainTHM}.
We directly compute the two compositions in the diagrams~\eqref{eq:hexagon-cross}.

For this, we denote by $\Phi_{\mathsf L}$ and $\Phi_{\mathsf R}$ the morphisms
\[
\Phi_{\mathsf L},\Phi_{\mathsf R}:
\diagTriviii \to \diagTriviii
\]
in the category $\Cob$ which applies the morphism $\Phi$ given in \eqref{eq:Phi-def} to the left and the right two strands respectively.
Note that the tangle in the domain (and the codomain) appears as a common state in the homological degree $0$ of the Khovanov complexes in \cref{mainTHM}.
Specifically, we have canonical inclusions
\[
\begin{gathered}
\diagTriviii = \diagRiiiLeftUParVVV
\hookrightarrow \KhOf*{\diagRiiiLeftLMR}^0, \KhOf*{\diagDeltaBrDLeft}^0
\quad,\\
\diagTriviii = \diagRiiiRightUParVVV
\hookrightarrow \KhOf*{\diagDeltaBrDRight}^0, \KhOf*{\diagRiiiRightRML}^0
\quad.
\end{gathered}
\]
In what follows, we refer them as the \emph{state $s_0$} in these complexes.

\begin{lemma}\label{lem:hexneg}
The compositions
\begin{gather}
\label{eq:hexneg-cntclock}
\KhOf*{\diagRiiiLeftLMRWith{a}{b}{c}}
\xrightarrow{\widehat\Phi_a} \KhOf*{\diagRiiiLeftMLRWith{a}{b}{c}}
\xrightarrow{\RMove3^{\mathsf O-}} \KhOf*{\diagRiiiRightMLRWith{a}{b}{c}}
\xrightarrow{\widehat\Phi_a} \KhOf*{\diagDeltaBrDRightWith{a}{b}{c}}
\quad,\\
\label{eq:hexneg-clock}
\KhOf*{\diagRiiiLeftLMRWith{a}{b}{c}}
\xrightarrow{\widehat\Phi_b} \KhOf*{\diagRiiiLeftLRMWith{a}{b}{c}}
\xrightarrow{\RMove3^{\mathsf U-}} \KhOf*{\diagRiiiRightLRMWith{a}{b}{c}}
\xrightarrow{\widehat\Phi_b} \KhOf*{\diagDeltaBrDRightWith{a}{b}{c}}
\quad,
\end{gather}
both vanish except on the state $s_0$ where they equal to the composition
\[
-\Phi_{\mathsf R}\Phi_{\mathsf L}:\;
\diagRiiiLeftUParVVV = \diagTriviii \to \diagTriviii = \diagRiiiRightUParVVV
\quad.
\]
\end{lemma}
\begin{proof}
By \eqref{eq:R3O-matrix}, the composition~\eqref{eq:hexneg-cntclock} equals
\[
\widehat\Phi_a\RMove3^{\mathsf O-}\widehat\Phi_a
=
\begin{bmatrix} -\widehat\Phi_a\omega_{\mathsf O}^{}\widehat\Phi_a & -\widehat\Phi_aF_{\mathsf O-}\widehat\Phi_a \\ 0 & \widehat\Phi_a\gamma_{\mathsf O}^{}\widehat\Phi_a \end{bmatrix}
\quad.
\]
Notice that the morphism $\widehat\Phi_a$ is non-zero only on the states where the crossing $a$ is smoothed vertically, and only such states appear in its image.
By the definitions~\eqref{eq:def-R3O:gamma} and \eqref{eq:def-R3O:F} of $\gamma_{\mathsf O}^{}$ and $F_{\mathsf O-}$, we hence have $\widehat\Phi_a\gamma_{\mathsf O}^{}\widehat\Phi_a = 0$ and $\widehat\Phi_a F_{\mathsf O} = 0$.
For the same reason, regarding \eqref{eq:def-R3O:omega}, one sees that the composition $\widehat\Phi\omega_{\mathsf O}^{}\widehat\Phi$ also vanishes except on the state $s_0$ where it is given by
\[
\widehat\Phi_a^0\omega_{\mathsf O}^0\widehat\Phi_a^0
= \widehat\Phi_a^0\circ\left(\BordTriId + \BordRTwoBarRUnit\right)\circ \widehat\Phi_a^0
= \Phi_{\mathsf R}\circ \BordTriId \circ \Phi_{\mathsf L}
= \Phi_{\mathsf R}\Phi_{\mathsf L}
\quad,
\]
which ensures the assertion for \eqref{eq:hexneg-cntclock}.

As for the other composition, by \eqref{eq:R3U-matrix}, we have
\[
\widehat\Phi_b\RMove3^{\mathsf U-}\widehat\Phi_b
=
\begin{bmatrix} -\widehat\Phi_b\omega_{\mathsf U}^{}\widehat\Phi_b & -\widehat\Phi_b F_{\mathsf U-}\widehat\Phi_b \\ 0 & \widehat\Phi_b\gamma_{\mathsf U}^{}\widehat\Phi_b \end{bmatrix}
\quad.
\]
By the same argument as above, we have $\widehat\Phi_b F_{\mathsf U-} = 0$, $\widehat\Phi_b\gamma_{\mathsf U}\widehat\Phi_b = 0$, and $\widehat\Phi_b\omega_{\mathsf U}\widehat\Phi_b = 0$ except the state $s_0$ where
\[
\widehat\Phi_b^0\omega_{\mathsf U}^0\widehat\Phi_b^0
= \widehat\Phi_b^0\circ \left(\BordTriId - \BordRTwoBarRUnit\right)\circ\widehat\Phi_b^0
= \Phi_{\mathsf R}\circ\BordTriId\circ\Phi_{\mathsf L}
= \Phi_{\mathsf R}\Phi_{\mathsf L}
\quad.
\]
This completes the proof.
\end{proof}

\begin{lemma}\label{lem:hexpos}
The compositions
\begin{gather}
\label{eq:hexpos-cntclock}
\KhOf*{\diagDeltaBrDLeftWith{a}{b}{c}}
\xrightarrow{\widehat\Phi_a} \KhOf*{\diagRiiiLeftMRLWith{a}{b}{c}}
\xrightarrow{\RMove3^{\mathsf O+}} \KhOf*{\diagRiiiRightMRLWith{a}{b}{c}}
\xrightarrow{\widehat\Phi_a} \KhOf*{\diagRiiiRightRMLWith{a}{b}{c}}
\quad,\\
\label{eq:hexpos-clock}
\KhOf*{\diagDeltaBrDLeftWith{a}{b}{c}}
\xrightarrow{\widehat\Phi_b} \KhOf*{\diagRiiiLeftRLMWith{a}{b}{c}}
\xrightarrow{\RMove3^{\mathsf U+}} \KhOf*{\diagRiiiRightRLMWith{a}{b}{c}}
\xrightarrow{\widehat\Phi_b} \KhOf*{\diagRiiiRightRMLWith{a}{b}{c}}
\quad,
\end{gather}
both vanishes except at the state $s_0$ where they equal to the composition
\[
\Phi_{\mathsf R}\Phi_{\mathsf L}:\;
\diagRiiiLeftUParVVV = \diagTriviii \to \diagTriviii = \diagRiiiRightUParVVV
\quad.
\]
\end{lemma}
\begin{proof}
Since the proof is almost identical to \cref{lem:hexneg}, we omit it.
\end{proof}

\begin{proof}[Proof of \cref{mainTHM}]
The part~\ref{sub:mainTHM:hexagon} directly follows from \cref{lem:hexneg,lem:hexpos}.
As for the part~\ref{sub:mainTHM:hexprism}, one can actually verify that the following two compositions are both trivial in a similar manner to \cref{lem:hexneg}:
\begin{gather}
\label{eq:prf:mainTHM:PhipsiOPhi}
\KhOf*{\diagRiiiLeftLMRWith{a}{b}{c}}
\xrightarrow{\widehat\Phi_a} \KhOf*{\diagRiiiLeftMLRWith{a}{b}{c}}
\xrightarrow{\psi_{\mathsf O}} \KhOf*{\diagRiiiRightMRLWith{a}{b}{c}}
\xrightarrow{\widehat\Phi_a} \KhOf*{\diagRiiiRightRMLWith{a}{b}{c}}
\quad,\\
\label{eq:prf:mainTHM:PhipsiUPhi}
\KhOf*{\diagRiiiLeftLMRWith{a}{b}{c}}
\xrightarrow{\widehat\Phi_b} \KhOf*{\diagRiiiLeftLRMWith{a}{b}{c}}
\xrightarrow{\psi_{\mathsf U}} \KhOf*{\diagRiiiRightRLMWith{a}{b}{c}}
\xrightarrow{\widehat\Phi_b} \KhOf*{\diagRiiiRightRMLWith{a}{b}{c}}
\quad.
\end{gather}
Indeed, by \eqref{eq:def-R4:psi}, we have
\begin{equation}\label{eq:prf:mainTHM:PhipsiOPhi-mat}
\widehat\Phi_a\psi_{\mathsf O}\widehat\Phi_a
= \begin{bmatrix}
-\widehat\Phi_a\Psi_{\mathsf O+}\widehat\Phi_a & 0 \\
\widehat\Phi_a G_{\mathsf O}\widehat\Phi_a & -\widehat\Phi_a\Psi_{\mathsf O-}\widehat\Phi_a 
\end{bmatrix}
\quad.
\end{equation}
Notice that, seeing \eqref{eq:def-R4:PsiO} and \eqref{eq:R4:GO}, $\Psi_{\mathsf O\pm}$ and $G_{\mathsf O}$ has no component both of whose domain and codomain are states such that the crossings $a$ are smoothed vertically.
This implies that each entry of the matrix~\eqref{eq:prf:mainTHM:PhipsiOPhi-mat} and hence the composition~\eqref{eq:prf:mainTHM:PhipsiOPhi} vanish.
Similarly, the composition~\eqref{eq:prf:mainTHM:PhipsiUPhi} vanishes.
Consequently, we obtain
\[
\widehat\Phi_a\psi_{\mathsf O}\widehat\Phi_a - \widehat\Phi_b\psi_{\mathsf U}\widehat\Phi_b = 0 - 0 = 0
\quad,
\]
which is exactly the required equation.
\end{proof}

\section{Proof of Main~Theorem~B}\label{sec:proof4T}

We finally prove the categorified $4T$-relation \cref{cat-4T}.
Recall that the Khovanov complex on a singular point is defined as the mapping cone of the crossing-change morphism $\widehat\Phi$ as in \eqref{eq:Phihat-def}.
Thus, the homotopy commutative squares~\eqref{eq:R4} induce chain homotopy equivalences
\begin{equation}\label{eq:R4-equiv}
\RMove4^{\mathsf O}: \KhOf*{\diagCrossSingRivOL} \xrightarrow\simeq \KhOf*{\diagCrossSingRivOR}
\ ,\quad
\RMove4^{\mathsf U}: \KhOf*{\diagCrossSingRivUL} \xrightarrow\simeq \KhOf*{\diagCrossSingRivUR}
\quad.
\end{equation}
We assert that the diagram below commutes strictly:
\begin{equation}\label{eq:hexagon-sing}
\begin{tikzcd}[column sep={3.5em,between origins}, row sep={5pc,between origins}]
& \KhOf*{\diagCrossSingRivLNNWith{a}{b}} \ar[rr,"\widehat\Phi_b"] \ar[dl,"\widehat\Phi_a"'] && \KhOf*{\diagCrossSingRivULWith{a}{b}} \ar[dr,"{\RMove4^{\mathsf U}}"] & \\
\KhOf*{\diagCrossSingRivOLWith{a}{b}} \ar[dr,"{\RMove4^{\mathsf O}}"'] &&&& \KhOf*{\diagCrossSingRivURWith{a}{b}} \ar[dl,"\widehat\Phi_b"] \\
& \KhOf*{\diagCrossSingRivORWith{a}{b}} \ar[rr,"\widehat\Phi_a"] && \KhOf*{\diagCrossSingRivRPPWith{a}{b}} &
\end{tikzcd}
\quad.
\end{equation}
Notice that the counter-clockwise and clockwise compositions are respectively induced by the following homotopy commutative squares:
\begin{equation}\label{eq:hexagon-cntclock-vs-clock}
\begin{tikzcd}[column sep=4em]
\KhOf*{\diagRiiiLeftLMRWith{a}{b}{c}} \ar[r,"{\widehat\Phi_a\RMove3^{\mathsf O-}\widehat\Phi_a}"] \ar[d,"{\widehat\Phi_c}"'] \ar[dr,dashed,"{\widehat\Phi_a\psi_{\mathsf O}}\widehat\Phi_a" description] & \KhOf*{\diagDeltaBrDRightWith{a}{b}{c}} \ar[d,"{\widehat\Phi_c}"] \\
\KhOf*{\diagDeltaBrDLeftWith{a}{b}{c}} \ar[r,"{\widehat\Phi_a\RMove3^{\mathsf O+}\widehat\Phi_a}"] & \KhOf*{\diagRiiiRightRMLWith{a}{b}{c}}
\end{tikzcd}
\ ,\quad
\begin{tikzcd}[column sep=4em]
\KhOf*{\diagRiiiLeftLMRWith{a}{b}{c}} \ar[r,"{\widehat\Phi_b\RMove3^{\mathsf U-}\widehat\Phi_b}"] \ar[d,"{\widehat\Phi_c}"'] \ar[dr,dashed,"{\widehat\Phi_b\psi_{\mathsf U}}\widehat\Phi_b" description] & \KhOf*{\diagDeltaBrDRightWith{a}{b}{c}} \ar[d,"{\widehat\Phi_c}"] \\
\KhOf*{\diagDeltaBrDLeftWith{a}{b}{c}} \ar[r,"{\widehat\Phi_b\RMove3^{\mathsf U+}\widehat\Phi_b}"] & \KhOf*{\diagRiiiRightRMLWith{a}{b}{c}}
\end{tikzcd}
\quad.
\end{equation}
\Cref{mainTHM} asserts that they are actually identical, so the diagram~\eqref{eq:hexagon-sing} commutes.

We are now ready to prove \cref{cat-4T}.

\begin{proof}[Proof of \cref{cat-4T}]
Since the diagram~\eqref{eq:hexagon-sing} is commutative, we have the following commutative diagram of complexes in $\Cob$:
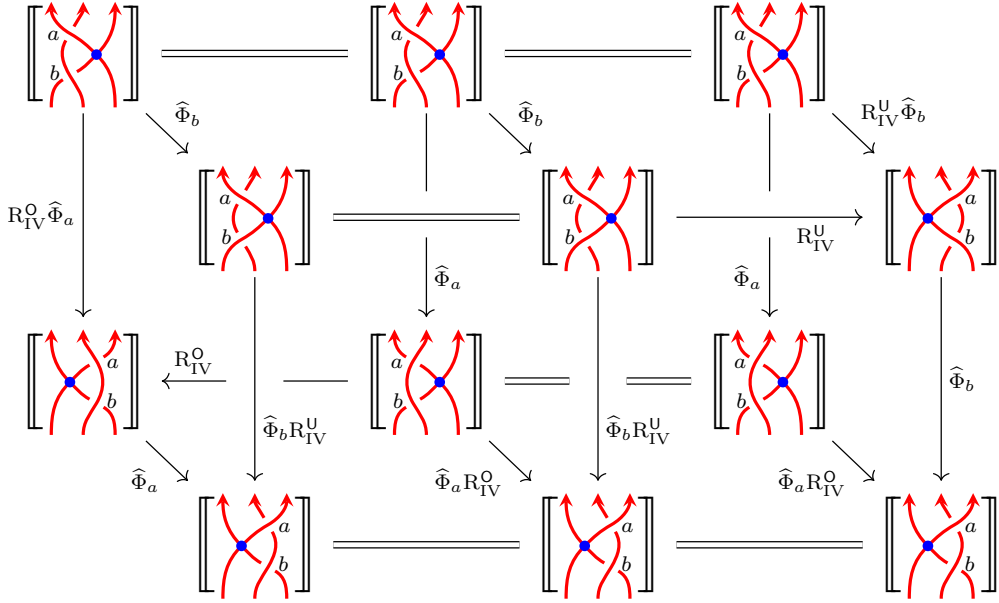
\begin{equation*}
\begin{tikzcd}[column sep=.5em, row sep=3ex]
\KhOf*{\diagCrossSingRivLNNWith{a}{b}} \ar[rr,equal] \ar[dr,"{\widehat\Phi_b}"] \ar[dd,"{\RMove4^{\mathsf O}\widehat\Phi_a}"'] && \KhOf*{\diagCrossSingRivLNNWith{a}{b}} \ar[rr,equal] \ar[dr,"{\widehat\Phi_b}"] \ar[d,dash] && \KhOf*{\diagCrossSingRivLNNWith{a}{b}} \ar[dr,"{\RMove4^{\mathsf U}\widehat\Phi_b}"] \ar[d,dash] \\
& \KhOf*{\diagCrossSingRivULWith{a}{b}} \ar[rr,equal] \ar[dd,"{\widehat\Phi_b\RMove4^{\mathsf U}}" pos=0.75] & \phantom{X} \ar[d,"{\widehat\Phi_a}"] & \KhOf*{\diagCrossSingRivULWith{a}{b}} \ar[rr,"{\RMove4^{\mathsf U}}"' pos=0.75] \ar[dd,"{\widehat\Phi_b\RMove4^{\mathsf U}}" pos=0.75] & \phantom{X} \ar[d,"{\widehat\Phi_a}"'] & \KhOf*{\diagCrossSingRivURWith{a}{b}} \ar[dd,"{\widehat\Phi_b}"] \\
\KhOf*{\diagCrossSingRivORWith{a}{b}} \ar[dr,"{\widehat\Phi_a}"'] & \phantom{X} \ar[l,"{\RMove4^{\mathsf O}}"'] \ar[r,dash] & \KhOf*{\diagCrossSingRivOLWith{a}{b}} \ar[r,equal] \ar[dr,"{\widehat\Phi_a\RMove4^{\mathsf O}}"'] & \phantom{X} \ar[r,equal] & \KhOf*{\diagCrossSingRivOLWith{a}{b}} \ar[dr,"{\widehat\Phi_a\RMove4^{\mathsf O}}"'] \\
& \KhOf*{\diagCrossSingRivRPPWith{a}{b}} \ar[rr,equal] && \KhOf*{\diagCrossSingRivRPPWith{a}{b}} \ar[rr,equal] && \KhOf*{\diagCrossSingRivRPPWith{a}{b}}
\end{tikzcd}
\quad.
\end{equation*}
This yields the following zigzag of chain homotopy equivalences between total complexes:
\begin{equation}\label{eq:prf:cat-4T:total}
\begin{split}
\operatorname{Tot}\left(
\begin{tikzcd}[ampersand replacement=\&]
\KhOf*{\diagCrossSingRivLNNWith{a}{b}} \ar[r,"{\widehat\Phi_b}"] \ar[d,"{\RMove4^{\mathsf O}\widehat\Phi_a}"'] \& \KhOf*{\diagCrossSingRivULWith{a}{b}} \ar[d,"{\widehat\Phi_b\RMove4^{\mathsf U}}"] \\
\KhOf*{\diagCrossSingRivORWith{a}{b}} \ar[r,"{\widehat\Phi_a}"] \& \KhOf*{\diagCrossSingRivRPPWith{a}{b}}
\end{tikzcd}
\right)
&\xleftarrow{\simeq}
\operatorname{Tot}\left(
\begin{tikzcd}[ampersand replacement=\&]
\KhOf*{\diagCrossSingRivLNNWith{a}{b}} \ar[r,"{\widehat\Phi_b}"] \ar[d,"{\widehat\Phi_a}"'] \& \KhOf*{\diagCrossSingRivULWith{a}{b}} \ar[d,"{\widehat\Phi_b\RMove4^{\mathsf U}}"] \\
\KhOf*{\diagCrossSingRivOLWith{a}{b}} \ar[r,"{\widehat\Phi_a\RMove4^{\mathsf O}}"] \& \KhOf*{\diagCrossSingRivRPPWith{a}{b}}
\end{tikzcd}
\right)
\\
&\xrightarrow{\simeq}
\operatorname{Tot}\left(
\begin{tikzcd}[ampersand replacement=\&]
\KhOf*{\diagCrossSingRivLNNWith{a}{b}} \ar[r,"{\RMove4^{\mathsf U}\widehat\Phi_b}"] \ar[d,"\widehat\Phi_a"'] \& \KhOf*{\diagCrossSingRivURWith{a}{b}} \ar[d,"\widehat\Phi_b"] \\
\KhOf*{\diagCrossSingRivOLWith{a}{b}} \ar[r,"{\widehat\Phi_a\RMove4^{\mathsf O}}"] \& \KhOf{\diagCrossSingRivRPPWith{a}{b}}
\end{tikzcd}
\right)
\quad.
\end{split}
\end{equation}
Note that, since the total complexes are isomorphic to the two-fold mapping cones, the first and the last complexes in \eqref{eq:prf:cat-4T:total} are respectively isomorphic to the mapping cones of the following forms:
\[
\operatorname{Cone}\left(\KhOf*{\diagFourTLND}\to\KhOf*{\diagFourTRDP}\right)
\ ,\quad
\operatorname{Cone}\left(\KhOf*{\diagFourTLDN}\to\KhOf*{\diagFourTRPD}\right)
\quad.
\]
By \eqref{eq:prf:cat-4T:total}, they are chain homotopic to each other, and this completes the proof.
\end{proof}

\section*{Acknowledgments}
The authors would like to thank Professor Keiichi Sakai for his comments.  
The work was partially supported by JSPS KAKENHI Grant Numbers JP20K03604, JP22K03603 and Toyohashi Tech Project of Collaboration with KOSEN.  

\bibliographystyle{plain}
\bibliography{reference}

\end{document}